%% file: contour-manifolds-v6.tex
\documentclass[a4paper,11pt]{article}
\pdfoutput=1
\usepackage[utf8]{inputenc}
\usepackage{geometry}
\usepackage{graphicx}
\usepackage{amsmath}
\usepackage{amsfonts}
\usepackage{amssymb}
\usepackage{amsthm}
\usepackage{bbold}
\usepackage{subfig}
\usepackage{enumerate}
\usepackage{theoremref}

\usepackage{tikz}
\usetikzlibrary{arrows,calc,decorations.pathmorphing,backgrounds,positioning,fit,petri}

\input{LatexDefinitions.tex}

\input{LatexDefinitions_Thm.tex}

\DeclareMathOperator{\spt}{spt}
\newcommand{\interior}{\textnormal{int}\,}
\newcommand{\dens}{\,\textnormal{dens}}
\newcommand{\indicator}{\textnormal{ind}}
\newcommand{\region}{\Omega}
\newcommand{\testFunction}{\mc{D}}
\newcommand{\project}{\textnormal{pr}}
\newcommand{\const}{\textnormal{const}}

\newcommand{\Meas}{\mc{P}}
\newcommand{\MeasGen}{\textnormal{Prob}}
\newcommand{\WD}{W}
\newcommand{\coupling}{\Pi}
\newcommand{\SMeas}{\mc{SP}}

\newcommand{\Tan}{\textnormal{Tan}}
\newcommand{\STan}{\textnormal{STan}}

\newcommand{\liftPoint}{F}
\newcommand{\liftPointSobolev}{\tilde{F}}
\newcommand{\liftPointEquiv}{F_\sfold}
\newcommand{\deliftPoint}{\gamma_n}
\newcommand{\deliftPointSobolev}{\tilde{\gamma}_n}
\newcommand{\liftTangent}{f}
\newcommand{\liftTangentEquiv}{f_\sfold}

\newcommand{\emb}{\textnormal{Emb}}
\newcommand{\sfold}{\textnormal{B}}
\newcommand{\diff}{\textnormal{Diff}}

\newcounter{ProofPointCounter}

\numberwithin{equation}{section}

\title{Contour Manifolds and Optimal Transport}
\author{Bernhard Schmitzer \qquad Christoph Schn\"orr\\Image \& Pattern Analysis Group, Heidelberg University}

\date{\today}

\begin{document}
\maketitle

\begin{abstract}
	\noindent Describing shapes by suitable measures in object segmentation, as proposed in \cite{SchmitzerSchnoerr-EMMCVPR2013}, allows to combine the advantages of the representations as parametrized contours and indicator functions.
	The pseudo-Riemannian structure of optimal transport can be used to model shapes in ways similar as with contours, while the Kantorovich functional enables the application of convex optimization methods for global optimality of the segmentation functional.
	
	In this paper we provide a mathematical study of the shape measure representation and its relation to the contour description. In particular we show that the pseudo-Riemannian structure of optimal transport, when restricted to the set of shape measures, yields a manifold which is diffeomorphic to the manifold of closed contours. A discussion of the metric induced by optimal transport and the corresponding geodesic equation is given.
\end{abstract}

\tableofcontents

\input{contour-manifolds-v6-Section-1-Intro}
\input{contour-manifolds-v6-Section-2-Background}
\input{contour-manifolds-v6-Section-3-MathContrib}
\input{contour-manifolds-v6-Section-4-SpaceProperties}
\input{contour-manifolds-v6-Section-5-Conclusion}

\bibliography{contour-manifolds}{}
\bibliographystyle{plain}

\end{document}

%% file: LatexDefinitions.tex
\newcommand{\bitem}{\begin{itemize}}
\newcommand{\eitem}{\end{itemize}}
\newcommand{\mc}[1]{\mathcal{#1}}

\newcommand{\N}{\mathbb{N}}
\newcommand{\R}{\mathbb{R}}

\newcommand{\bpm}{\begin{pmatrix}}
\newcommand{\epm}{\end{pmatrix}}
\newcommand{\bsm}{\left(\begin{smallmatrix}}
\newcommand{\esm}{\end{smallmatrix}\right)}
\newcommand{\T}{\top}
\newcommand{\ol}[1]{\overline{#1}}

\newcommand{\veps}{\varepsilon}

\DeclareMathOperator{\tr}{tr}
\DeclareMathOperator{\id}{id}

\DeclareMathOperator{\ddiv}{div}

\DeclareMathOperator{\Proj}{Proj}

%% file: LatexDefinitions_Thm.tex
\theoremstyle{plain}
\newtheorem{theorem}{Theorem}[section]
\newtheorem{lemma}[theorem]{Lemma}
\newtheorem{proposition}[theorem]{Proposition}

\theoremstyle{definition}
\newtheorem{definition}[theorem]{Definition}

\theoremstyle{remark}
\newtheorem{remark}[theorem]{Remark}

%% file: contour-manifolds-v6-Section-1-Intro.tex
\section{Introduction}
	\subsection{Motivation}
		\label{sec:Motivation}
		Shape is a very general concept in image processing and computer vision.
		It manifests itself in a wide variety of representations, for example:
		point clouds \cite{MemoliGHPointCloud2005}, level sets, a wide range of signatures (e.g.~the spectrum of the Laplacian \cite{ReuterLaplaceSpectrum2006} or distributions of path lengths \cite{GeodesicShapeMassTransport-10}), measures on metric spaces \cite{memoli-gromov-shape-11}, parametrized contours \cite{PR-Kernel-Shape-03,Charpiat2007} or indicator functions \cite{CremersKlodtICCV11Moments,SchmitzerSchnoerr-SSVM2011}.
		The choice of the representation depends strongly on the application in mind.
		Typical tasks in connection with shape analysis are
		\begin{itemize}
			\item measuring similarity between shapes, i.e.~finding a metric on shapes, applicable for classification and recognition,
			\item computing meaningful registrations between similar shapes,
			\item statistical analysis and modelling of distributions of shapes,
			\item optimizing w.r.t.~exterior criteria, be it technical specifications in product design or local appearance features for object localization and pose estimation in image data,
			\item the abstraction from geometric transformations, such as Euclidean isometries in the ambient space or non-isometric changes in the pose of articulated objects, that a shape can undergo while still perceptually remaining the same shape.
		\end{itemize}
		Each of the named representations has its strengths and weaknesses among that list. In this paper we study a representation of shape from the viewpoint of variational image segmentation, that removes the competition between these points.
		
		The problem of segmenting a given image into a set of predefined classes (e.g. fore- and background) is a prototypical problem that calls for the application of shape priors.
		In a variational approach there are usually two parts in the functional: for a given segmentation candidate, one part is concerned with estimating how well the local appearance features of the image match the assigned class. The other part, the shape prior, rates the plausibility of the segmentation regions according to a shape model. The prior must distinguish familiar from unfamiliar shapes and estimate the likelihood of a given input (usually based on a set of training shapes).
		Often location, orientation or even pose of the sought-after object are unknown, so it is critical that the shape model can deal with these degrees of freedom.
		Hence, this problem indeed involves all of the tasks listed above and thus the question which representation to choose is a difficult one. In fact, many different approaches have been tried (cf.~Section \ref{sec:RelatedLiterature}).
		
		Two common representations for this problem are \emph{indicator functions} and \emph{parametrized contours}. Mathematically they are equivalent, as one representation can be converted into the other. But practically they are somewhat complimentary.
		While indicator functions tend to be well suited for optimization w.r.t.~local image features and local boundary regularity, it is rather hard to formulate non-local functionals such as shape priors.
		Conversely, parametrized contours have been used for sophisticated statistical shape modelling. However, optimization w.r.t.~local image features can usually only be performed through small deformations, yielding non-convex models that are prone to get stuck in suboptimal local minima.
		
		Recently, the pseudo-Riemannian structure on Wasserstein spaces was proposed to be used for analysis of measures with characteristic spatial distributions \cite{OptimalTransportTangent2012}. In \cite{SchmitzerSchnoerr-EMMCVPR2013} this approach was extended to be suitable for modelling shapes on indicator functions in a way that is in style reminiscent of the way shape priors have been modelled on parametrized contours, but at the same time being better suited for evaluation of appearance functionals, hence leading to an overall functional that could be optimized to global optimality.
		
		In this paper we shed more light on what `reminiscent in style' means and establish a precise mathematical relation between the shape representation used in \cite{SchmitzerSchnoerr-EMMCVPR2013} and contour manifolds. As a result, we show that the complimentary aspects of both shape representations can be removed.
		
	\subsection{Contribution and Organization of the Paper}
		\begin{figure}
			\centering%
			\includegraphics[width=7cm]{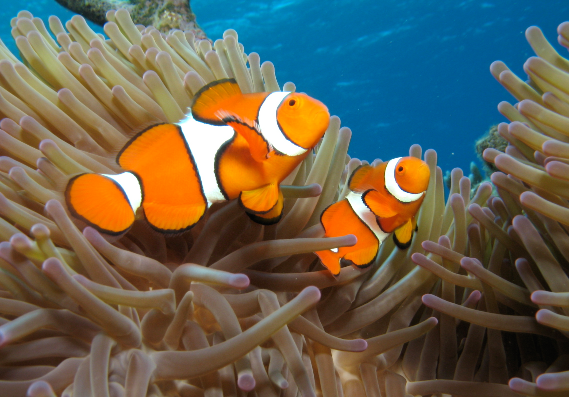}%
			\hfill
			\includegraphics[width=7cm]{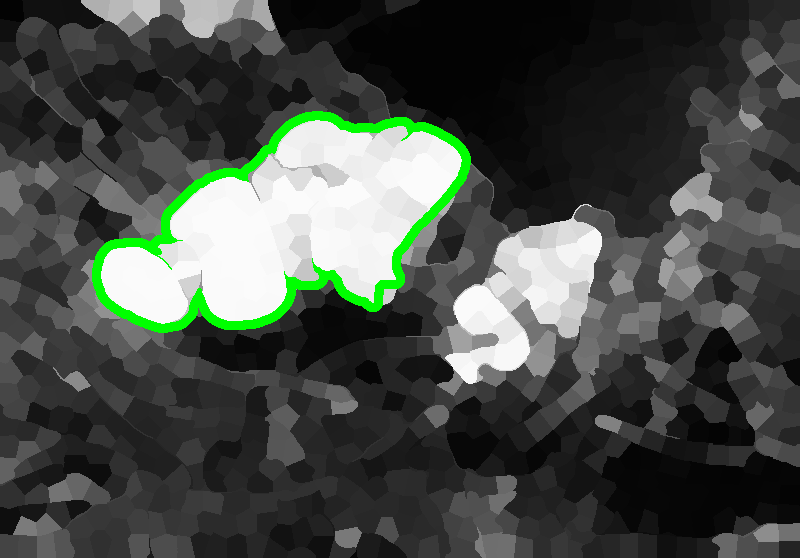}%
			\caption{Globally optimal object segmentation with shape measures, based on approach presented in \cite{SchmitzerSchnoerr-EMMCVPR2013}. Right: original image, left: gray values indicate local foreground affinity (white: foreground) based on a rudimentary color model. Green outline indicates segmentation with shape prior based on shape measures. It was specifically searched for the larger of the two fish.
			An approach with a shape prior based on parametrized contours might get stuck on the smaller of the two fish, if initialized poorly.
			By contrast, the approach presented in \cite{SchmitzerSchnoerr-EMMCVPR2013}, based on shape measures, yields a \emph{globally optimal segmentation}, independent of initialization.}
			\label{fig:WassersteinModes}
		\end{figure}
		
		In Section \ref{sec:RelatedLiterature} we will go through related literature and some notational conventions are established in Section \ref{sec:Notation}.
		Then, throughout Section \ref{sec:Background} we will gather the necessary mathematical background that this article builds on. We will touch upon contour manifolds, flows and their induced diffeomorphisms, optimal transport, in particular the Riemannian structure of Wasserstein spaces and some elementary facts on Poisson's equation.
		
		In Sections \ref{sec:MathContrib} and \ref{sec:RiemannianMetric} we will then present the main part of our contributions.
		\begin{enumerate}[(i)]
			\item We start in \ref{sec:MathContrib} by viewing the set of shape measures whose densities are scaled indicator functions informally as a submanifold of the set of all measures, as introduced in Sect.~\ref{sec:BackgroundWasserstein}, and give the corresponding subspace of flow-fields that are tangent to this submanifold. \label{item:ContribShapeMeasures}
			\item Then we show that there is a lifting from the manifold of contours modulo parametrization onto the shape measures. Correspondingly there is a lifting from the tangent space at a given contour to the tangent space at the corresponding shape measure. These liftings are bijective and consistent in the sense that the operations `lifting' and `taking the tangent' commute (see Fig. \ref{fig:LiftingCommutation}). Conveniently the parametrization ambiguity of contours disappears in the measure representation and one need no longer handle equivalence classes of objects.
			\item It is then investigated how paths on the contour manifold translate into paths on the set of shape measures and vice versa.
			\item Finally we establish that, when equipped with the appropriate topology, the set of shape measures is indeed a manifold, diffeomorphic to the manifold of contours modulo parametrization. \label{item:ContribManifold}
			\item In Section \ref{sec:RiemannianMetric} we equip the manifold of shape measures with the Riemannian inner product implied by optimal transport. This yields a new type of metric on the manifold of contours, beyond the Sobolev-type metrics discussed in the literature. We study the resulting metric structure of the tangent space.
			\item Eventually, as an instructive excursion beyond the mathematically rigorous core of the paper, we discuss a candidate for the geodesic equation on the manifold of shape measures and investigate some particular solutions as well as numerical approximations.
		\end{enumerate}
		The quintessence of items (\ref{item:ContribShapeMeasures}) to (\ref{item:ContribManifold}) is to demonstrate, that the measure representation is in fact equivalent to the contour representation in terms of shape modelling and even comes without the parametrization ambiguity. At the same time, as we  will elaborate further in Sect.~\ref{sec:RelatedLiterature}, it can more easily be combined with functionals rating the agreement with local appearance information, thus making it an adequate choice for problems such as shape prior assisted image segmentation (cf.~Fig.~\ref{fig:WassersteinModes}).
		We conclude in Sect.~\ref{sec:Conclusion}.

		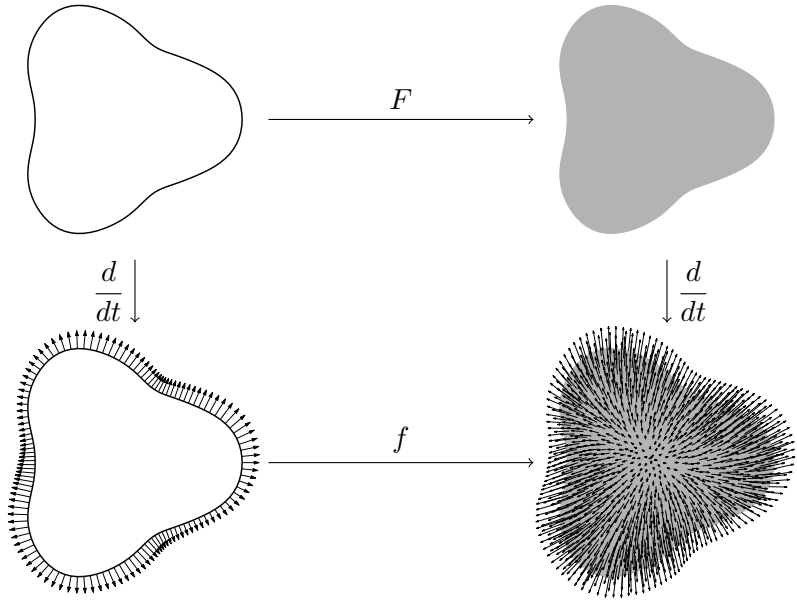
\begin{figure}[hbt]
			\centering
			\newlength{\figwidthB}%
			\setlength{\figwidthB}{3.5cm}%
			\begin{tikzpicture}[img/.style={anchor=center, inner sep=0}]
				\node[img] (n1) at (0,0) {\includegraphics[width=\figwidthB]{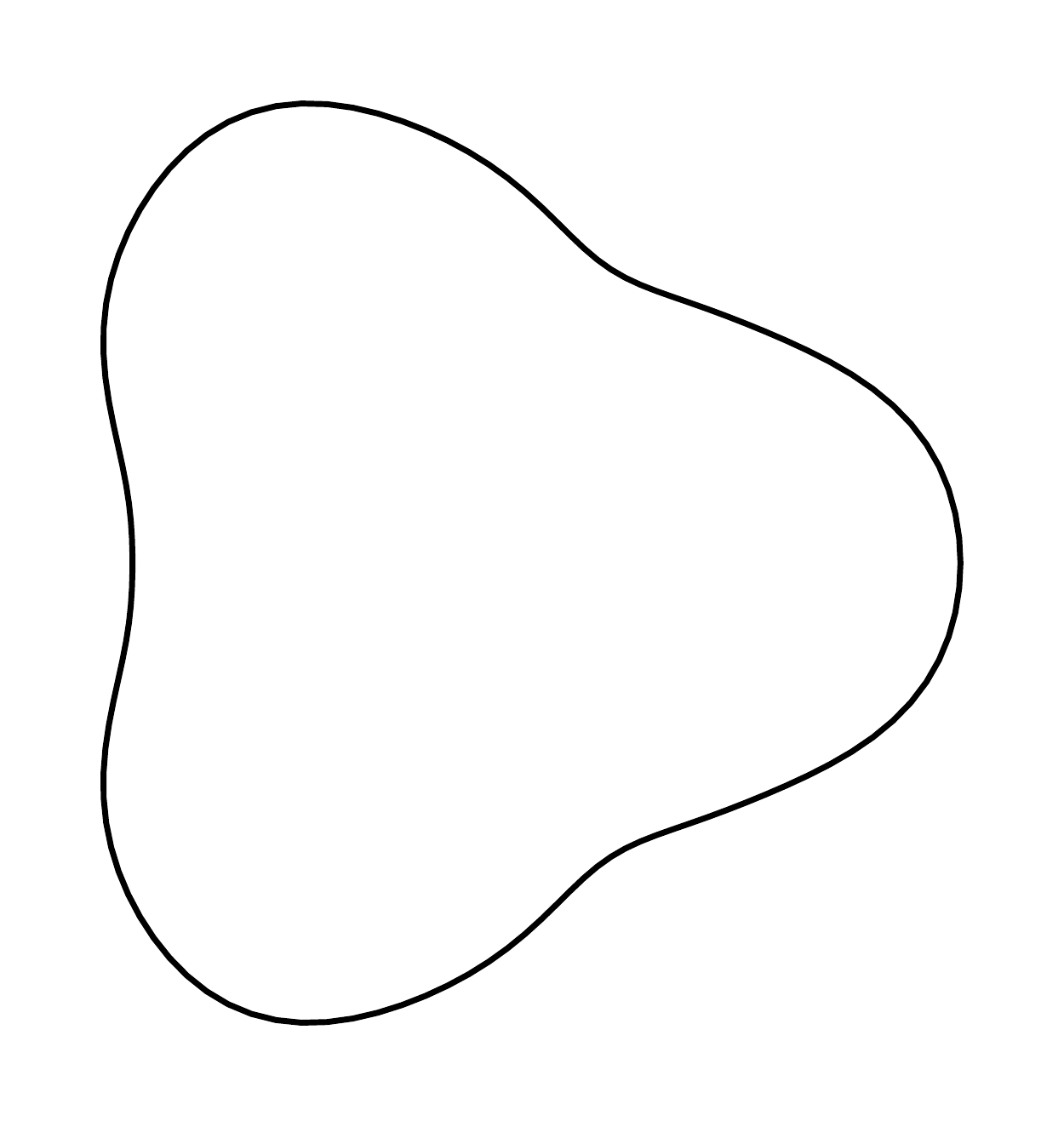}};
				\node[img] (n2) at (2*\figwidthB,0) {\includegraphics[width=\figwidthB]{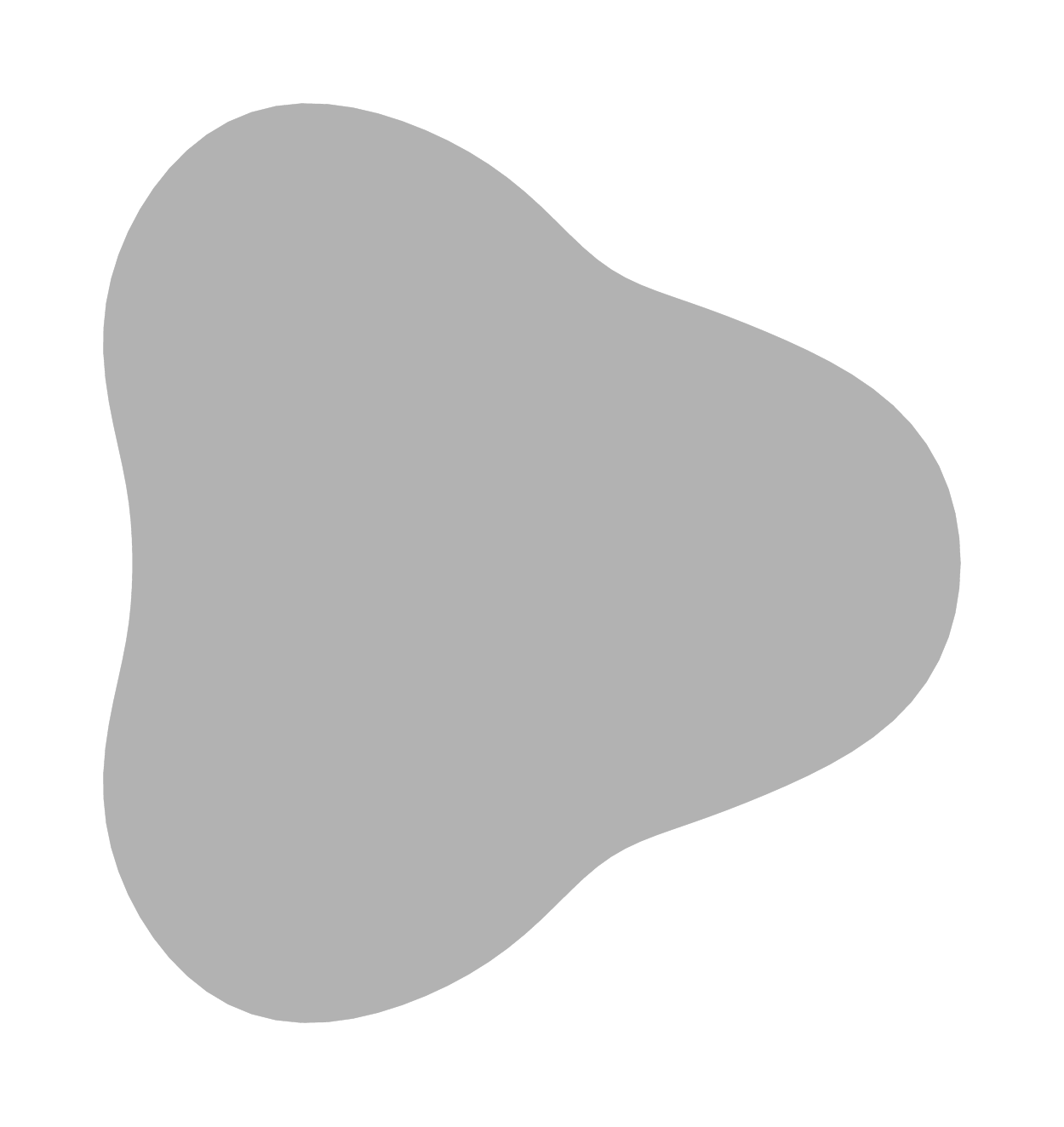}};
				\node[img] (n3) at (0,-1.3*\figwidthB) {\includegraphics[width=\figwidthB]{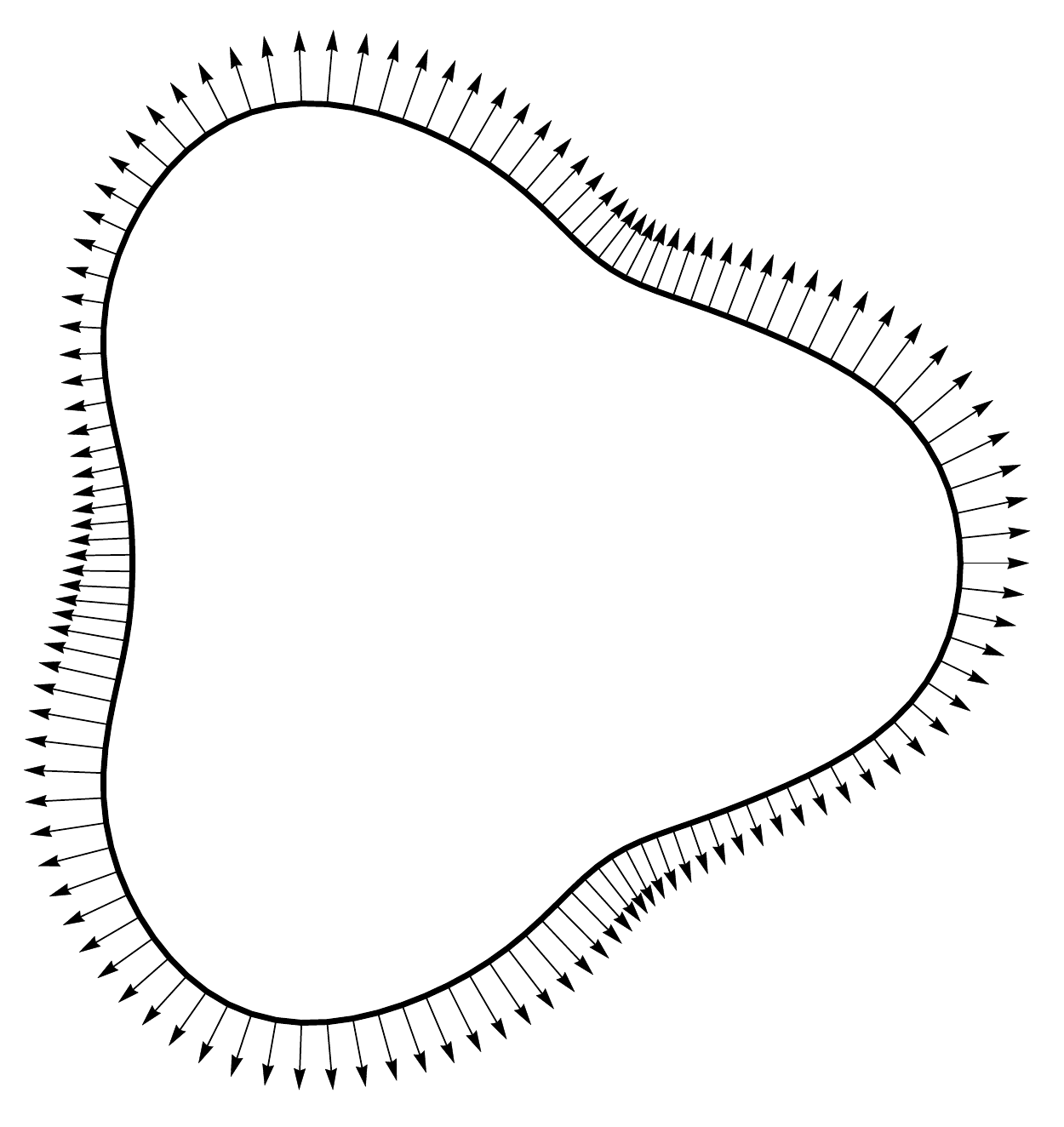}};
				\node[img] (n4) at (2*\figwidthB,-1.3*\figwidthB) {\includegraphics[width=\figwidthB]{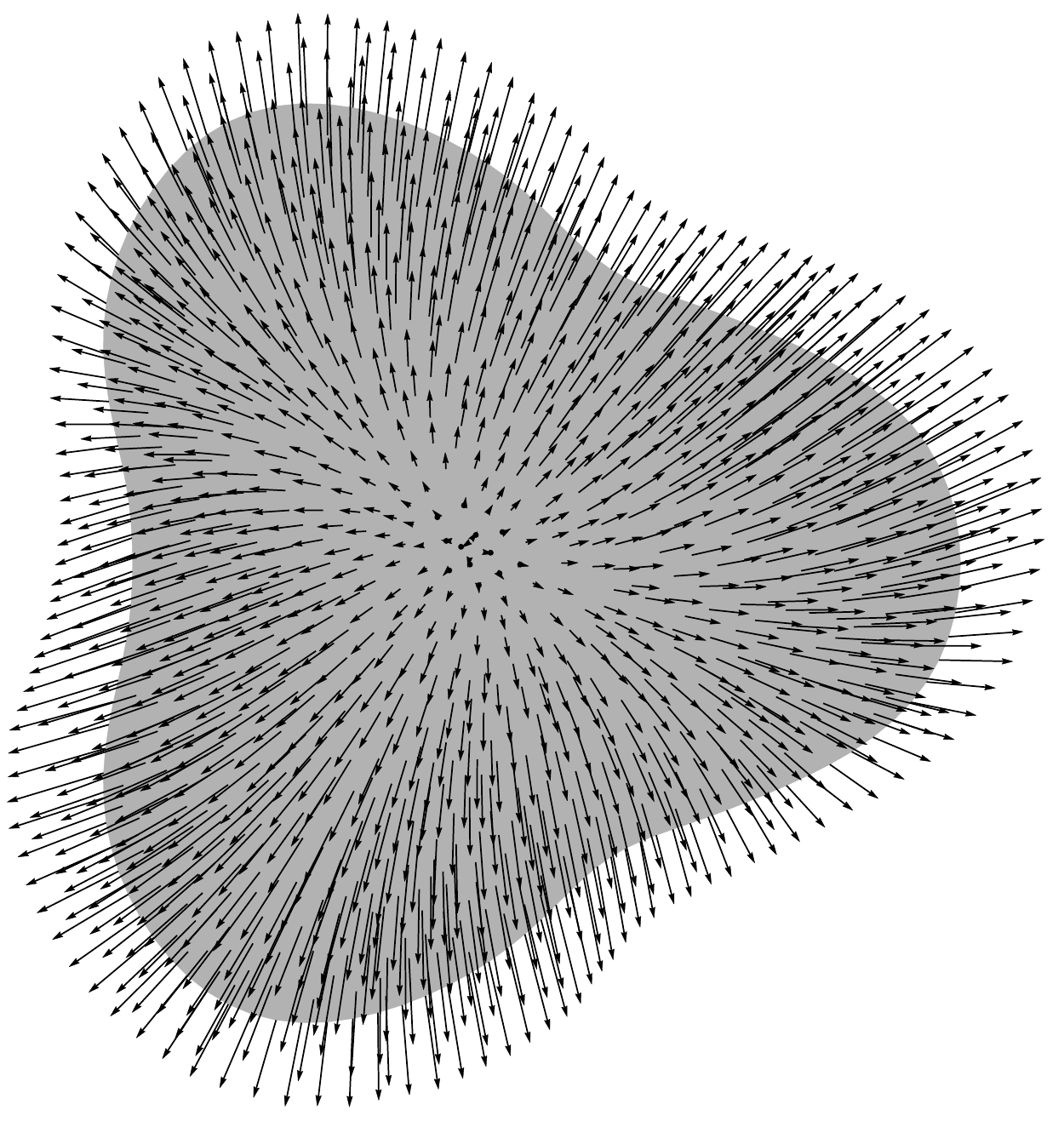}};
				\draw[->] (n1) -- node[above]{$\liftPoint$} (n2);
				\draw[->] (n1) -- node[left]{$\displaystyle \frac{d}{dt}$} (n3);
				\draw[->] (n2) -- node[right]{$\displaystyle \frac{d}{dt}$}(n4);
				\draw[->] (n3) -- node[above]{$\liftTangent$} (n4);
			\end{tikzpicture}
			\caption{Shapes can be represented by parametrized contours (top left) or probability measures with constant density in the interior and zero elsewhere (top right). Let $\liftPoint$ be the map that takes contours to measures. Taking the time derivative of a path of contours will yield a normal deformation field (bottom left), on a path of measures it will yield a flow-field according to the continuity equation in optimal transport (bottom right). In this article we will discuss a map $f$ that takes contour deformation fields to measure deformation fields, such that the diagram above commutes (\protect{\thref{thm:LiftingCommutation}}).}
			\label{fig:LiftingCommutation}
		\end{figure}

	\subsection{Related Literature}
		\label{sec:RelatedLiterature}
		\paragraph{Indicator Functions.}
			Indicator functions are a common shape representation in convex variational image segmentation. Data terms, based on local image features, as well as boundary regularizers, such as total variation and its extensions can be conveniently formulated and the corresponding functionals can be solved exactly or in good approximation through convex relaxation onto a suitable space of real functions \cite{ContinuousGlobalBinary06,Pock-et-al-CVPR09,Lellmann-Schnoerr-SIIMS-11}.
			Attempts to formulate shape priors in this representation are however often rather simplistic or lack important features, such as geometric invariance \cite{CremersKlodtICCV11Moments,SchmitzerSchnoerr-SSVM2011}.
			This may be owed to the fact that the linear structure provided by the vector space of real functions is not quite suitable for describing shapes: for example the linear interpolation between two shapes is nowhere a shape itself (see Fig.~\ref{fig:LinearStructures}). Also, $L^p$-type metrics are no good measures of shape similarity: they only measure the area of difference, regardless where the differing regions are.
			
		\paragraph{Contour Manifolds.}
			The set of parametrized contours can be treated as an infinite dimensional manifold \cite{MichorGlobalAnalysis}. This manifold, equipped with various metrics has been studied theoretically in great depth in the context of shape analysis \cite{Michor2006,PlaneCurveGeodesicExplicit-08,sunmensoa10}.
			Local approximation by its tangent space yields a natural linear structure for applying machine learning techniques for elaborate statistical modelling \cite{PR-Kernel-Shape-03,Charpiat2007}.
			But this representation, too, has its disadvantages: the parametrization ambiguity, while eliminated in theory by resorting to a suitable quotient manifold, remains a practical problem in implementations (see Fig.~\ref{fig:LinearStructures}). In particular, it is more complicated as with indicator functions to evaluate and optimize the contour w.r.t.~local image features. Usually, internally the contour has to be converted into the corresponding region representation to evaluate the functional and optimization is only performed w.r.t.~small local updates, resulting in non-convex models that are prone to get stuck in suboptimal local minima (e.g.~\cite{PR-Kernel-Shape-03}).
			
			The difference in the natural linear structures on parametrized contours and indicator functions is illustrated in Figure \ref{fig:LinearStructures}.
		\begin{figure}[hbt]
			\centering%
			\newlength{\figwidthA}%
			\setlength{\figwidthA}{2.7cm}%
			\begin{tabular}{ccccc}%
				\includegraphics[width=\figwidthA]{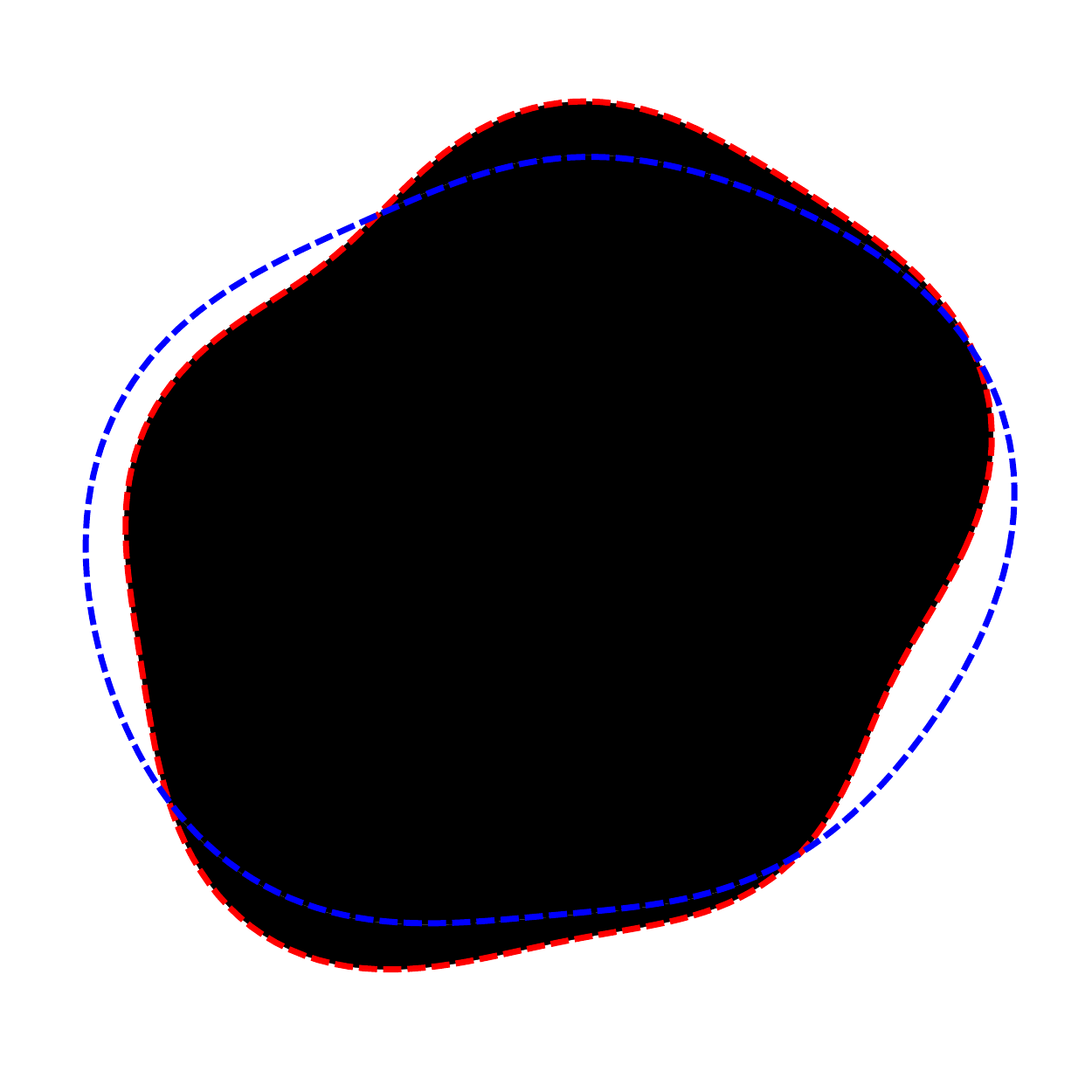} &%
				\includegraphics[width=\figwidthA]{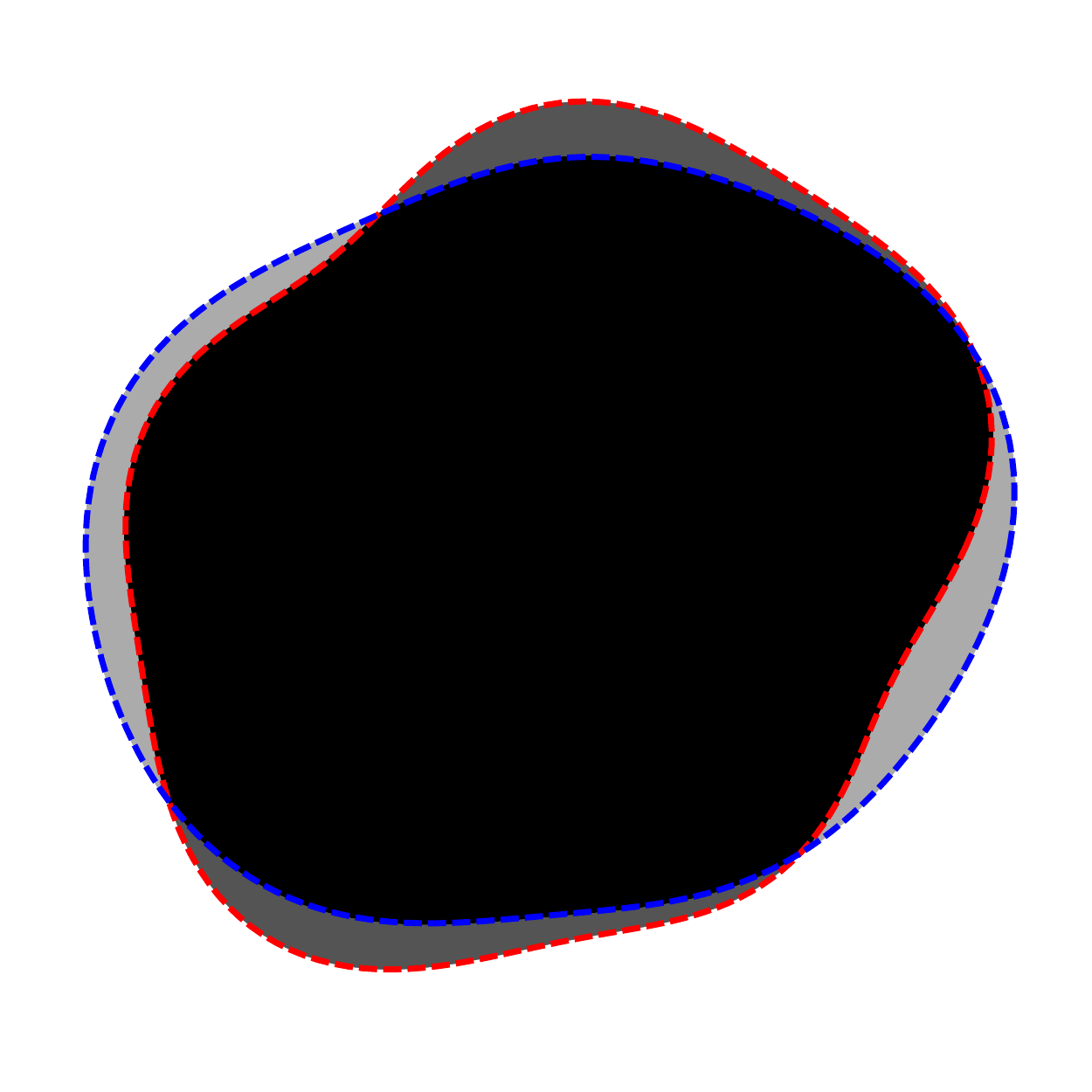} &%
				\includegraphics[width=\figwidthA]{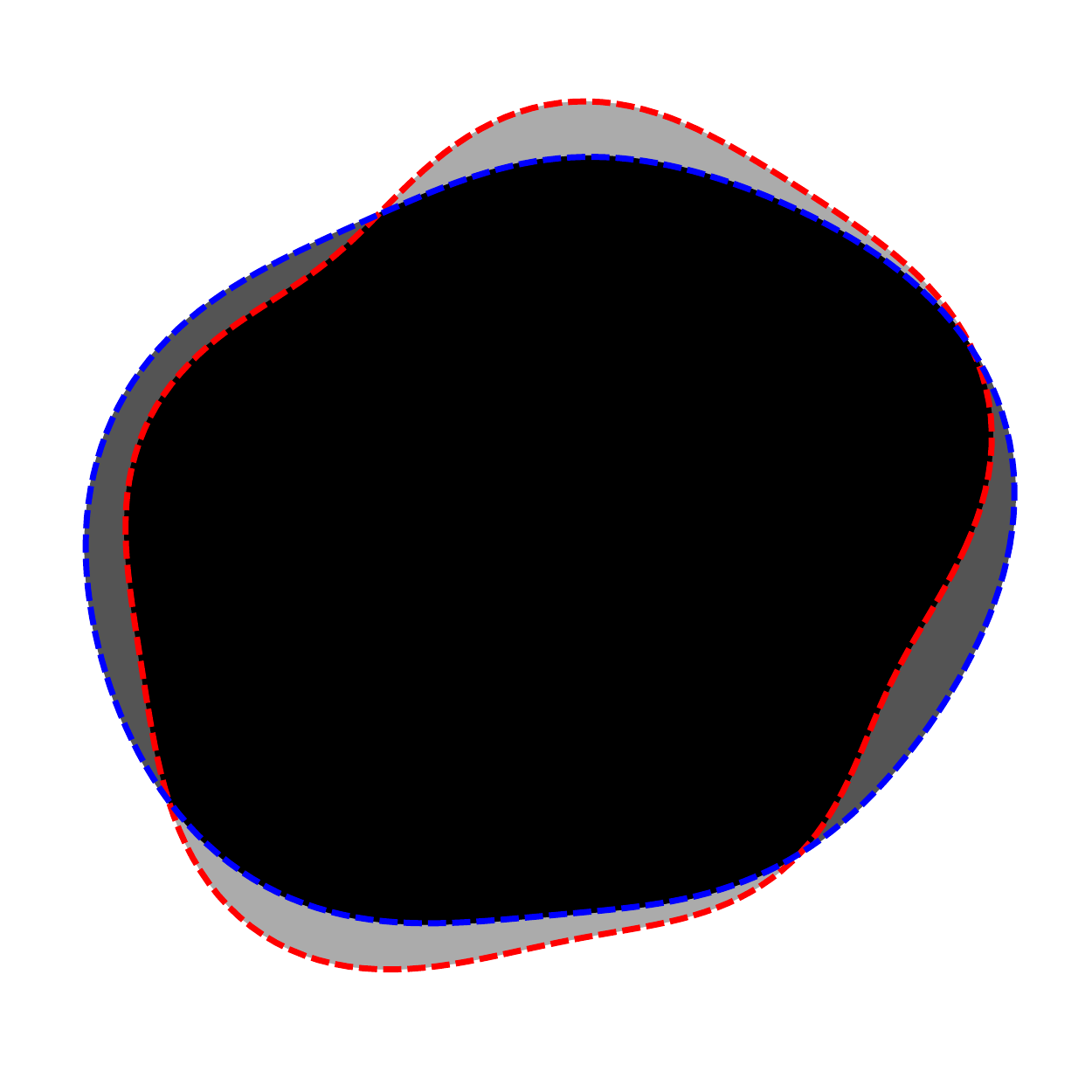} &%
				\includegraphics[width=\figwidthA]{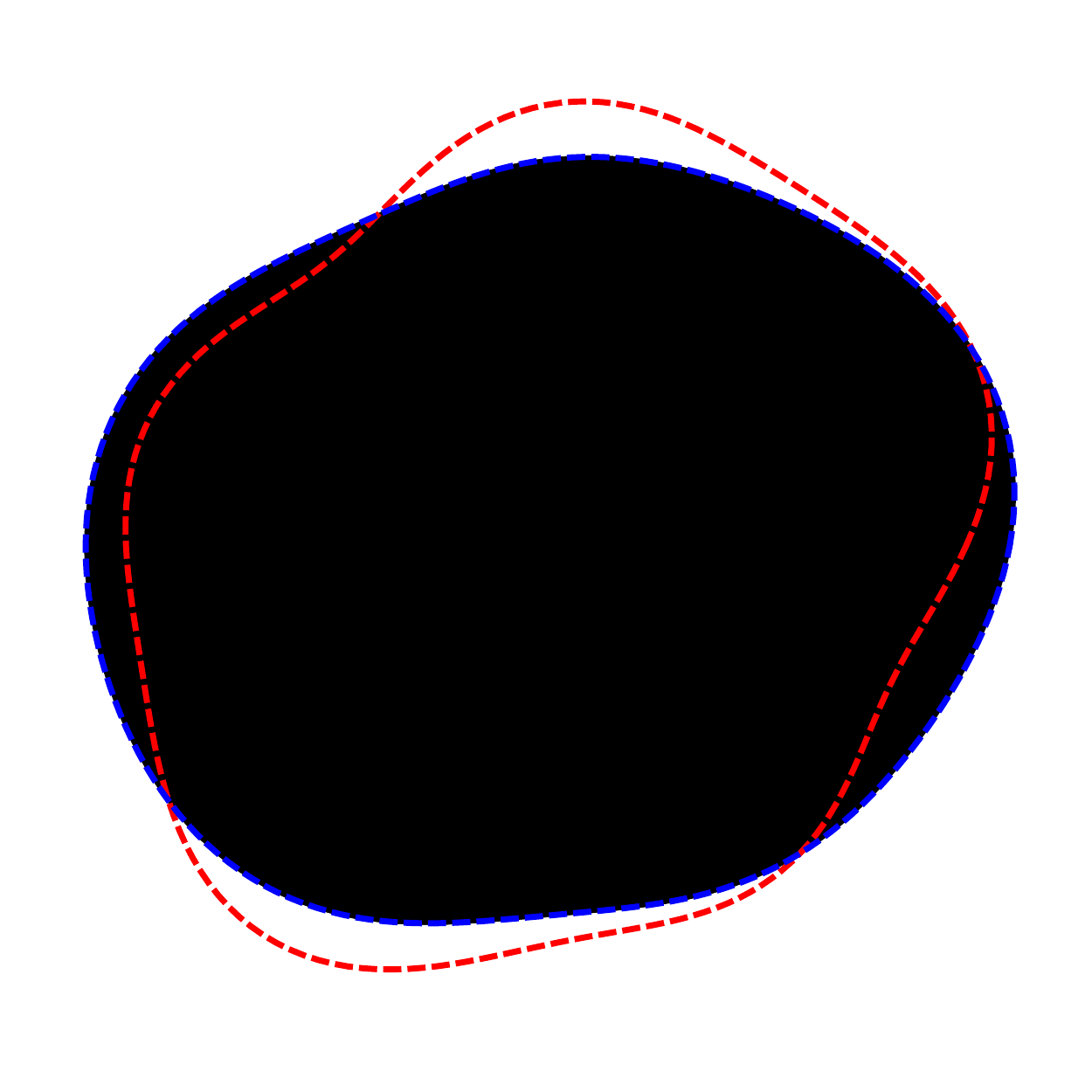} &%
				\includegraphics[width=\figwidthA]{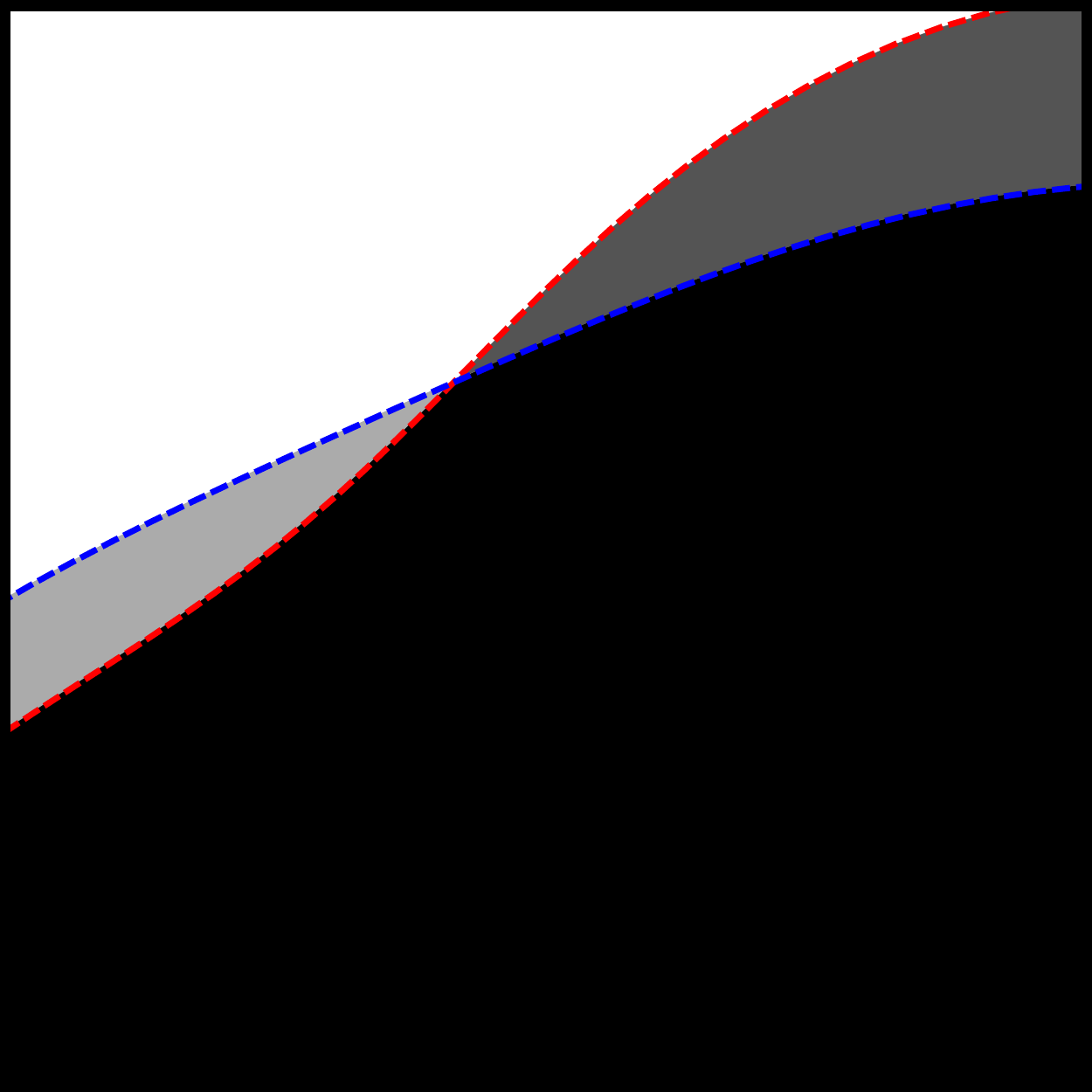} \\%
				\includegraphics[width=\figwidthA]{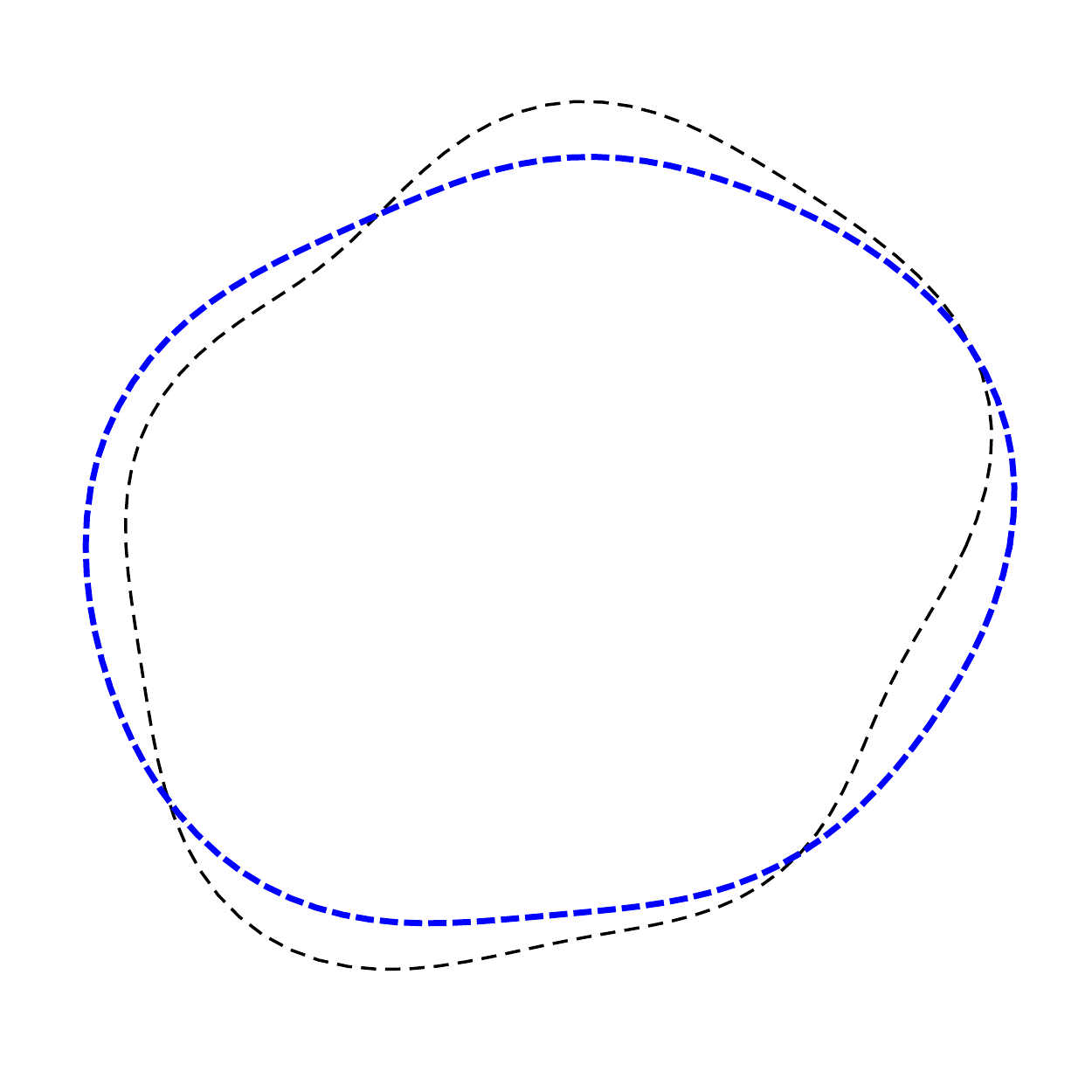} &%
				\includegraphics[width=\figwidthA]{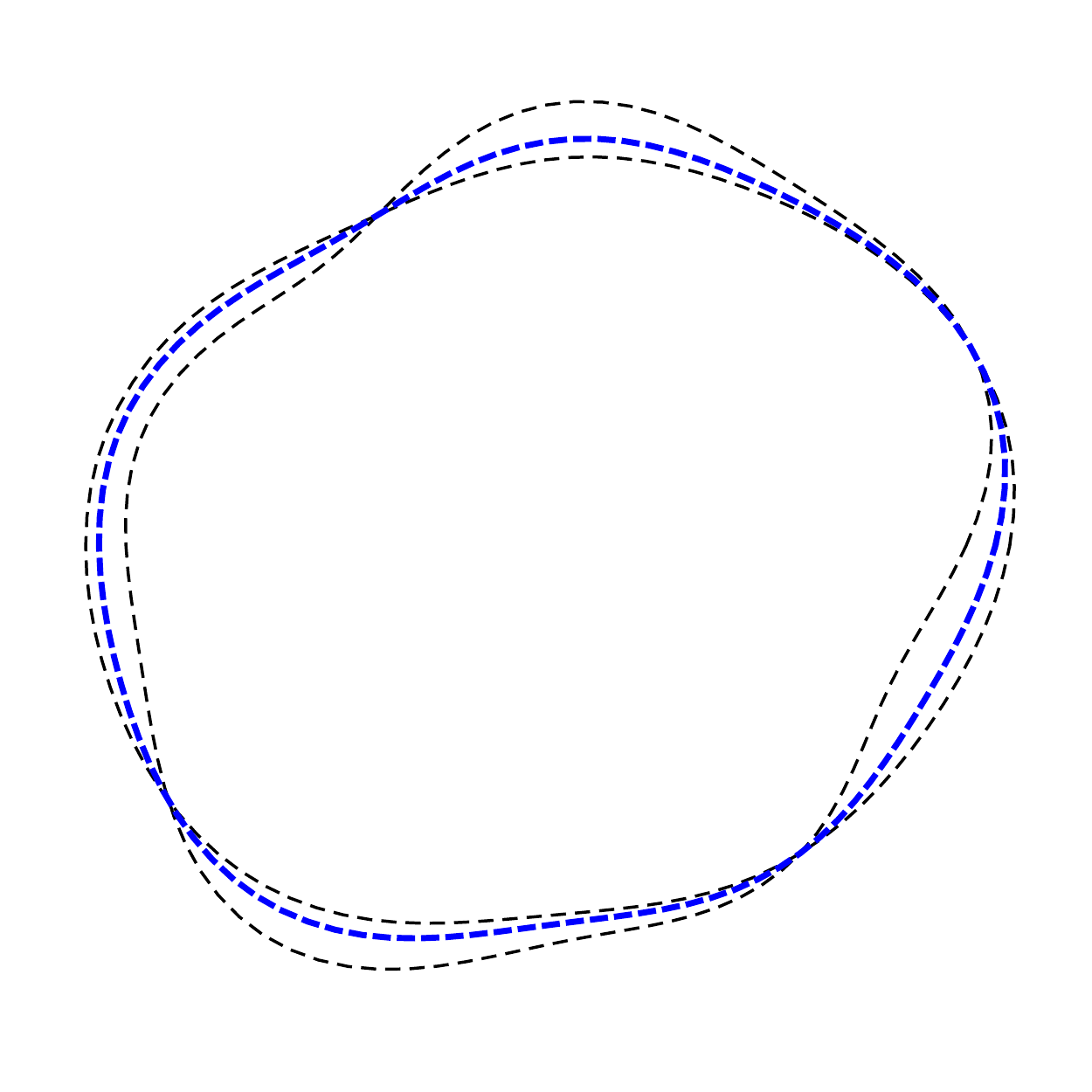} &%
				\includegraphics[width=\figwidthA]{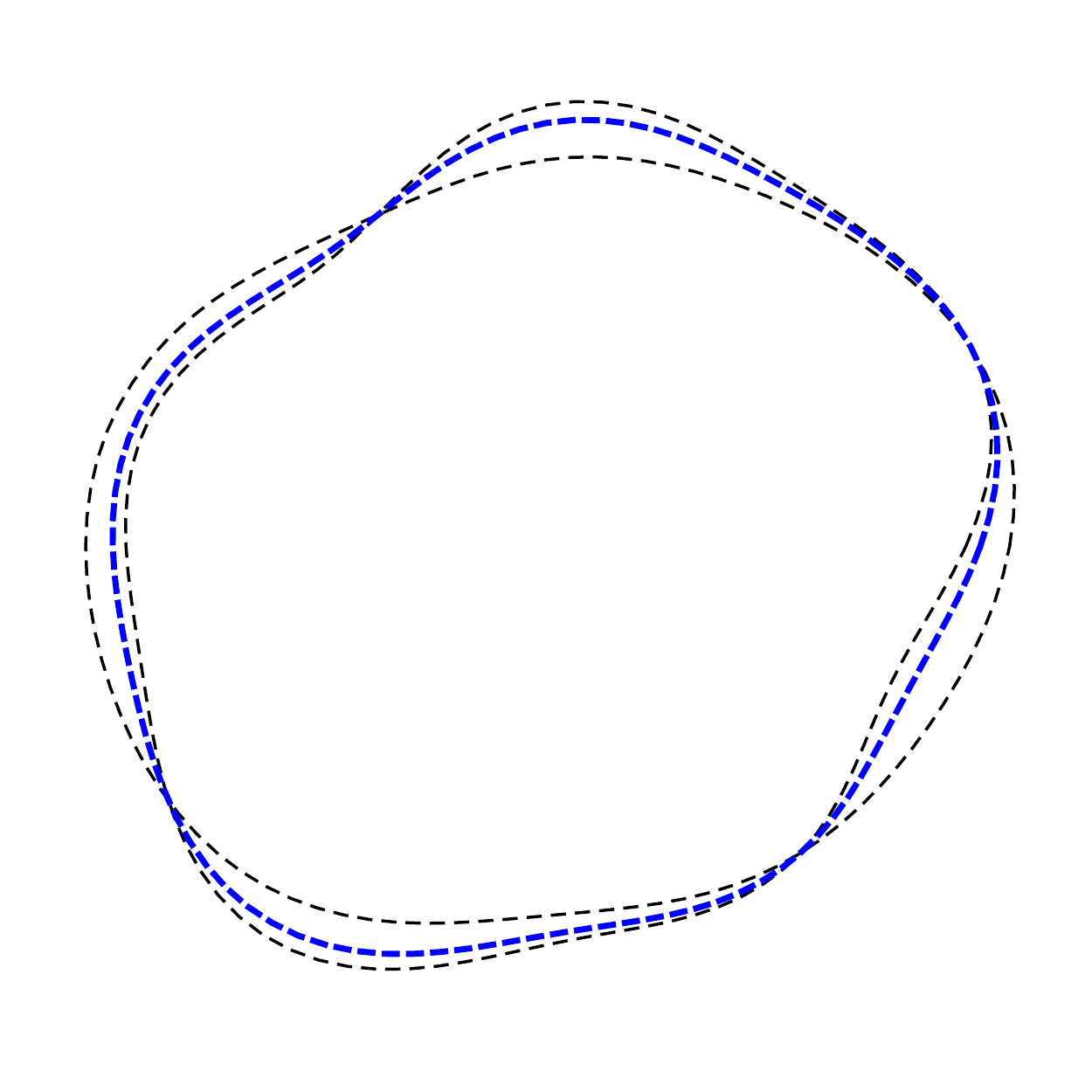} &%
				\includegraphics[width=\figwidthA]{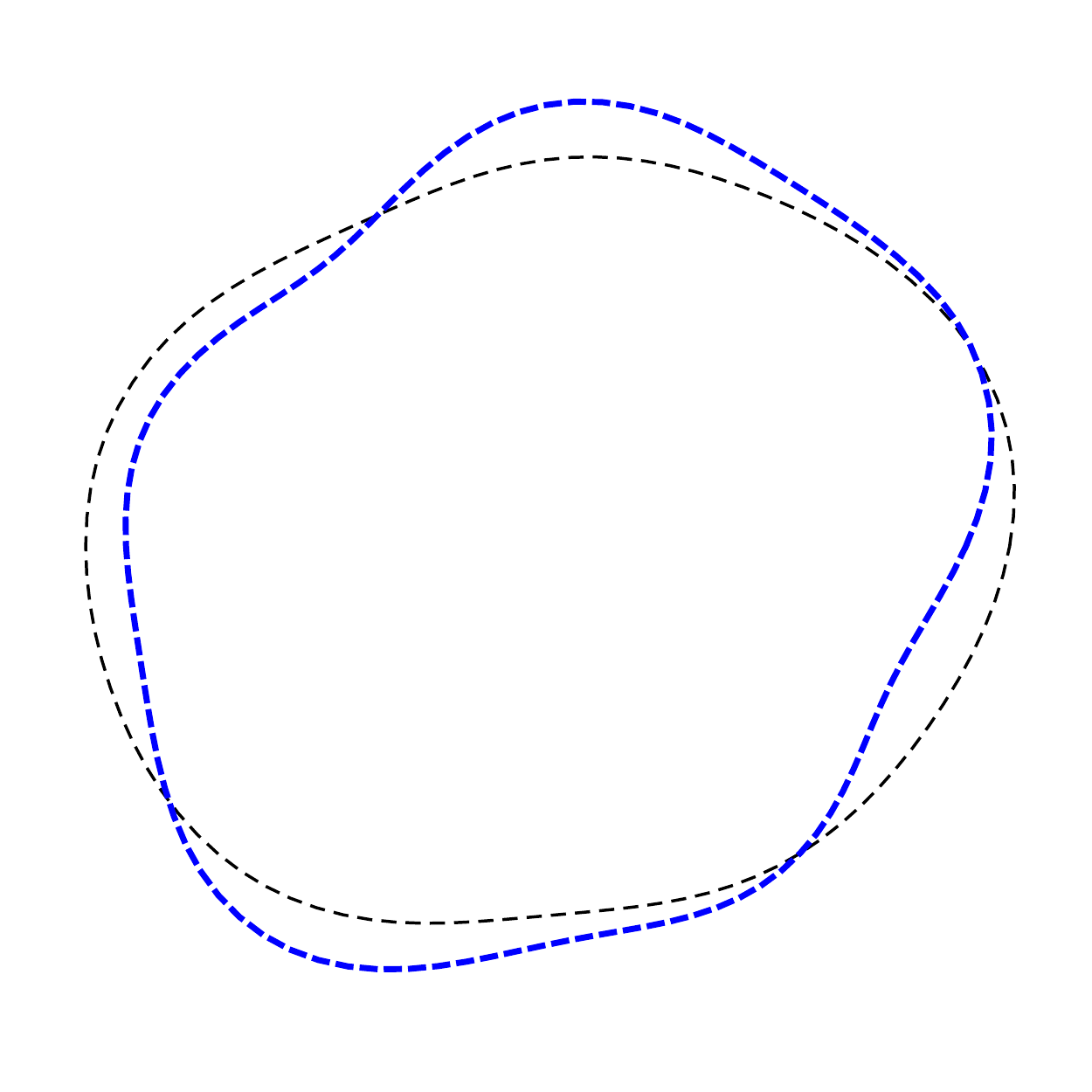} &%
				\includegraphics[width=\figwidthA]{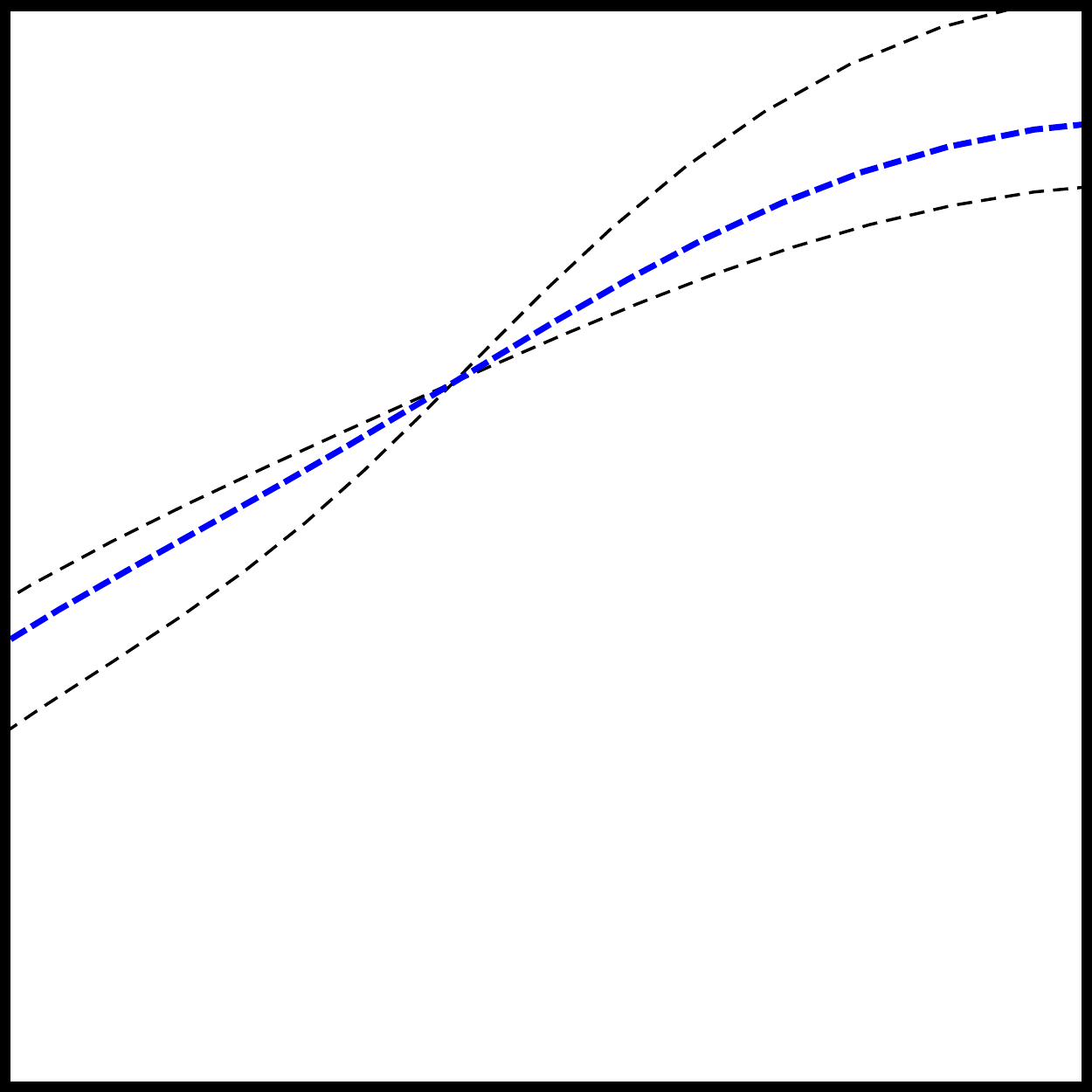} \\%
				\includegraphics[width=\figwidthA]{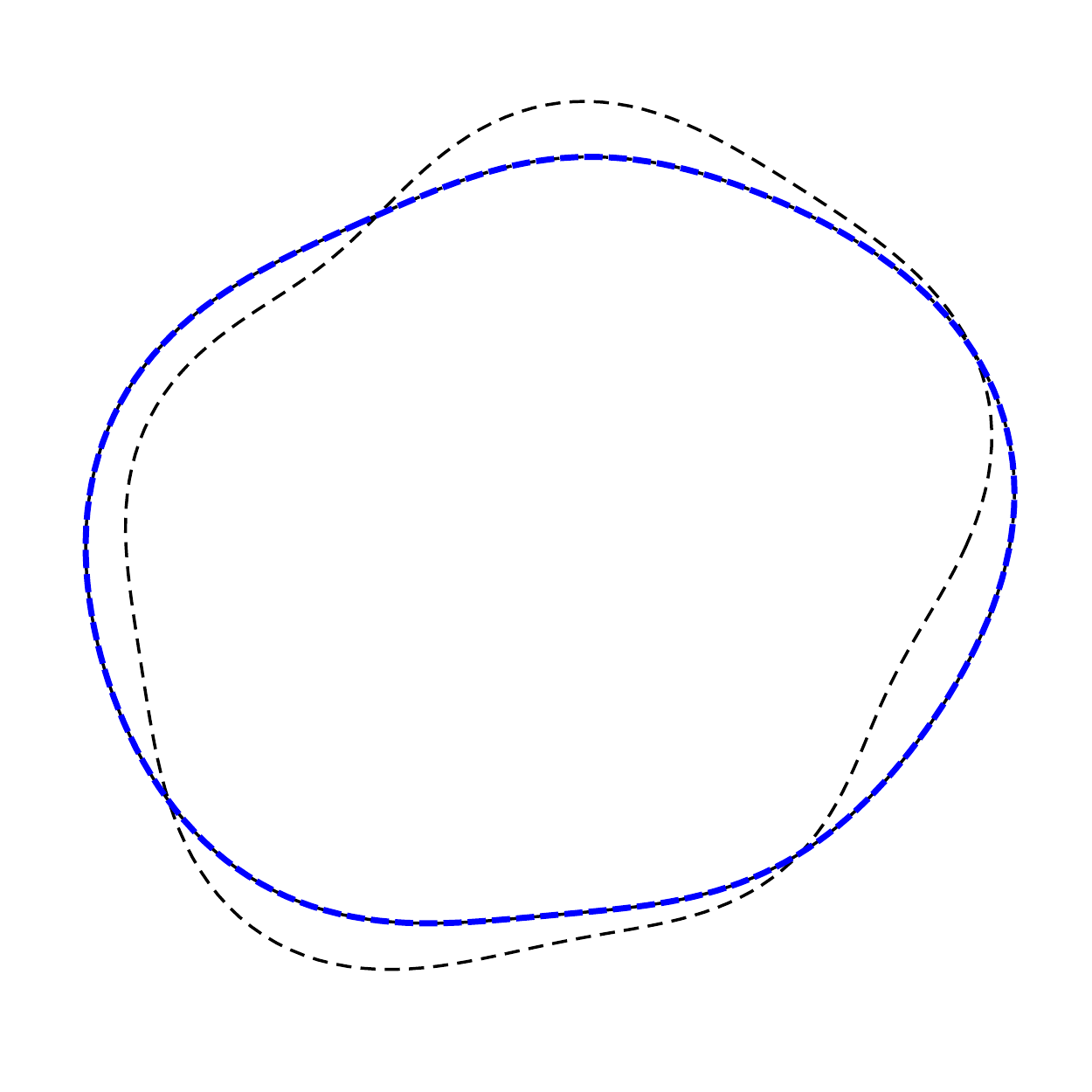} &%
				\includegraphics[width=\figwidthA]{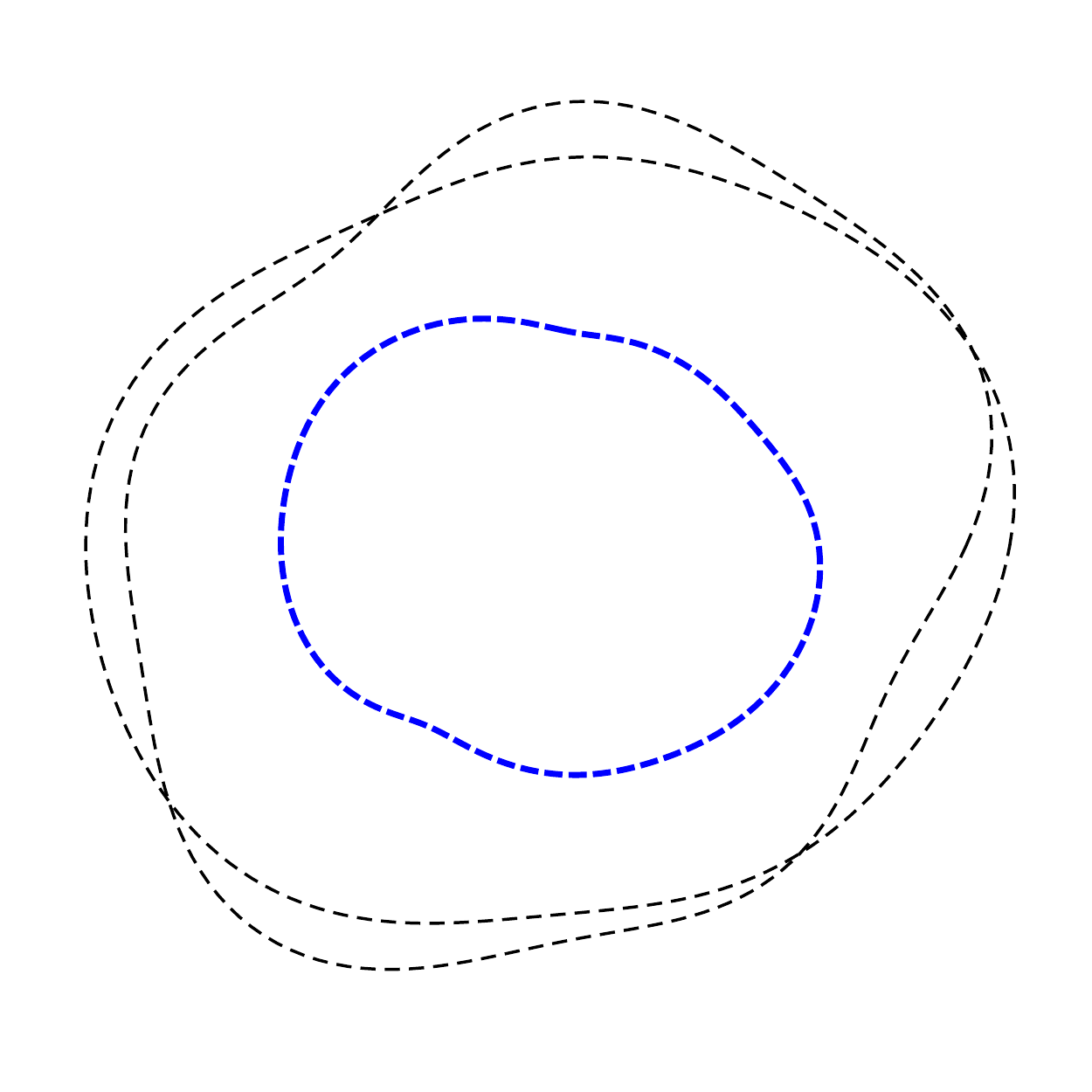} &%
				\includegraphics[width=\figwidthA]{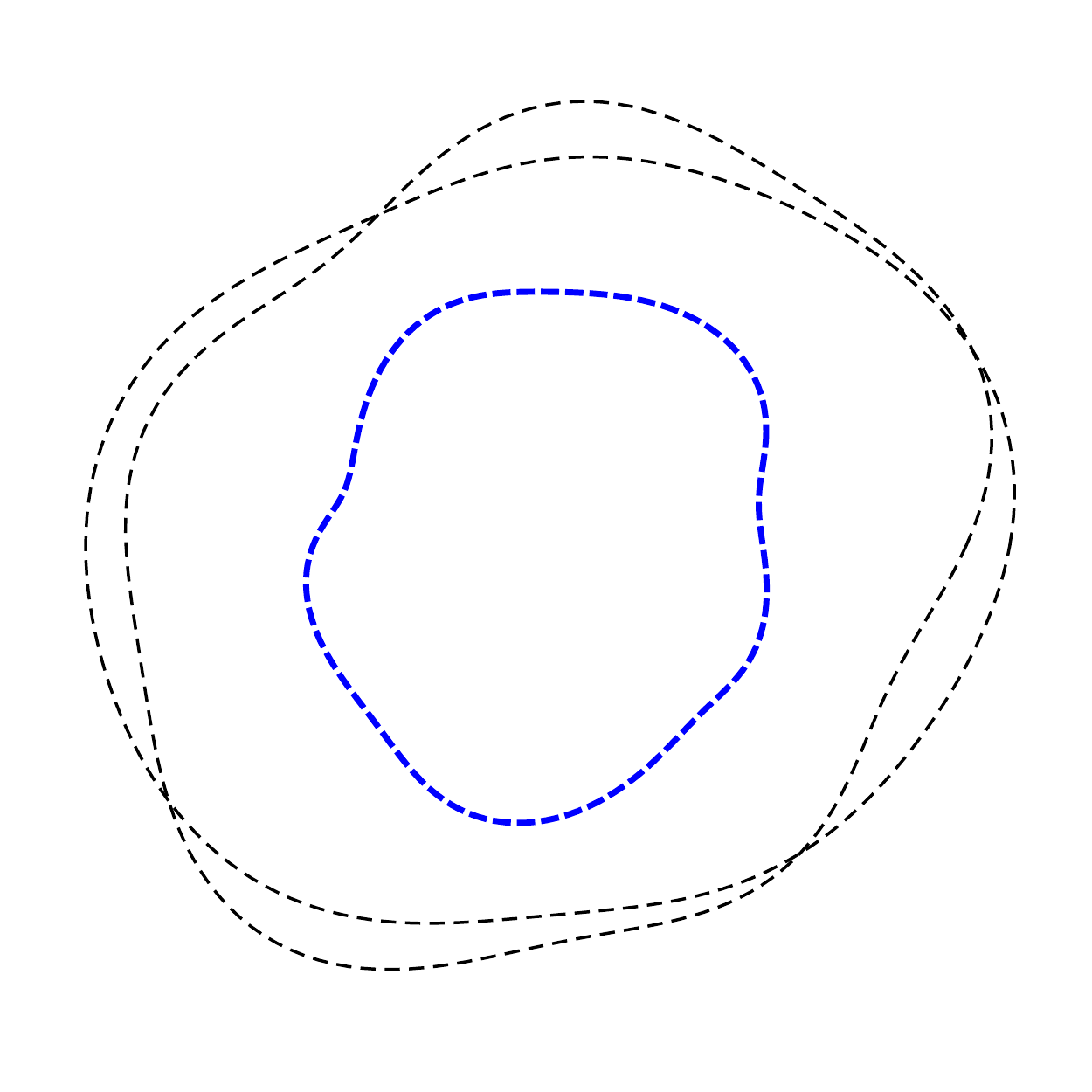} &%
				\includegraphics[width=\figwidthA]{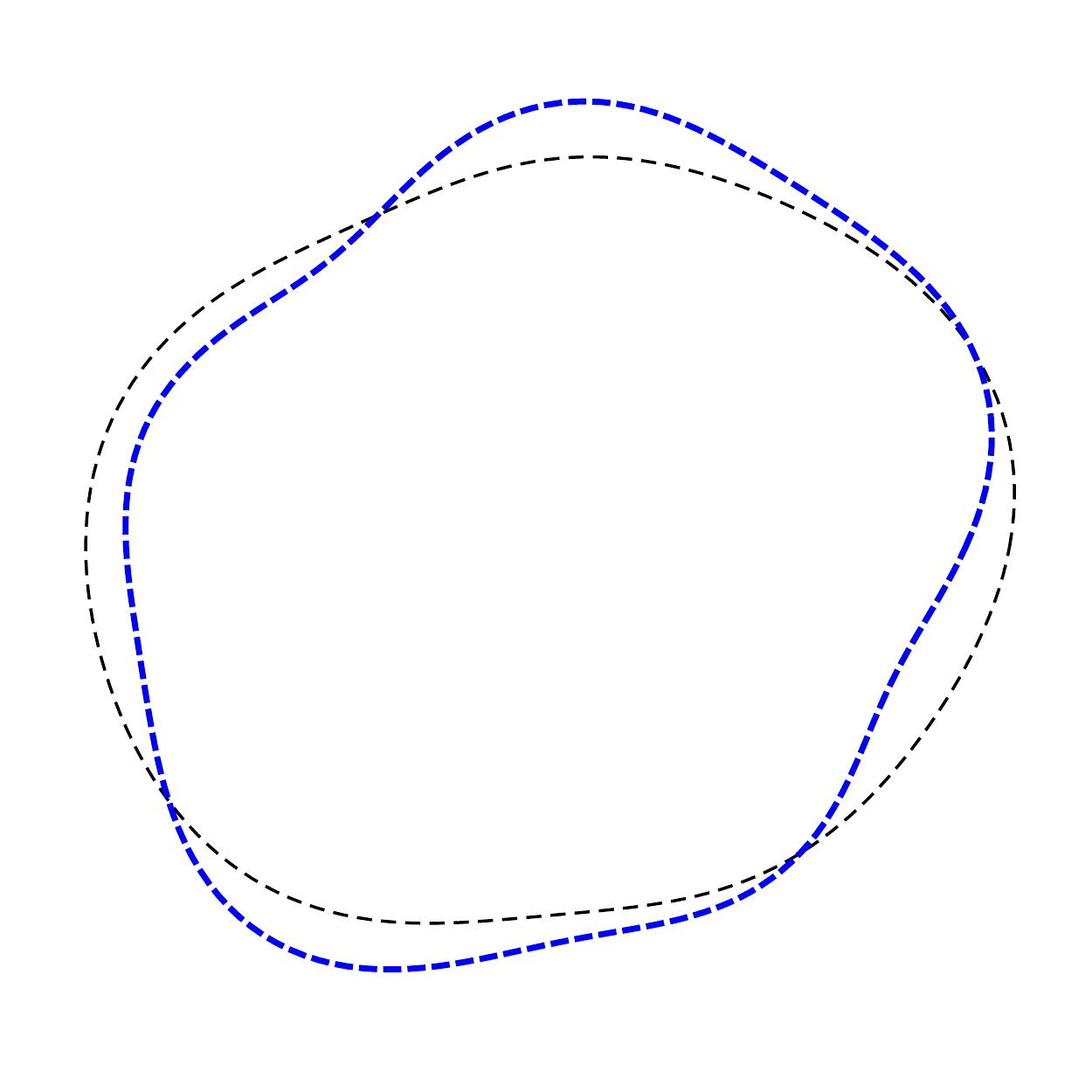} &%
				\includegraphics[width=\figwidthA]{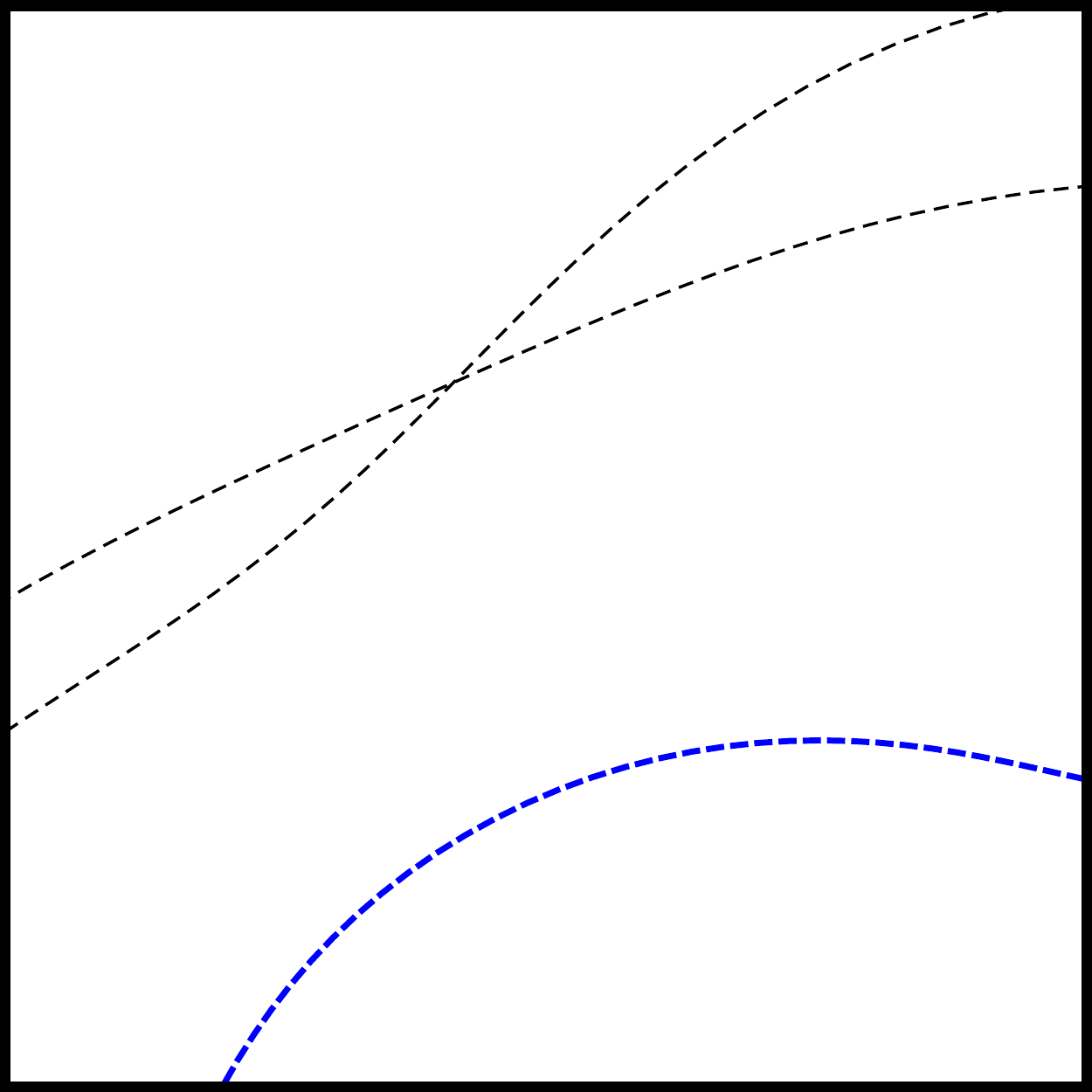} \\%
			\end{tabular}
			\caption{Linear structure on shape representations. From left to right: first four columns show a linear interpolation between two shapes in different representations. Fifth column shows a close-up of the second step for better visibility.
				First row: indicator functions (contours given for visual orientation). The intermediate objects are no indicator functions themselves.
				Second row: parametrized contours with compatible parametrizations. Blue contour describes a smooth interpolation between the two shapes given in black.
				Third row: similar contours with strongly differing parametrizations. The intermediate blue contours describe no meaningful interpolation.}
			\label{fig:LinearStructures}
		\end{figure}		
		
		\paragraph{Optimal Transport.}
			Optimal transport in general has become a popular tool in machine learning and image processing. For example by providing a metric on measures it can be used for classification in bag-of-feature representations \cite{Pele2009}, or one may extract object registrations from the optimal transport plan \cite{OptimalTransportWarping}.
			Wasserstein spaces have found to exhibit a structure akin to Riemannian manifolds \cite[Sect.~2.3.2]{Ambrosio2013}, which allows to interpret certain partial differential equations informally as gradient descents in these spaces \cite[Chap.~15]{Villani-OptimalTransport-09}.
			The `tangent space' of this `manifold' consists of flow-fields that describe small deformations of the footpoint measure via the continuity equation.
			Recently it has been shown that this space provides a meaningful basis for classifying measures with distinctive spatial structure \cite{OptimalTransportTangent2012}.
			In \cite{SchmitzerSchnoerr-EMMCVPR2013} this approach was extended: shapes were described by measures with constant density within their support and those flow-fields that are tangent to this subset of measures were determined.
			It was argued that this provides a shape representation which can unite the advantages of indicator function and the contour parametrization representation without suffering too much from the disadvantages: while the framework of optimal transport in its convex formulation by the Kantorovich functional provides the tool for local appearance feature matching, the tangent space provides a basis for statistical shape modelling.
			The joint functional can be optimized to global optimality by a hierarchy of adaptive convex relaxations.
		
		\paragraph{Flows and Diffeomorphisms.}
			Under sufficient regularity conditions the flow-fields that are tangent to paths of measures in the above sense can be integrated into diffeomorphisms that transform one measure into the other by push-forward.
			Flow-fields and their induced diffeomorphisms have been extensively studied in their own right, beyond the connection to optimal transport \cite{YounesShape2010,ZolesioShape2011}. These general results will help us prove our theorems.
			By choosing a metric on the space of flow-fields one induces a metric on the group of integrated diffeomorphisms by finding the `shortest' flow-field path between them. This, in turn, induces a metric on shapes, represented by open sets, by looking at the diffeomorphisms that transform one into the other.
			As pointed out, we seek to remain within the framework of optimal transport for its closeness to convex optimization methods. Hence, the corresponding marginal constraints impose restrictions on the Jacobian determinant of diffeomorphisms, similar to the constraint of volume preservation in fluid dynamics.
			In what follows, we will study this and other aspects in detail.

	\subsection{Notation}
		\label{sec:Notation}
		\paragraph{Calculus.}
		The space of test functions $\testFunction = C^\infty_0(\R^2)$ is the space of smooth real functions on $\R^2$ with compact support.
		For a multi-index $I = (i_1,i_2,\ldots,i_n)$ denote by $\partial_I \varphi = \partial_{i_1} \partial_{i_2} \ldots \partial_{i_n} \varphi$ the corresponding partial derivative and by $|I|=n$ its order.
		Given a differentiable map $\varphi : \R^2 \rightarrow \R^2$, we write $J_\varphi$ for the Jacobian matrix.
	
		\paragraph{Differential Geometry.}
		For a smooth manifold $M$ we denote by $TM$ its tangent bundle and for $x \in M$ by $T_xM$ the tangent space at $x$. Generally for any fiber bundle $\mc{F}M$ over $M$ we denote the fiber at $x$ by $\mc{F}_xM$. For a differentiable map $f: M \rightarrow N$ between two smooth manifolds, denote by $\mc{D}f: TM \rightarrow TN$ the differential.

		\paragraph{Measures.}
		The volume of a measurable set $\region \subset \R^2$ w.r.t.~the Lebesgue measure is denoted by $|\region|$.
		For a measurable space $A$ denote by $\MeasGen(A)$ the set of probability measures thereon. For a measurable map $f: A \rightarrow B$ by $f_\sharp \mu$ we denote the pushforward of measure $\mu$ from $A$ to $B$, defined by $(f_\sharp \mu)(\sigma) = \mu(f^{-1}(\sigma))$ for all measurable $\sigma \subset B$.
		We write $\spt \mu$ for the support of the measure $\mu$, which is the smallest closed set such that $\mu(A \setminus \spt \mu) = 0$.
		Throughout this paper we will assume that all measures on $\R^2$ are absolutely continuous, i.e. they yield zero on any Lebesgue negligible set. This is equivalent to the existence of a locally integrable map $\rho$ such that
		\[ \mu(A) = \int_A \rho\,dx\,. \]
		$\rho$ is called density of $\mu$ and is unique Lebesgue almost everywhere.

		\paragraph{Sobolev Spaces.}
		Denote by $H(\ddiv,\region)$ the space of square-integrable vector fields $u: \region \rightarrow \R^2$, with square integrable divergence
		\[ \ddiv u = \nabla \cdot u = \sum_{i=1}^2 \frac{\partial}{\partial x_i} u_i \,. \]
		This is a Hilbert space with scalar product and norm
		\[ \langle u, v \rangle_{\ddiv; \region} = \int_\region u \cdot v \,dx
			+ \int_\region (\ddiv u) (\ddiv v)\,dx, \qquad
			\|u\|^2_{\ddiv; \region} = \langle u, u \rangle_{\ddiv; \region}\,.
		\]

		For some Sobolev space $W$, by $[u]$ we denote the equivalence class of functions $u \in W$ that only differ by a constant. They form a unique element of the quotient space $W/\R$. The corresponding norm is given by
		\[ \| [u] \|_{W/\R} = \inf \{ \|u\|_W \,\colon\, u \in [u] \}\,. \]
		

%% file: contour-manifolds-v6-Section-2-Background.tex
\section{Mathematical Background}
	\label{sec:Background}	
	\subsection{Contour Manifolds}
		\label{sec:BackgroundContours}
		The set of embeddings of the unit circle $S^1$ into $\R^2$ can formally be treated as infinite dimensional manifold. A corresponding framework is laid out in \cite{MichorGlobalAnalysis} and an overview is given for example in \cite{Michor2006}. We now summarize the facts that are relevant for this article.
		
		\begin{definition}[Space of Smooth Mappings, Manifold of Embeddings, Manifold of Submanifolds]
			Denote by $C^\infty(S^1,\R^2)$ the vector space of smooth mappings from $S^1$ into $\R^2$, equipped with the topology of uniform convergence in all derivatives and spatial components.
			By $\emb$ we denote the set of $C^\infty$-embeddings $S^1 \rightarrow \R^2$. It is an open submanifold of $C^\infty(S^1,\R^2)$. Its tangent bundle $T\emb$ is given by $\emb \times C^\infty(S^1,\R^2)$.		
			Let $\diff$ be the Lie group of $C^\infty$-diffeomorphisms on $S^1$.
			Then, by $\sfold$ we denote the quotient set $\emb / \diff$ of equivalence classes in $\emb$, two contours in $\emb$ being equivalent if there is a reparametrization in $\diff$ that transforms one into the other by right composition.
			That is for $c_1,c_2 \in \emb$ have $c_1 \sim c_2$ if $c_2 = c_1 \circ \varphi$ for some $\varphi \in \diff$.
			
			The set $\sfold$ of equivalence classes $[c]$ on $\emb$ is then itself a smooth manifold and the continuous map
			\begin{align}
				\pi : \emb \rightarrow \sfold, \qquad c \mapsto [c]
			\end{align}
			that takes contours to their equivalence class is a principal bundle with total space $\emb$, base space $\sfold$ and structure group $\diff$.
		\end{definition}
		The relevant parts of \cite{MichorGlobalAnalysis} for this definition are: Sect.~6.1, Thm.~42.1 for the structure of $C^\infty(S^1,\R^2)$, Thm.~44.1 for the principal bundle $(\emb,\pi,\sfold,\diff)$.

		For a given contour $c \in \emb$ we denote by $n_c \in C^\infty(S^1,\R^2)$ its \emph{outward pointing unit-normal field}.
		
		\begin{definition}[Vertical and Horizontal Bundle, Horizontal Lifting \protect{\cite[Sect.~37]{MichorGlobalAnalysis}}]
			The vertical bundle on $\emb$ with respect to $\pi$, $V\emb = \ker \mc{D}\pi$, is at each point $c \in \emb$ the set of tangent vectors
			\begin{equation}
				V_c \emb = \left \{ a \in T_c \emb = C^\infty(S^1,\R^2) \,\colon\, \langle a(\theta), n_c(\theta) \rangle_{\R^2} = 0 \,\forall\,\theta \in S^1\right\}
			\end{equation}
			which are locally orthogonal to the normal field $n_c$ on $c$.
			A corresponding choice for the horizontal bundle is then given by
			\begin{equation}
				H_c \emb = \left \{ a \cdot n_c \,\colon\, a \in C^\infty(S^1,\R) \right\}\,.
			\end{equation}
			This is the orthogonal complement of $V_c \emb$ w.r.t.~the $L^2$-inner product on $T_c \emb$.
			
			The projection $\pi : \emb \rightarrow \sfold$ induces an isomorphism
			\begin{align}
				\pi_{c,\ast} : H_c \emb \rightarrow T_{\pi(c)} \sfold
			\end{align}
			whose inverse is referred to as \emph{horizontal lift}. For every tangent vector $v \in T_{\pi(c)}\sfold$ there is a unique horizontal vector field $a \in H_c \emb$ such that $\pi_{c,\ast}(a) = v$.
		\end{definition}
		
		\begin{lemma}[Horizontal Parametrization]
			\thlabel{thm:HorizontalParametrization}
			Any $C^1$ contour-family $[0,1] \ni t \mapsto c_t \in \emb$ can be reparametrized such that $\dot{c}_t = a_t \cdot n_{c_t}$, $a_t \in C^\infty(S^1,\R)$, i.e. such that the temporal deformation is normal to the contour and the tangent vectors lie in the horizontal bundle.
		\end{lemma}
		The proof is analogous to that of the proposition in \cite[Sect.~2.5]{Michor2006}, see also discussion ibid., Sect.~2.3.
		\begin{remark}
			\label{rem:HorizontalParametrization}
			Based on \thref{thm:HorizontalParametrization}, in the course of this paper, we will always describe contour deformations by scalar fields, that give the local deformation along the normal field, i.e. within the horizontal bundle. By aid of the unit normal-field on contours we will canonically identify
			\begin{equation}
				H \emb \cong \emb \times C^\infty(S^1,\R)\,.
			\end{equation}
		\end{remark}
		
		We show next that diffeomorphisms $\varphi \in \diff$ preserve horizontal lifting.
		\begin{proposition}
			\thlabel{thm:TangentEquivalence}
			For any $v \in T_{[c]} \sfold$ and $\varphi \in \diff$ have
			\begin{align}
				\label{eq:HorizontalLiftCommutation}
				\big( \pi_{c,\ast}^{-1}(v) \big) \circ \varphi = \pi_{c \circ \varphi,\ast}^{-1}(v)\,.
			\end{align}
			This implies that any element in the tangent bundle $T\sfold$ can be represented by an equivalence class $[(c,a)]$ in $H\emb$, equivalence $(c_1,a_1) \sim (c_2,a_2)$ holding when there is some $\varphi \in \diff$ such that $c_2 = c_1 \circ \varphi$ and $a_2 = a_1 \circ \varphi$.
		\end{proposition}
		\begin{proof}
			Let $c_1 \in \emb$ and $c_2 = c_1 \circ \varphi$ for some $\varphi \in \diff$. Let further $v \in T_{[c_1]}\sfold$ and let $a_1 = \pi_{c_1,\ast}^{-1}(v)$.
			There is then a horizontal $C^1$-path $c_{1,t}$ in $\emb$ with $\partial_t c_{1,t}|_{t=0}=a_1$. The path $c_{1,t} \circ \varphi$ is a horizontal path through $c_2$ at $t=0$. By differentiation we find that it is tangent to $a_2 = a_1 \circ \varphi$ in $t=0$ with $a_2 \in H_{c_2} \emb$.
			Since $\pi(c_{1,t}) = \pi(c_{2,t})$ for all times, we must have that $\mc{D}\pi(c_1,a_1) = \mc{D}\pi(c_2,a_2) = ([c_1],v)$. Hence $\pi_{c_2,\ast}(a_2) = v$ and therefore $a_2 = \pi_{c_2,\ast}^{-1}(v)$. This establishes \eqref{eq:HorizontalLiftCommutation}.
			
			Hence, analogously to $\emb$, we introduce an equivalence relation on $H\emb$ by stating $(c_1,a_1) \sim (c_2, a_2)$ if $c_2=c_1 \circ \varphi$ and $a_2 = a_1 \circ \varphi$ for some $\varphi \in \diff$. We can then represent the point $([c_1],v) \in T\sfold$ by the equivalence class $[(c_1,a_1)]$ in $H\emb$. By virtue of \eqref{eq:HorizontalLiftCommutation} all elements in $[(c_1,a_1)]$ consistently represent the same element $([c_1],v)$ and by virtue of horizontal lifting we know that every point in $([c],v) \in T\sfold$ has a representing equivalence class with one element in $H_{c}\emb$ for every $c \in [c]$.
		\end{proof}
		
		Finally we need to establish how to verify convergence on $\sfold$.
		\begin{proposition}[Convergence on $\sfold$]
			\thlabel{thm:ConvergenceOnSFold}
			A sequence $[c_n]$ in $\sfold$ converges to some $[c] \in \sfold$ if and only if there is a sequence $c'_n$ and a point $c'$ in $\emb$ with $c'_n \in [c_n]$ for all $n$ and $c' \in [c]$ such that $c'_n \rightarrow c'$ in $\emb$.
		\end{proposition}
		
		\begin{proof}
			The `if' part follows immediately from the continuity of $\pi$. The `only if' part works as follows: let $U$ be an open neighbourhood of $[c]$ in $\sfold$ such that $\pi^{-1}(U) \simeq U \times \diff$. Then, since $[c_n] \rightarrow [c]$, all $[c_n]$ will eventually lie in $U$. We can then pick any element $\varphi$ from $\diff$ and employ the local isomorphism of the fiber bundle to turn the sequence $([c_n],\varphi)$ into some sequence in $\emb$ converging to the $c$ corresponding to $([c],\varphi)$.
		\end{proof}
		
	\subsection{Flows and Diffeomorphisms}
		Flow-fields and the diffeomorphisms they induce are important tools in this article because of the way they act on subsets of $\R^2$. We collect some corresponding facts.

		Let $B$ denote the open unit ball in $\R^2$ centered at the origin and
		\begin{equation}
			B_0 = \{ x \in B : x_1 = 0 \},\quad B_+ = \{ x \in B : x_1 > 0 \}, \quad B_- \{ x \in B : x_1 < 0 \}\,.
		\end{equation}
		\begin{definition}[\protect{\cite[Defn.~3.1]{ZolesioShape2011}}]
			A subset $\region \subset \R^2$ is \emph{locally of class $C^k$} if for any $x \in \partial \region$ there exists a neighbourhood $U(x)$ of $x$ and a map $g_x \in C^k\big(U(x),B\big)$ with inverse $g^{-1} \in C^k\big(B,U(x)\big)$ such that
			\begin{equation}
				g_x\big(\interior \region \cap U(x)\big) = B_+,
				\quad g_x\big(\partial \region \cap U(x)\big) = B_0\,.
			\end{equation}
			
			If $g_x$ and $g_x^{-1}$ are also bi-Lipschitzian for all $x \in \partial \region$ then $\region$ is said to be \emph{locally $k$-Lipschitzian}.
		\end{definition}		
		
		If an open set $\region$ of class $C^\infty$ is simply connected, its boundary $\partial \region$ is diffeomorphic to $S^1$ and can be parametrized by a map $c \in \emb$.
		
		In the context of image segmentation an important type of shape functional is the integration of a given function over the interior of the shape. The following \thnameref{thm:ShapeDerivative} gives the derivative of such an integration in the contour representation w.r.t.~a contour deformation.
		\begin{lemma}[Shape Derivative]
			\thlabel{thm:ShapeDerivative}
			For a family of contours $[0,1] \ni t \mapsto c_t \in \emb$ which is $C^1$ in time and for some $\phi \in C^{\infty}_{\textnormal{loc}}(\R^2)$
			\begin{equation}
				\frac{d}{dt} \int_{\region(c_t)} \phi\,dt = \int_{\partial \region(c_t)} \phi \, \langle \dot{c}_t \circ c^{-1}_t, n_{c_t} \circ c_t^{-1} \rangle \,ds\,.
			\end{equation}
		\end{lemma}		
		\begin{proof}
			By virtue of \cite[Chap.~4, Sect.~3.3.2]{ZolesioShape2011}, for any given time $\tilde{t} \in [0,1]$, we can use the normal component of the time derivative of $c_t$ and extend it to some $C^\infty$-field around the boundary. This field will only describe the deformation correctly up to first order in $t$, being exact only in $\tilde{t}$ itself. This is however sufficient to apply \cite[Chap.~9, Thm.~4.2]{ZolesioShape2011} for this instant. As we can do this for any $t$, the proof is complete.
		\end{proof}

		Now we make some definitions similar to \cite[Sect.~8.2.1]{YounesShape2010}. The goal is to establish existence and regularity of diffeomorphisms associated to flow-fields by integration.
		
		For some bounded open $\region \subset \R^2$ and a positive integer $p$ we denote by $C^p_0(\region,\R^2)$ the Banach space of continuously differentiable vector fields $\alpha$ on $\region$, such that the support of $\alpha$ and its derivatives up to $p$-th order lies within $\region$.		
		Denote the corresponding norm by
		\begin{align}
			\|\alpha\|_{\region,p,\infty} = \sum_{I \colon |I|\leq p} \| \partial_I \alpha \|_{\region,\infty}
		\end{align}
		with $\|\cdot\|_{\region,\infty}$ denoting the supremum-norm on $\region$.
		
		We then define the set of absolutely integrable functions from $[0,1]$ to $C^p_0(\region,\R^2)$ by
		\begin{align}
			\label{eq:YounesFlowFieldSpace}
			\mc{X}_p(\region) = \{ \alpha : [0,1] \rightarrow C^p_0(\region,\R^2) \, \colon \,
				\|\alpha\|_{\mc{X}_p(\region)} < \infty \}
				\quad \text{with} \quad
			\|\alpha\|_{\mc{X}_p(\region)} = \int_0^1 \| \alpha_t \|_{\region,p,\infty} \,dt\,.
		\end{align}
		
		Given these regularity conditions, we find:		
		\begin{theorem}[\protect{\cite[Thms.~8.7,8.9]{YounesShape2010}}]
			\thlabel{thm:FlowDiffs}
			A flow-field path $\alpha \in \mc{X}_p(\region)$ induces a family of diffeomorphisms $\varphi_t$, $t \in [0,1]$, on $\region$ via the differential equation
			\begin{align}
				\partial_t \varphi_t = \alpha_t \circ \varphi_t, \qquad \varphi_0 = \id\,.
			\end{align}
			$\varphi_t$ is $p$-times differentiable and for all $I$ with $|I| \leq p$ have
			\begin{align}
				\partial_t \partial_I \varphi_t = \partial_I (\alpha_t \circ \varphi_t)
			\end{align}
			with corresponding initial conditions.
		\end{theorem}
	
		We will later need the following small Lemma, based on the theorem above.
		
		\begin{lemma}[Uniform convergence of $\varphi_t$]
			\thlabel{thm:UniformDPhiContinuity}
			For $\alpha \in \mc{X}_p(\region)$ the corresponding family of diffeomorphisms $\varphi_t$ according to \thref{thm:FlowDiffs} is continuous in time w.r.t.~uniform convergence in its derivatives up to $p$-th order.
		\end{lemma}
		\begin{proof}
			In $0$-th order we have
			\begin{align}
				\|\varphi_{t_1}-\varphi_{t_2}\|_{\region,\infty} & \leq \int_{t_1}^{t_2} \| \partial_t \varphi_t \|_{\region,\infty}\,dt = \int_{t_1}^{t_2} \|\alpha_t\|_{\region,\infty}\,dt \rightarrow 0 \quad \text{as} \quad t_1 \rightarrow t_2\,.
			\end{align}
			For all orders from $1$ up to $p$ the proof can be established by induction: according to \thref{thm:FlowDiffs} for any multi-index $I$, $1 \leq |I|\leq p$ have
			\begin{align}
				\partial_t \partial_I \varphi_t & = \partial_I (\alpha_t \circ \varphi_t) \\
				\intertext{Using \cite[Lemma~8.3]{YounesShape2010} to disentangle the expression one finds}
					& = \sum_j \big((\partial_j \alpha_t) \circ \varphi_t \big) \big( \partial_I (\varphi_t)_j \big) + C_t
			\end{align}
			where $C_t$ is a combination of derivatives up to order $|I|-1$ of $\varphi_t$ and of derivatives up to order $|I|$ of $\alpha_t$. Assuming the Lemma holds for orders up to $|I|-1$ and using the assumption  $\alpha \in \mc{X}_p(\region)$, we can find some bound $\hat{C}_t$ with $\hat{C}_t \geq \| C_t \|_{\region,\infty}$ and $\int_0^1 \hat{C}_t dt < \infty$.
			Consider then the following ODE:
			\begin{align}
				\label{eq:FlowConvergenceProof1}
				\partial_t \hat{\varphi}_t & = \hat{\alpha_t}\,\hat{\varphi}_t + \hat{C}_t \qquad \text{with} \qquad \hat{\alpha}_t = \sum_j \|\partial_j \alpha_t \|_{\region,\infty}
			\end{align}
			For some initial condition $\hat{\varphi}_0 \geq \max_{j} \|\partial_I (\varphi_0)_j\|_{\region,\infty}$ the continuous solution $\hat{\varphi}_t$ to \eqref{eq:FlowConvergenceProof1} will satisfy $\hat{\varphi}_t \geq \max_{j} \|\partial_I (\varphi_t)_j\|_{\region,\infty}$.
			One then has
			\begin{align}
				\|\partial_I \varphi_{t_1} - \partial_I \varphi_{t_2} \|_{\region,\infty} & \leq
					\int_{t_1}^{t_2} \| \partial_t \partial_I \varphi_t \|_{\region,\infty} \, dt \leq
					\int_{t_1}^{t_2} \hat{\alpha}_t\,\hat{\varphi}_t + \hat{C}_t \, dt \rightarrow 0 \quad \text{as} \quad t_1 \rightarrow t_2\,.
			\end{align}
			For $|I|=1$ one finds $C_t = 0$, i.e.~the first step holds. The higher orders then follow from induction.
		\end{proof}
		
	\subsection{Optimal Transport and Riemannian Structures in the Space of Measures}
		\label{sec:BackgroundWasserstein}
		Throughout this article optimal transport will provide the underlying framework for representing shapes in terms of measures. Here we recall some definitions, in particular on the pseudo-Riemannian structure of Wasserstein spaces.
				
		Denote by $\Meas$ the space of probability measures on $\R^2$ with finite second moments. This set can be metrized with the Wasserstein distance, defined next.
		\begin{definition}[Wasserstein Distance]
			For $\mu_1, \mu_2 \in \Meas$ let the Wasserstein distance be given by
			\begin{align}
				\label{eq:WassersteinDistance}
				\WD(\mu_1,\mu_2) & = \left( \inf \left \{ \int_{\R^2 \times \R^2} \|x-y\|^2 d\mu(x,y) \,\colon\, \mu \in \coupling(\mu_1,\mu_2) \right\} \right)^{1/2} \\
				\intertext{where}
				\label{eq:Couplings}
				\coupling(\mu_1,\mu_2) & = \left \{ \mu \in \MeasGen(\R^2 \times \R^2) \, \colon \,
					\project_{1\,\sharp} \mu = \mu_1,\,
					\project_{2\,\sharp} \mu = \mu_2
					\right \}
			\end{align}
			denotes the set of couplings between $\mu_1$ and $\mu_2$. $\project_i$ denotes the projection onto the first and second marginal respectively.
		\end{definition}
		For a proof that $\WD$ in fact satisfies the axioms of a metric as well as a general thorough introduction to the subject of optimal transport see for example \cite{Villani-OptimalTransport-09}.
		
		The minimizing coupling in \eqref{eq:WassersteinDistance} can be thought of as describing how the mass from $\mu_1$ is arranged into $\mu_2$ in the most efficient fashion. The notion of optimal transport, given above, can be seen as somewhat static: only the cost $\|x-y\|^2$ and the final assignment $\mu \in \coupling(\mu_1,\mu_2)$ play a role but not, how the mass in $\mu_1$ `moves' through $\R^2$ to actually reach the distribution of $\mu_2$. The following definitions and statements reveal a more dynamic perspective on optimal transport.
		
		\begin{definition}[Continuity Equation]
			For a measure path $[0,1] \ni t \mapsto \mu_t \in \Meas$ and a flow vector field path $[0,1] \ni t \mapsto \alpha_t \in (L^2(\mu_t))^2$ the continuity equation states that
			\begin{subequations}
			\begin{equation}
				\label{eq:ContinuityEq}
				\frac{d}{dt} \mu_t + \nabla \left( \alpha_t\,\mu_t \right) = 0
			\end{equation}
			in the sense of distributions. That is
			\begin{equation}
				\frac{d}{dt} \int \phi\,d\mu_t = \int \langle \nabla \phi, \alpha_t \rangle \, d\mu_t
			\end{equation}
			\end{subequations}
			for all test functions $\phi \in \testFunction$.
		\end{definition}
		
		\begin{definition}[Absolutely Continuous Paths in Metric Spaces \protect{\cite[Def.~2.28]{Ambrosio2013}}]
			Let $[0,1] \ni t \mapsto x_t \in X$ be a path in a metric space $(X,d)$. $(x_t)$ is said to be absolutely continuous if there exists a function $f \in L^1(0,1)$ such that
			\begin{equation}
				d(x_s,x_t) \leq \int_s^t f(r)\,dr,\qquad \forall s < t \in [0,1]\,.
			\end{equation}
		\end{definition}
		
		On $\Meas$, equipped with the metric $\WD$, there is a more useful characterization of absolutely continuous paths:
		\begin{theorem}[Characterization of Absolutely Continuous Paths in the Space of Measures \protect{\cite[Thm.~2.29]{Ambrosio2013}}]
			A measure path $[0,1] \ni t \mapsto \mu_t$ is absolutely continuous (up to redefinition of $\mu_t$ on a $t$-negligible set) if and only if there exists a flow field path $\alpha_t$ such that $(\mu_t,\alpha_t)$ satisfy the continuity equation \eqref{eq:ContinuityEq} and $\int_0^1 \|\alpha_t\|_{L^2(\mu_t)} dt < \infty$.
		\end{theorem}
		
		\begin{remark}
			In this article the attribute of absolute continuity has been discussed in the context of measures and of paths of measures. They need to be carefully distinguished.
		\end{remark}
		
		This notion of absolutely continuous paths allows for an alternative variational formulation of the Wasserstein distance.
		
		\begin{proposition}[Benamou-Brenier formula \protect{\cite[Proposition 2.30]{Ambrosio2013}}]
			Let $\mu_0,\mu_1 \in \Meas$. Then it holds
			\begin{equation}
				\label{eq:BenamouBrenier}
				\WD(\mu_0,\mu_1) = \inf \left \{ \int_0^1 \|\alpha_t \|_{L^2(\mu_t)} dt \right \}\,,
			\end{equation}			
			where the infimum is taken among all weakly continuous distributional solutions of the continuity equation $(\mu_t,\alpha_t)$ such that $\mu_{t=0}=\mu_0$ and $\mu_{t=1}=\mu_1$.
		\end{proposition}
		
		The path of vector fields that a certain measure path satisfies the continuity equation with is in general not unique: if $\nabla \cdot (\beta_t\,\mu_t) = 0$ in the distributional sense for a.e.~$t$, then the pair $(\mu_t,\alpha_t + \beta_t)$ will satisfy the continuity equation if $(\mu_t,\alpha_t)$ does. The minimization in \eqref{eq:BenamouBrenier} however gives a natural way to select a distinct vector field. It turns out that minimizers of \eqref{eq:BenamouBrenier} for a.e.~$t$ lie in a particular subspace of vector fields, to be specified next.
		
		\begin{definition}[The Tangent Space \protect{\cite[Def.~2.31]{Ambrosio2013}}]
			\label{def:TangentSpace}
			Let $\mu \in \Meas$. Then the tangent space $\Tan(\mu)$ at $\mu$ is defined as
			\begin{align}
				\Tan(\mu) = & \overline{ \left \{
					\nabla u \, \colon \, u \in \testFunction
					\right \} }^{L^2(\mu)}\,.
			\end{align}
		\end{definition}
		\begin{remark}
			\label{rem:TangentNotion}
		The term `tangent' stems from the notion, that some $\alpha \in \Tan(\mu)$ can describe small deformations of $\mu$ of the form $t \mapsto \mu_t = (\id+t \cdot \alpha)_\sharp \mu$ such that $\mu_t$ and $\alpha$ satisfy the continuity equation at $t=0$.
			$\Tan(\mu)$ is not so much to be thought of as set of tangent vectors directly but more as set of representatives of them.
			We will refer to functions $u$, whose gradients represent tangent vectors, as \emph{potential functions}.
		\end{remark}

		The expression $\|\alpha_t \|_{L^2(\mu_t)}$ can be interpreted as the (pseudo-)norm of $\alpha_t$, induced by the following inner product on $\Tan(\mu)$:
		\begin{equation}
			\label{eq:WassersteinRiemannInnerProduct}
			\langle \alpha_1, \alpha_2 \rangle_{L^2(\mu)} = \int \langle \alpha_1, \alpha_2 \rangle_{\R^2}\,d\mu \qquad \textnormal{for } \alpha_1, \alpha_2 \in \Tan(\mu)\,.
		\end{equation}
		So $\WD$ appears in fact to behave similar to a Riemannian metric on $\Meas$ with the Riemannian inner product given by \eqref{eq:WassersteinRiemannInnerProduct}: the distance between two elements is given by the length of the shortest path between them and path length is measured by integrating along the path the norm of the tangent vectors w.r.t.~a local tangent metric.
		In \cite[Sect.~2.3.2]{Ambrosio2013} further results are given that motivate to consider the viewpoint of $\Meas$, metrized by $\WD$ as a Riemannian manifold.
		
		The set of absolutely continuous probability measures with smooth density functions is viewed in \cite{LottWassersteinRiemannian2008} as a manifold in the precise sense of \cite{MichorGlobalAnalysis}. Expressions for typical notions in differential geometry, such as the Levi-Civita connection, parallel transport or the geodesic equation are derived.
		It is however a very tedious task to extend these results in formally rigorous way to less smooth settings.
		
		\label{sec:BackgroundWassersteinManifold}
		We will now recall some of the results from \cite{LottWassersteinRiemannian2008} for the particular case of optimal transport on a compact subset of $\R^2$. Here, we will not explicitly denote limitation but simply assume that we are on some compact subset, but the set is large enough such that for all our purposes it looks like the whole $\R^2$. We consider the set $\Meas^\infty$ of measures that are absolutely continuous and have a smooth density function. Hence, concepts like the continuity equation can be expressed directly in terms of flow-fields and density functions, and one need not fall back on a distributional formulation.
		
		The tangent space at a point $\mu \in \Meas^\infty$ is isomorphic to $\{ \nabla u : u \in C^\infty(\R^2) \}$, where due to the additional smoothness, in contrast to Definition \ref{def:TangentSpace} we do not consider the $L^2$-completion. Any element in the tangent space describes a local deformation of the footpoint measure as discussed in Remark \ref{rem:TangentNotion}.
		More generally, a function $u \in \testFunction$ and its gradient not only represent a tangent vector at one footpoint $\mu$, one can think of them as representing a tangent vector at any footpoint on $\Meas^\infty$, i.e.~a vector field \cite[Sect.~2]{LottWassersteinRiemannian2008}. One then finds:
		
		\begin{proposition}[Covariant Derivative on $\Meas^\infty$ \protect{\cite[Sect.~2]{LottWassersteinRiemannian2008}}]
			\thlabel{thm:WassersteinCovariantDerivative}
			Let $\nabla u_1, \nabla u_2$ represent two vector fields on $\Meas^\infty$. Then the Levi-Civita covariant derivative $\ol{\nabla}_{u_1} u_2$ of $u_2$ w.r.t.~$u_1$ is given by
			\begin{align}
				\label{eq:WassersteinCovariantDerivative}
				(\ol{\nabla}_{u_1} u_2)_i & = \sum_{j=1}^2 (\partial_j u_1) (\partial_i \partial_j u_2)\,.
			\end{align}
		\end{proposition}
		This can be interpreted as the change that the flow field $\nabla u_2$ exhibits when being pushed along the flow induced by flow-field $\nabla u_1$.
		
		This covariant derivative only applies to vector fields on $\Meas^\infty$ for which the vector at any footpoint is represented by the gradient of the same static function $u_2$. In general, this function can also change throughout $\Meas^\infty$. Let $\mu_t$ be a path in $\Meas^\infty$ with tangent vector $\nabla u_t$ at time $t$ and let $\nabla w_t$ be a vector field on the path $\mu_t$ with potential function $w_t$ at footpoint $\mu_t$. Then the covariant derivative of $\nabla w_t$ at $\mu_t$ w.r.t.~$\nabla u_t$ is given by
		\begin{align}
			\label{eq:WassersteinCovariantDerivativeDT}
			(\ol{\nabla}_{u_t} w_t)_i & = \sum_{j=1}^2 (\partial_j u_t) (\partial_i \partial_j w_t) + \partial_t \partial_i w_t \,.
		\end{align}
		This is the change of the vector field as given by \thref{thm:WassersteinCovariantDerivative}, plus the change induced by the change of the potential function $w_t$ along the path.
		
		Since a vector $\nabla u_t$ is specified by the unique potential $u_t$, up to a constant, one can for $w_t = u_t$, by suitably fixing this constant, express the covariant derivative directly in terms of the potential function:
		\begin{align}
			\label{eq:WassersteinCovariantDerivativePotential}
			(\ol{\nabla}_{u_t} u_t) & = \nabla \left( \frac{1}{2} \|\nabla u_t\|^2 + \partial_t u_t \right) \,.
		\end{align}
		Setting this, the covariant derivative of $u_t$ along itself, to zero, one finds the geodesic equation on $\Meas^\infty$ in terms of the potential function $u_t$:
		
		\begin{proposition}[\protect{\cite[Prop.~4]{LottWassersteinRiemannian2008}}]
			The geodesic equation on $\Meas^\infty$ in terms of the potential function $u_t$ is given by
			\begin{align}
				\label{eq:WassersteinGeodesicEquation}
				\partial_t u_t +\frac{1}{2}\| \nabla u_t \|^2 = 0\,.
			\end{align}
		\end{proposition}
		
		\begin{remark}
			\label{rem:GeodesicGenerality}
			This geodesic equation is in fact known to be satisfied in a more general setting as $\Meas^\infty$. For absolutely continuous measures it is also satisfied by minimizers to \eqref{eq:BenamouBrenier}.
		\end{remark}
		
	\subsection{Poisson's Equation}
		In Section \ref{sec:MathContrib} a map from the tangent space $T_c M$ of deformations of some contour $c$ onto the tangent space $\Tan(\mu)$ of a suitable measure $\mu$ will be defined through solutions to Poisson's equation with appropriate data terms. For the analysis of this map general facts on existence, uniqueness and smoothness properties of solutions to Poisson's equation will be used.

		We assume for now that $\region \subset \R^2$ is a simply connected, bounded, open set with a Lipschitz-continuous boundary.
		
		\begin{lemma}[\protect{\cite[Thm.~2.5]{Girault-Raviart-86}}]
			\label{thm:ProjectionNormalComponent}
			The operator $\deliftPoint \colon C^{\infty}(\ol{\region},\R^2) \to C^{\infty}(\partial\region)$ mapping a vector to its normal component on the boundary, can be continuously extended to an operator $\deliftPointSobolev \colon H(\ddiv;\region) \to H^{-1/2}(\partial\region)$.
		\end{lemma}

		The following Green's formula holds:
		\begin{equation}
			\label{eq:Green-Hdiv}
			\int_{\region} u \cdot (\nabla q)\,dx + \int_{\region} (\ddiv u) q\,dx
				= \int_{\partial\region} (u \cdot n) q\,ds, \qquad
				u \in H(\ddiv;\region),\quad q \in H^{1}(\region),
		\end{equation}
		where the integral on the r.h.s.~is understood as the duality pairing between $H^{-1/2}(\partial\region)$ and $H^{1/2}(\region)$.

		The following Lemma can be deduced from the basic theory of elliptic PDEs \cite[Prop.~1.2, Cor.~2.7]{Girault-Raviart-86}.
		\begin{lemma}
			\thlabel{thm:NeumannSobolev}
			Let $g \in H^{-1/2}(\partial\region)$ be given. Then the mapping $\liftPointSobolev \colon H^{-1/2}(\partial \region) \to H(\ddiv;\region)$ given by $\alpha = \liftPointSobolev(g) = \nabla u$, where $u$ is up to a constant the unique solution to the Neumann problem
		\begin{equation} \label{eq:LiftingSobolev}
		\Delta u = f \;\; \text{in}\;\region,\qquad
		\frac{\partial u}{\partial n} = g \;\;\text{on}\;\partial\region,\qquad
		f = \frac{1}{|\region|} \int_{\partial\region} g ds,
		\end{equation}
		maps $g$ to the unique vector field $\alpha$ with constant divergence $\ddiv \alpha = f$ and normal component $\deliftPointSobolev(\alpha) = n \cdot \alpha = g$.
		\end{lemma}

		The solution $u$ to \eqref{eq:LiftingSobolev} inherits additional regularity of the data as follows.
		\begin{theorem}[{\cite[Thm.~1.10]{Girault-Raviart-86}}]
		\thlabel{thm:Regularity}
		Let $\region \subset \R^{2}$ be a locally $(m+1)$-Lipschitzian domain with boundary $\partial \region$ and assume that the data $f$ and $g$ satisfy
		\begin{equation}
		f \in W^{m,p}(\region),\qquad 
		g \in W^{m+1-1/p,p}(\partial\region),\qquad
		1 < p < \infty.
		\end{equation}
		Then $[u] \in W^{m+2,p}(\region)/\R$ and there exists a constant $C = C(m,p,\region)$ such that
		\begin{equation}
		\|[u]\|_{W^{m+2,p}(\region)/\R} \leq C( \|f\|_{W^{m,p}(\region)}
		+ \|g\|_{W^{m+1-1/p,p}(\partial\region)} ).
		\end{equation}
		\end{theorem}
		Based on the assumptions of Theorem \ref{thm:Regularity} and the Sobolev embedding theorem \cite[Thm.~1.3]{Girault-Raviart-86}, the continuous injection
		\begin{equation}
		W^{m,p}(\region) \hookrightarrow C^{n}(\ol{\region}),\qquad
		1/p < (m-n)/2
		\end{equation}
		holds.
		
		\begin{remark}
			\thlabel{rem:MoreRegularity}
			\thref{thm:Regularity} in fact holds for more general differential operators than the Laplacian.
			The weak solution to \eqref{eq:LiftingSobolev} is given by the minimizer of
			\begin{align}
				\label{eq:NeumannEnergy}
				E(u, \region, f, g) = & \frac{1}{2} \int_\region \| \nabla u\|^2\,dx + \int_\region f\,u\,dx - \int_{\partial \region} g\,u\,ds \\
				\intertext{over $u \in H^1(\region)/\R$. More generally, \thref{thm:Regularity} holds for minimizers of functionals of the type}
				E(u, \region, A, f, g) = & \frac{1}{2} \int_\region \langle A \nabla u, \nabla u \rangle dx + \int_\region f\,u\,dx - \int_{\partial \region} g\,u\,ds
			\end{align}
			with spatially varying matrix $A \in C^{m+1}(\ol{\region},\R^{2\times2})$ such that there are constants $0 < \lambda < \Lambda$ with
			\begin{align}
				\label{eq:LambdaBounds}
				\lambda \|r\|^2_{\R^2} \leq \langle A(x) r, r \rangle \leq \Lambda \|r\|^2_{\R^2}
			\end{align}
			for all $x \in \ol{\region}$ and $r \in \R^2$.
			
			In particular the dependency of the constant $C$ on $A$ is only through the parameters $\lambda, \Lambda$ and on upper bounds to the supremum norms $\sup_{x \in \ol{\region}} |\partial_I A_{ij}(x)|$ of derivatives up to order $m+1$ of coefficients of $A$. That is, for a set of matrices for which common $\lambda, \Lambda$ and common upper bounds can be found, the same constant $C$ in \thref{thm:Regularity} can be applied uniformly.
			For a detailed exposition of regularity results based on $C^{\infty}$ assumptions, we refer to \cite{Schechter1959,Agmon1959}.
		\end{remark}
		
		Now we adopt the settings for the main part of this paper, that is $\region$ is of class $C^\infty$, bounded and simply connected, $f$ is constant as in \eqref{eq:LiftingSobolev} and $g \in C^{\infty}(\partial\region)$.
		First we establish additional regularity of the images $\liftPointSobolev(g)$ in \thref{thm:NeumannSobolev}.
		If $m, p \to \infty$, as in this setting, then $n < m - 2/p \to \infty$. In view of the isomorphism established by equation \eqref{eq:Green-Hdiv} due to \thref{thm:NeumannSobolev}, we conclude
		\begin{proposition}
			\thlabel{thm:Neumann}
			Let $\region \subset \R^{2}$ be a $C^{\infty}$ domain, $g \in C^{\infty}(\partial\region)$ and $f = |\region|^{-1} \int_{\partial\region} g\,ds$. Then there is a unique vector field $\alpha \in C^{\infty}(\ol{\region},\R^{2})$ with $\langle \alpha, n\rangle=g$ and $\ddiv \alpha = f$, given by the unique solution to the corresponding Neumann problem \eqref{eq:LiftingSobolev}.
		\end{proposition}

%% file: contour-manifolds-v6-Section-3-MathContrib.tex
\section{Equivalence between Contours and Shape Measures}
	\label{sec:MathContrib}
	\subsection{Absolutely Continuous Trajectories in the Space of Shape Measures}
		We now introduce the set of shape measures which are measures whose density is given by indicator functions. All measures are normalized to have unit-mass to keep them comparable through optimal transport.
		\begin{definition}[Shape Measure]
			\label{def:ShapeMeasure}
			A measure $\mu$ is called a shape measure, if there is an open set $\region$, $0 < |\region| < \infty$ with a connected $C^\infty$-boundary such that for all measurable $A \subset \R^2$
			\begin{equation}
				\mu(A) = |\region|^{-1} \cdot |A \cap \region|\,.
			\end{equation}
			Integration w.r.t.~$\mu$ can then be written as
			\begin{equation}
				\int_A \phi\,d\mu = |\region|^{-1} \int_{A \cap \region} \phi\,dx\,.
			\end{equation}
			Denote by $\dens(\mu)$ the corresponding density function w.r.t.~the Lebesgue measure. Identifying functions that are equal a.e., the density is unique and given by $\dens(\mu) = |\region|^{-1} \indicator_\region$ where $\indicator_\region$ denotes the indicator function of $\region$. We can then write
			\begin{equation}
				\int_A \phi\,d\mu = \int_A \phi\,\dens(\mu)\,dx\,.
			\end{equation}
			Denote the set of shape measures by $\SMeas$.
		\end{definition}

		\begin{remark}
			We require that $\region$ is of class $C^\infty$ and has a connected boundary to obtain compatibility with the contour description of shapes as introduced in Sect.~\ref{sec:BackgroundContours}.
			Within this class of regular shape measures one can however describe a much more general class of shapes by metric completion in the sense of \cite[Chap.~3, Thm.~3.1]{ZolesioShape2011}.
			The analogous step in the context of contours is briefly discussed in \cite{Michor2006}.
		\end{remark}

		In analogy to Definition \ref{def:TangentSpace} we now introduce a corresponding tangent space for shape measures.
	
		\begin{definition}[Shape Tangent Space]
			\label{def:ShapeTangentSpace}
				For a shape measure $\mu$ the shape tangent space $\STan(\mu)$ at $\mu$ is defined as
			\begin{align}
				\STan(\mu) = & \left \{
					\nabla u \, \colon \, u \in \testFunction
					\wedge \Delta u = \const \text{ in } \spt(\mu)
					\right \} \,.
			\end{align}
		\end{definition}
	
		Compared to the tangent space for conventional optimal transport, Definition \ref{def:TangentSpace}, two modifications have been made: an additional constraint $\Delta u = \const$ is introduced and no completion w.r.t.~$L^2(\mu)$ is made.
	
		The first modification ensures that vectors in $\STan(\mu_t)$ correspond to deformations that, to first order, keep the density of $\mu_t$ constant within its support.
		Consider for a given shape measure $\mu$ and some $\alpha \in \STan(\mu)$ the path $\mu_t = (\id + t\cdot \alpha)_\sharp \mu$.
		The density of $\mu_t$ can be determined from the density of $\mu$ and Jacobian of $\id + t \cdot \alpha$. A quick calculation shows that near $t=0$
		\begin{align}
			\label{eq:FirstOrderJacobian}
			\det J_{\id + t\cdot \alpha} = 1 + t\cdot \ddiv \alpha + \mc{O}(t^2)\,.
		\end{align}
		Since $\ddiv \alpha = \const$ on $\spt \mu$ for $\alpha \in \STan(\mu)$, we see that deformation keeps the density of $\mu_t$ homogeneous to first order.
	
		The second change ensures that shape measure trajectories with tangents in $\STan(\mu_t)$ retain a $C^\infty$ boundary during evolution.	
		The following theorem properly establishes the relation between $\STan(\mu)$ and absolutely continuous paths of shape measures.

		\begin{theorem}[Paths with Tangents in $\STan(\mu_t)$ are Absolutely Continuous Shape Measure Paths]
			\thlabel{thm:ShapeMeasurePaths}
			Given a measure path $t \mapsto \mu_t$ and a flow field path $t \mapsto \alpha_t,\,t \in [0,1]$ such that
			\begin{enumerate}[(i)]
				\item \label{item:SMP.ShapeMeasure} $\mu_{t}$ is a shape measure at $t=0$,
				\item \label{item:SMP.AlphaSTan} $\alpha_t \in \STan(\mu_t)$,
				\item \label{item:SMP.Continuity} $\mu_t$ and $\alpha_t$ satisfy the continuity equation,
				\item \label{item:SMP.HoldAll} there is an open bounded set $\hat{\region}$ such that $\spt \mu_t \subset \hat{\region}$ for all $t \in [0,1]$,
				\item \label{item:SMP.AlphaSup} $\alpha \in \mc{X}_p(\hat{\region})$ for any positive integer $p$, see \eqref{eq:YounesFlowFieldSpace},
			\end{enumerate}
			then $\mu_t$ is an absolutely continuous shape measure path .
		\end{theorem}
		\begin{remark}
			Note hat in condition (\ref{item:SMP.AlphaSTan}) we demand $\alpha_t \in \STan(\mu_t)$ without knowing whether $\mu_t \in \SMeas$. However, the definition of $\STan(\mu_t)$ is formally valid also for this case.
		\end{remark}
		\begin{proof}[Proof]
			Since $\alpha \in \mc{X}_p(\hat{\region})$, by virtue of \thref{thm:FlowDiffs} the family of maps $\varphi_t,\,t \in [0,1]$ defined by
			\[
				\partial_t \varphi_t = \alpha_t \circ \varphi_t,
				\qquad \varphi_0 = \id
			\]
			is a family of $C^\infty$-diffeomorphisms on $\hat{\region}$. One has $J_{\varphi^{-1}} = (J_{\varphi} \circ \varphi^{-1})^{-1}$. Therefore, one finds
			\begin{align}
				0 & = \partial_t \left( \varphi_t \circ \varphi_t^{-1} \right)
					= (\partial_t \varphi_t) \circ \varphi_t^{-1} + (J_{\varphi_t} \circ \varphi_t^{-1}) \partial_t(\varphi_t^{-1}) \\
				\intertext{and then}
				\partial_t (\varphi_t^{-1}) & = -(J_{\varphi_t} \circ \varphi_t^{-1})^{-1} \big((\partial_t \varphi_t) \circ \varphi_t^{-1}\big)
					= -J_{\varphi_t^{-1}} \, \alpha_t\,.
			\end{align}
			Consider now the following integral and its time derivative for a test function $\phi \in \testFunction$:
			\begin{align}
				& \frac{d}{dt} \int \phi\,d(\varphi^{-1}_{t\,\sharp} \mu_t) =
					\frac{d}{dt} \int \phi \circ \varphi^{-1}_t \, d\mu_t \nonumber \\
				= & \int \langle (\nabla \phi) \circ \varphi^{-1}_t, 
					\underbrace{\partial_t \varphi_t^{-1}}_{-J_{\varphi_t^{-1}} \, \alpha_t}
					\rangle\,d\mu_t
					+ \int \underbrace
						{\langle \nabla ( \phi \circ \varphi_t^{-1}), \alpha_t \rangle }_{ \langle (\nabla \phi) \circ \varphi^{-1}_t, J_{\varphi^{-1}_t}\,\alpha_t \rangle}
					d\mu_t = 0
			\end{align}
			Where we applied the chain rule in the first term and the continuity equation in the second. Thus we have by integration
			\begin{equation}
				\int \phi\,d(\varphi^{-1}_{{t_1}\,\sharp} \mu_{t_1}) = \int \phi\,d(\varphi^{-1}_{{t_2}\,\sharp} \mu_{t_2})
			\end{equation}
			for all $t_1,t_2 \in [0,1]$.
			From \cite[Thm.~1.2.5]{Hormander1990} we know that if
			\begin{equation*}
				\int f\,\phi\,dx = \int g\,\phi\,dx \qquad \textnormal{for all} \quad \phi \in \testFunction
			\end{equation*}
			for locally integrable $f,g$ then $f=g$ almost everywhere. Let $f$ and $g$ be the density functions of $(\varphi^{-1}_{{t_i}\,\sharp} \mu_{t_i})$ for $i=1,2$. These exist since the density functions of $\mu_{t_i}$ exist and $\varphi^{-1}_{t_i}$ are diffeomorphisms. Then we can conclude that these density functions agree a.e.~and hence the measures are in fact identical.
			
			We thus have
			\begin{align}
				\varphi^{-1}_{t\,\sharp} \mu_t & = \varphi^{-1}_{0\,\sharp} \mu_0 = \mu_0 \\
				\intertext{and by conjugation of the push-forward with $\varphi_t$ find}
				\mu_t & = \varphi_{t\,\sharp} \mu_0\,.
			\end{align}
			Now check the Jacobian determinant of $\varphi_t$: recall for a differentiable family of matrices
			\begin{equation}
				\frac{d}{dt} \det(A_t) = \det(A_t)\,\tr(A^{-1}_t\,\partial_t A_t)\,.
			\end{equation}
			From \thref{thm:FlowDiffs} we have that the Jacobian of $\varphi_t$ satisfies
			\begin{align}
				\partial_t J_{\varphi_t} = (J_{\alpha_t} \circ \varphi_t ) J_{\varphi_t}
			\end{align}
			thus we find
			\begin{align}
				\frac{\partial}{\partial t} \det(J_{\varphi_t}) & = \det(J_{\varphi_t}) \tr( J_{\varphi_t}^{-1} \, \partial_t {J}_{\varphi_t} ) \\
					& = \det(J_{\varphi_t}) \tr( J_{\varphi_t}^{-1} \, (J_{\alpha_t} \circ \varphi_t) J_{\varphi_t} ) \\
					& =\det(J_{\varphi_t}) (\ddiv \alpha_t) \circ \varphi_t
			\end{align}
			where we have interpreted the vector field $\alpha_t$ as a map and denote its Jacobian accordingly by $J_{\alpha_t}$.
			Since for any $x \in \spt(\mu_0)$ the path $\varphi_t(x)$ over $t$ always lies within the support of $\mu_t$ for all $t \in [0,1]$ and since $\ddiv \alpha_t = \const$ throughout $\spt \mu_t$, the temporal derivative of the determinant of the Jacobian of $\varphi_t$ is spatially constant.
			Since $J_{\varphi_0} = J_{\id} = 1$, i.e. $\det(J_{\varphi_0})=1$, one finds $\det J_{\varphi_t}$ is spatially constant at all times within $\spt(\mu_0)$.
			Since $\mu_0$ is a shape measure, it has a density function, with constant value within $\spt(\mu_0)$ and zero elsewhere, i.e. a rescaled indicator function. One can then, through the push-forward via $\varphi_t$ find density functions for other $t \in [0,1]$. Since $\det J_{\varphi_t}$ is constant within $\spt(\mu_0)$ one easily finds, that the density functions of $\mu_t$ is also a rescaled indicator function. Since $\varphi_t$ is a $C^\infty$-diffeomorphism it preserves simple connectedness and $C^\infty$-smoothness of the boundary of $\spt(\mu_0)$. Therefore $\mu_t$ must be a shape measure at all times.
			Absolute continuity of the path $\mu_t$ is given by the assumption $\alpha \in \mc{X}_p(\hat{\region})$, from which absolute integrability with respect to $L^2(\mu_t)$ follows.
		\end{proof}
		
		\begin{remark}
			Note that so far the term tangent space is only used in a sense of analogy, in the way that \cite{Ambrosio2013} discusses the weak Riemannian structure of $\Meas$.
		\end{remark}
	
	\subsection{Lifting of Contours}
		Every contour $c \in \emb$ has a well-defined interior $\region(c)$ of class $C^\infty$. We formally define the map that takes $c$ to the shape measure associated with $\region(c)$.
		\begin{definition}[Lifting of Contours]
			\label{def:ContourLifting}
			For a $C^\infty$-embedding $c: S^1 \rightarrow \R^2$ that parametrizes the boundary of some open, simply connected domain $\region(c)$ the corresponding shape measure $\liftPoint(c)$ is given by
			\begin{subequations}
			\begin{align}
				\big(\liftPoint(c)\big)(A) & = |\region(c)|^{-1} \cdot | A \cap \region(c)| \,.\\
				\intertext{As in \ref{def:ShapeMeasure}, integration w.r.t.~$\liftPoint(c)$ is given by}
				\int_A \phi \,d\liftPoint(c) & = |\region(c) |^{-1} \int_{A \cap \region(c)} \phi \,dx\,.
			\end{align}
			\end{subequations}
		\end{definition}
		
		It is evident that if two contours are related by some $C^\infty$-diffeomorphism $\varphi : S^1 \rightarrow S^1$ such that $c_1 = c_2 \circ \varphi$ then $\liftPoint(c_1) = \liftPoint(c_2)$. Vice versa, if two contours $c_1$ and $c_2$ both parametrize the boundary of some shape measure $\mu$, then there is a $C^\infty$-diffeomorphism $\varphi$ such that $c_1 = c_2 \circ \varphi$. Therefore we have:
		
		\begin{proposition}
			\thlabel{thm:LiftPointInverse}
			 The map
			 \begin{equation}
			 	\liftPointEquiv : \sfold \rightarrow \SMeas\,, \qquad [c] \mapsto \liftPoint(c)
			 \end{equation}
			 is a bijection between the quotient manifold $\sfold$ and the space of shape measures $\SMeas$.
		\end{proposition}
		
	\subsection{Lifting of Contour Tangent Vectors}
		Let $c \in \emb$ describe a shape in the contour representation and let $\mu = \liftPoint(c)$ be the corresponding shape measure representation.
		Let $\alpha \in \STan(\mu)$ describe to first order a deformation of $\mu$. The information encoded in $\alpha$ can also be encoded in some normal deformation $a \in H_c \emb$ of the contour $c$.
		This raises the question how the two descriptions for deformation are related. As it turns out $\alpha$ is already completely determined by its behaviour on the boundary of $\spt(\mu)$. This can be used to define maps that convert between $a$ and $\alpha$.
		
		\begin{definition}[Lifting of Contour Tangent Vectors]
			\label{def:TangentLifting}
			For a contour $c$ and a normal deformation field $a \in H_c \emb$, see Remark \ref{rem:HorizontalParametrization}, we define the lifting
			\begin{align}
				\liftTangent_c : H_c \emb \rightarrow \STan \big(\liftPoint(c) \big), \qquad a \mapsto \alpha = \liftTangent_c(a)
			\end{align}
			from $c$ onto $\liftPoint(c)$ as the gradient of the extended unique solution (up to constant shifts) of the Neumann problem
			\begin{subequations}
				\label{subeq:TangentLifting}
				\begin{align}
					\label{eq:TangentLiftingLaplace}
					\Delta\,u & = S & & \text{in } \region(c) \\
					\langle n, \nabla u \rangle & = a \circ c^{-1} & & \text{on } \partial \region(c) \\
					\intertext{where $n$ is the outward pointing unit normal vector on $\partial \region(c)$ and}
					\label{eq:TangentLiftingDivergence}
					S & = \left| \region(c) \right|^{-1} \int_{\partial \region(c)} a \circ c^{-1}\,ds \,.
				\end{align}
			\end{subequations}
		\end{definition}
		
		\begin{proof}
			By virtue of \thref{thm:Neumann} the solution $u$ to the PDE is unique (up to constants, which we can fix arbitrarily) and sufficiently smooth, i.e. in $C^\infty(\ol{\region})$. We can then specify any well-designed extension method that maps $C^\infty(\ol{\region})$ to $\testFunction$ to extend $u$.
			Since the extended $u$ is in $\testFunction$ and its Laplacian is constant within $\region(c)$, we find that $\liftTangent_c(a) = \nabla\,u \in \STan\big(\liftPoint(c)\big)$.
		\end{proof}		
		Some examples for the lifting of contour tangent vectors to the shape measure tangent space are illustrated in Fig. \ref{fig:TangentLifting}.
		\begin{remark}
			\label{rem:STanIdentification}
			The extension of the solution $u$ to \eqref{subeq:TangentLifting} is formally necessary such that $\alpha = \nabla u$ is contained in $\STan(F(c))$. For a unique solution $u$ (up to the constant shift) there are many valid extensions, all of them however coincide on $\region(c)$.
			Hence, from now on we will identify functions in $\STan(\mu)$ that coincide on $\spt(\mu)$.
		\end{remark}
		
		With this identification rule, for a fixed $c$ the map $\liftTangent_c$ is a bijection between $H_c \emb$ and $\STan\big(\liftPoint(c)\big)$, with inverse given by
		\begin{align}
			\label{eq:LiftTangentInverseFormula}
			\liftTangent_c^{-1}(\alpha)(\theta) & = \langle \alpha \circ c(\theta), n_c(\theta) \rangle_{\R^2} \text{ for } \theta \in S^1\,,
		\end{align}
		that is taking the normal component on the restriction of $\alpha$ to the boundary.
		
		In analogy to \thref{thm:LiftPointInverse} we then find:
		\begin{proposition}
			\thlabel{thm:LiftTangentInverse}
			The map
			\begin{align}
				{\liftTangentEquiv}_{[c]} : T_{[c]}\sfold \rightarrow \STan\big( \liftPointEquiv({[c]}) \big), \qquad [a] \mapsto \liftTangent_c(a)
			\end{align}
			is a bijection between the tangent space $T_{[c]}\sfold$ on the quotient manifold and the shape measure tangent space $\STan\big(\liftPointEquiv([c])\big)$ at the shape measure obtained by lifting the footpoint contour.
		\end{proposition}
		
		\begin{proof}
			Keep in mind the identification rule in Remark \ref{rem:STanIdentification} and the resulting bijectivity through \eqref{eq:LiftTangentInverseFormula}.
			Further, let $c_1 \sim c_2$ be two contours, related by some $\varphi \in \diff$, i.e. $c_2 = c_1 \circ \varphi$, i.e. $\liftPoint(c_1) = \liftPoint(c_2)$, and let $a_1, a_2$ be two respective normal deformation fields. Then obviously $\liftTangent_{c_1}(a_1) = \liftTangent_{c_2}(a_2)$ if and only if $a_2 = a_1 \circ \varphi$, that is when $(c_1,a_1) \sim (c_2,a_2)$ in the sense of \thref{thm:TangentEquivalence}.
			Then, by the representation property from \thref{thm:TangentEquivalence} the proposition follows.
		\end{proof}		
		\begin{figure}[hbt]
			\centering
			\newlength{\figwidthC}%
			\setlength{\figwidthC}{5.5cm}%
			\begin{tabular}{ccc}
				\includegraphics[width=\figwidthC]{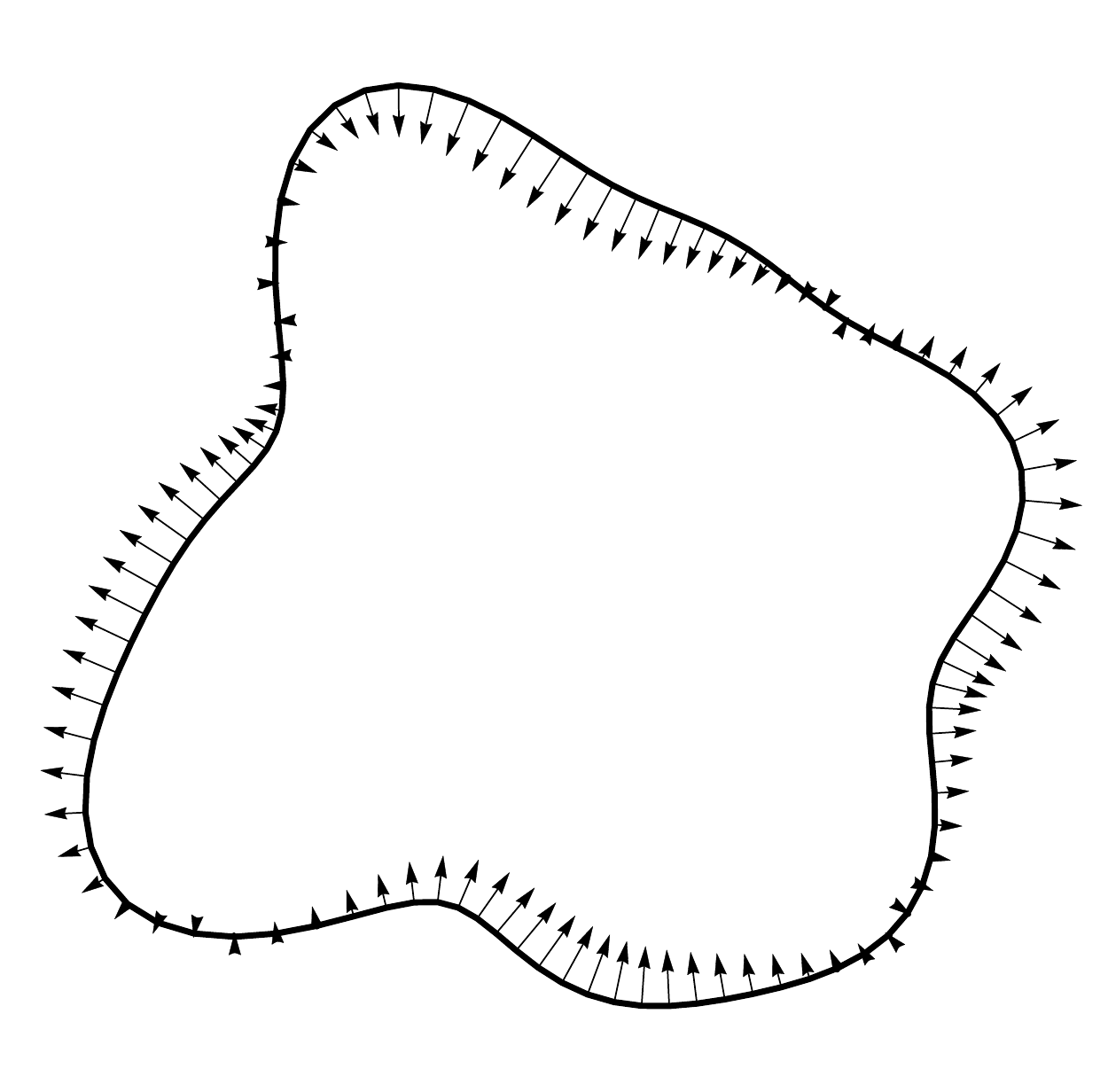} &%
				\hspace{1cm}&%
				\includegraphics[width=\figwidthC]{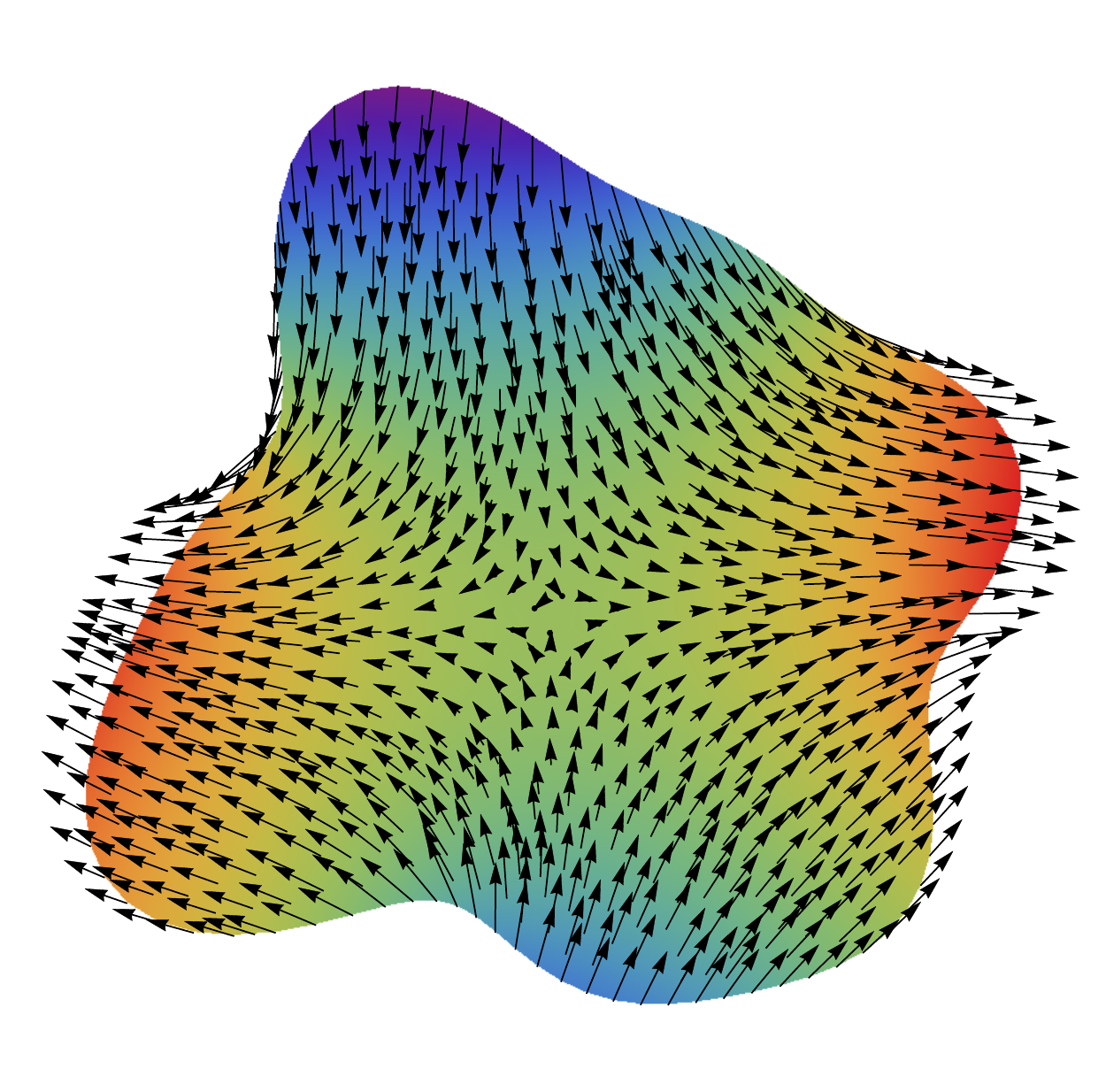} \\%
				\includegraphics[width=\figwidthC]{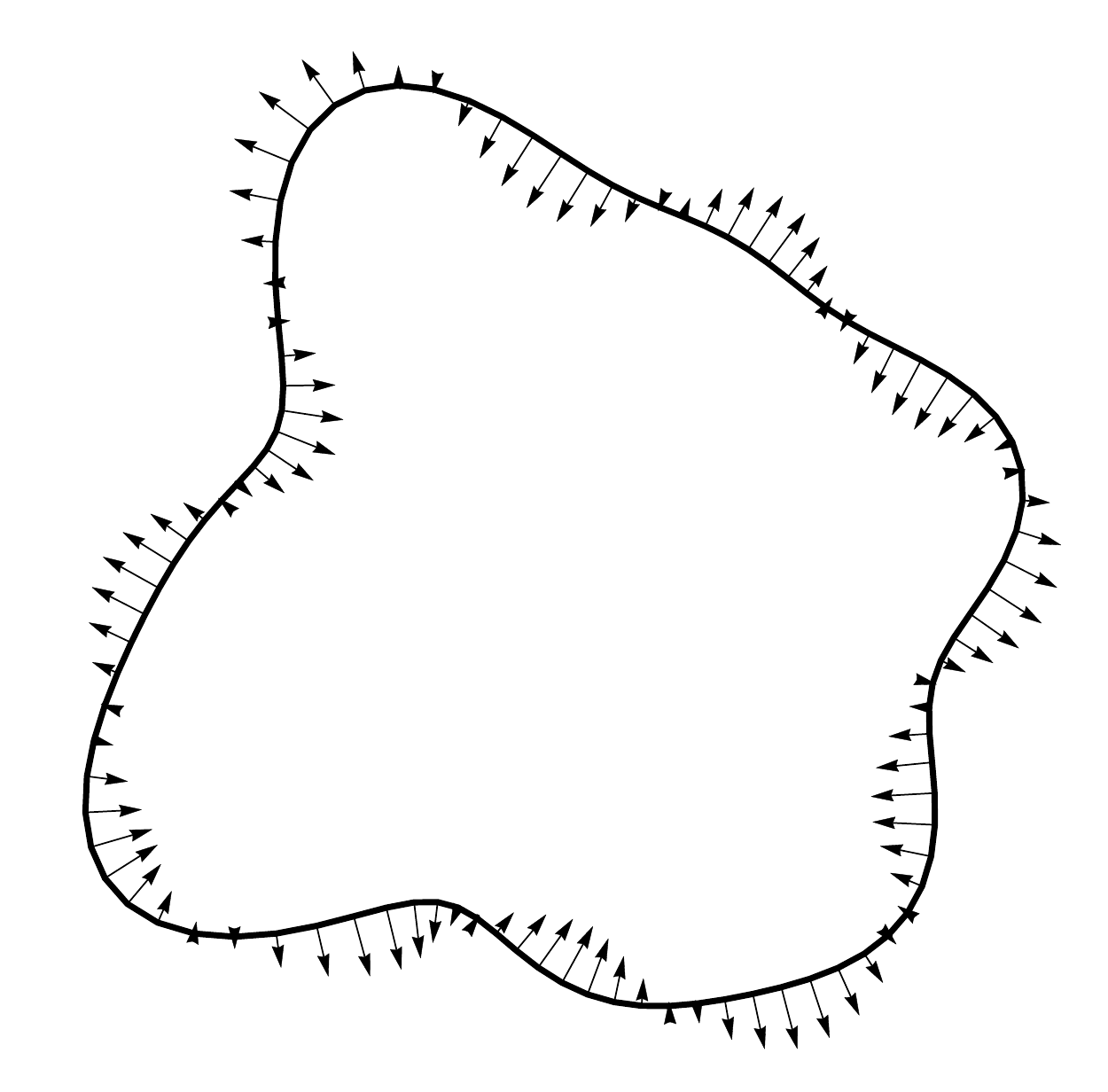} &%
				\hspace{1cm}&%
				\includegraphics[width=\figwidthC]{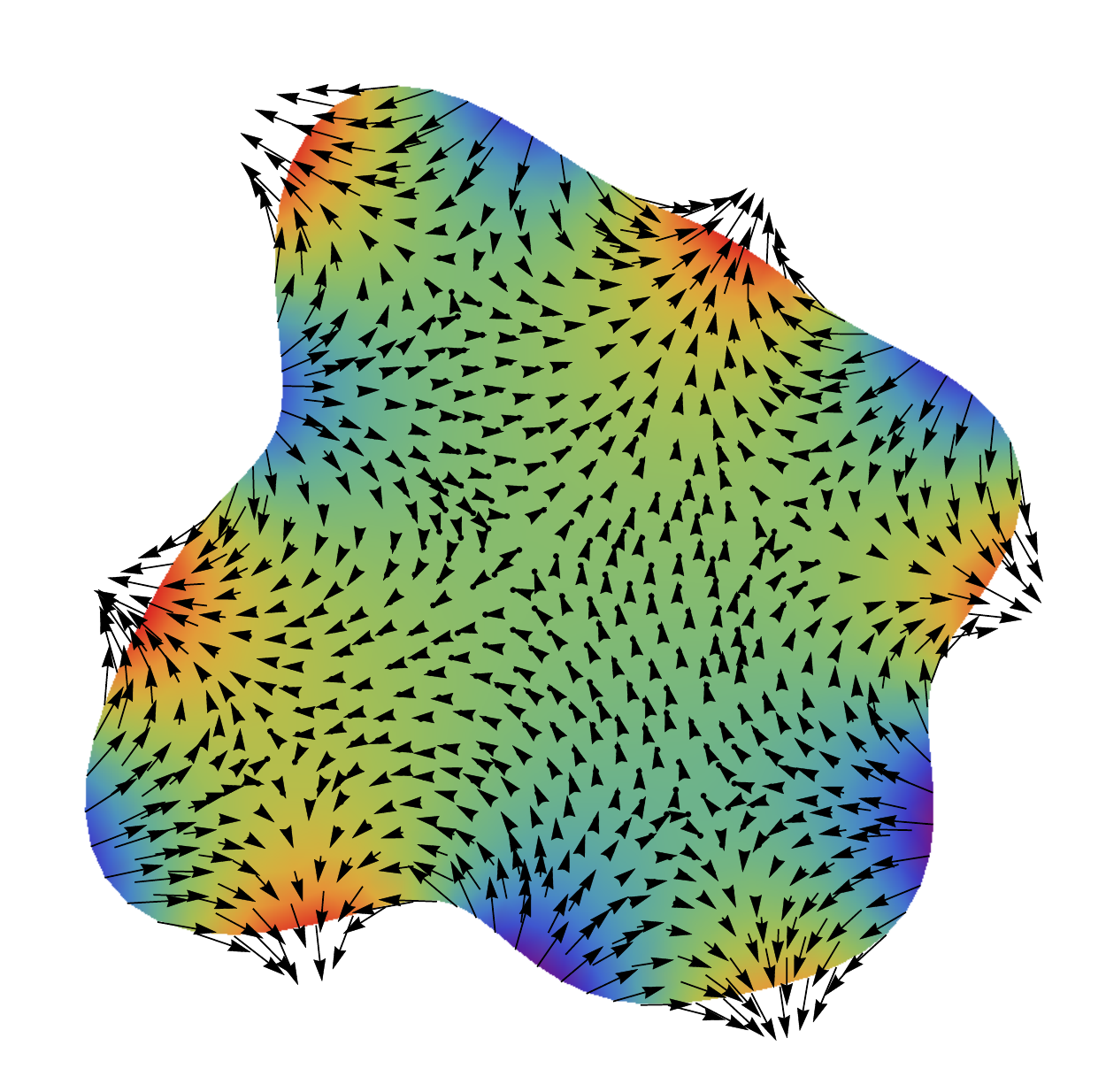}%
			\end{tabular}
			\caption{Lifting of tangential deformation fields. Left column: two different normal deformation fields on a given contour. Right column: lifted flow-fields with constant divergence on the corresponding shape measure. Color shading indicates the potential function that solves the involved Neumann problem.
			In the first example the contour deformation has a low frequency and the lifted flow-field has large amplitudes throughout the interior. In the high-frequency example in the second row, the lifted flow-field has non-vanishing amplitude only near the boundary.
			Note that the lifted flow-field is not normal to the contour at the shape boundary.}
			\label{fig:TangentLifting}			
		\end{figure}
		So far we have established that there is a map $\liftTangent_c$ that takes the (horizontal part of the) tangent space $H_c \emb$ to $\STan\big(\liftPoint(c)\big)$. Let $c_t$ be a path of contours and let $\partial_t c_t$ be the tangent vectors. We need yet to check that $\liftTangent_{c_t}( \partial_t c_t)$ is tangent to $\liftPoint(c_t)$ in the sense of the continuity equation \eqref{eq:ContinuityEq}.
		\begin{theorem}[Commutation of Deformation and Lifting]
			\thlabel{thm:LiftingCommutation}
			Given a contour path $c_t$ which is $C^1$ in time, with normal temporal deformation $a_t$, the following commutation relation holds for all test functions $\phi \in \testFunction$:
			\begin{align}
				\label{eq:LiftingCommutation}
				\frac{d}{dt} \int \phi\,d\liftPoint(c_t) = 
				\int \langle \nabla \phi, \liftTangent_{c_t}(a_t) \rangle \, d\liftPoint(c_t)
			\end{align}
		\end{theorem}

		The implication is that the measure path $F(c_t)$ generated by lifting $c_t$ satisfies the continuity equation \eqref{eq:ContinuityEq} together with the flow-field $f(a_t)$ generated by lifting the tangent path $a_t$.
		In analogy to \eqref{eq:ContinuityEq} we can write for \eqref{eq:LiftingCommutation}:
		\begin{align}
			\frac{d}{dt} \liftPoint(c_t) = - \nabla \left( \liftTangent_{c_t}(a_t)\,\liftPoint(c_t) \right)
		\end{align}
		This corresponds to the following commutation diagram (cf.~Fig.~\ref{fig:LiftingCommutation}):
		\begin{align*}
			\begin{array}{ccc}
				c_t & \xrightarrow{\text{time derivative}} & a_t = \frac{d}{dt} c_t \\
				\downarrow \text{\scriptsize{lift}} & & \downarrow \text{\scriptsize{lift}} \\
				\liftPoint(c_t) & \xrightarrow{\text{time derivative}} & \liftTangent_{c_t}(\frac{d}{dt} c_t) \equiv \frac{d}{dt} \liftPoint(c_t)
			\end{array}
		\end{align*}	
		
		\begin{proof}
			From \thref{thm:ShapeDerivative} we have:
			\begin{align}
				\frac{d}{dt} \int_{\region(c_t)} \phi\,dx = & \int_{\partial \region(c_t)} a_t \circ c_t^{-1}\,\phi\,ds
			\end{align}
			Then one finds:
			\begin{align}
				& \frac{d}{dt} \int \phi\,dF(c_t) \\
				= & \frac{d}{dt} \left| \region(c_t) \right|^{-1} \int_{\region(c_t)} \phi \,dx \\
				= & \left( \frac{d}{dt} \left| \region(c_t) \right|^{-1} \right) \int_{\region(c_t)} \phi(x)\,dx +
					\left| \region(c_t) \right|^{-1} \left( \frac{d}{dt} \int_{\region(c_t)} \phi \,dx \right) \\
				= & - \left| \region(c_t) \right|^{-2} \int_{\partial \region(c_t)} a_t \circ c_t^{-1}\,ds \cdot \int_{\region(c_t)} \phi\,dx
					+ \left| \region(c_t) \right|^{-1} \int_{\partial \region(c_t)} a_t \circ c_t^{-1}\, \phi\,ds \\
				= & - \left| \region(c_t) \right|^{-2} \int_{\partial \region(c_t)} a_t \circ c_t^{-1}\,ds \cdot \int_{\region(c_t)} \phi\,dx
					+ \left| \region(c_t) \right|^{-1} \int_{\region(c_t)} \ddiv \big( \liftTangent_{c_t}(a_t)\,\phi \big)\,dx \\
				\intertext{(in the second term the properties of the lifting $\liftTangent_{c_t}(a_t)$ and the divergence theorem were used)}
				= & - \left| \region(c_t) \right|^{-2} \int_{\partial \region(c_t)} a_t \circ c_t^{-1}\,ds \cdot \int_{\region(c_t)} \phi\,dx \nonumber \\
					& + \left| \region(c_t) \right|^{-1} \int_{\region(c_t)} \underbrace{\big( \ddiv \liftTangent_{c_t}(a_t) \big)}_{= \left| \region(c_t) \right|^{-1} \int_{\partial \region(c_t)} a_t \circ c_t^{-1}\, ds} \,\phi\,dx							
					+ \left| \region(c_t) \right|^{-1} \int_{\region(c_t)} \langle \liftTangent_{c_t}(a_t), \nabla \phi \rangle \,dx \\
				\intertext{($\ddiv \liftTangent(a_t)$ is determined by $\liftTangent(a_t)= \nabla u$, where $u$ is the solution to \eqref{subeq:TangentLifting})}
				= & \int \langle \liftTangent_{c_t}(a_t), \nabla \phi \rangle \,d\liftPoint(c_t) & & \qedhere
			\end{align}	
		\end{proof}
		
		This means, the region deformations encoded in $a_t$ and $\liftTangent_{c_t}(a_t)$ in their respective representations coincide.

	\subsection{Paths in Contour and Shape Measure Description}
		In the assumptions to \thref{thm:ShapeMeasurePaths} we have established a regularity class of paths of shape measures. We will now show that smooth paths on $\emb$ transform into such paths via lifting by $\liftPoint$ and vice versa that `de-lifting' of such paths on $\SMeas$ will result in smooth paths on $\emb$.
		\begin{proposition}[Contours $\rightarrow$ Measures]
			\thlabel{thm:Contours2Measures}
			Let $[0,1] \ni t \mapsto c_t \in \emb$ be a $C^\infty$-family of contours. Then the shape measure path generated by lifting, $\mu_t = \liftPoint(c_t)$ together with $\alpha_t = \liftTangent_{c_t}(\partial_t c_t)$ satisfies the conditions of \thref{thm:ShapeMeasurePaths}.
		\end{proposition}
					
		\begin{proof}
			Conditions (\ref{item:SMP.ShapeMeasure}) and (\ref{item:SMP.AlphaSTan}) follow immediately from $\mu_t = \liftPoint(c_t) \in \SMeas$ and $\alpha_t = \liftTangent_{c_t}(\partial_t c_t)$. Condition (\ref{item:SMP.Continuity}) is implied by \thref{thm:LiftingCommutation}.
			
			Condition (\ref{item:SMP.HoldAll}) is established as follows: let $\region(c_t)$ be the interior of the region enclosed by contour $c_t$. Let further, for any $\region$
			\begin{align}
				d(x,\region) = \inf \{ \| x-y\| \,\colon\, y \in \region \} \\
				\intertext{and}
				\region(c_t,\varepsilon) = \left\{ x \in \R^2 \,\colon\,
					d\big(x,\region(c_t) \big) < \varepsilon \right \} \,.
			\end{align}
			Each set $\region(c_t,\veps)$ is open and bounded.
			The topology on $\emb$ guarantees that for any $t \in [0,1]$ and $\varepsilon > 0$ there is some $\delta > 0$ such that
			\[ c_{t'}(S^1) \subset \region(c_t,\varepsilon) \quad \text{for all} \quad t' \in \tau = [t-\delta, t+\delta] \cap [0,1]\,. \]
			Pick then a set of $(t_i,\varepsilon_i)$ such that the corresponding intervals $\tau_i$ cover $[0,1]$. Since $[0,1]$ is compact, there must be a finite subcovering. We assume $\{(t_i,\veps_i)\}_{i=1}^m$ induces such a covering. Then we find
			\begin{align}
				\spt \mu_t = \ol{\region(c_t)} \subset \hat{\region} = \bigcup_{i=1}^m \region(c_{t_i},\veps_i)
			\end{align}
			for all $t \in [0,1]$ where $\hat{\region}$ is open and bounded.
			
			Let us finally turn to the last condition (\ref{item:SMP.AlphaSup}).
			First extend the normal boundary deformations $\partial_t c_t$ to a flow-field $[0,1] \ni t \mapsto \beta_t \in C^\infty(\hat{\region},\R^2)$, for example as outlined in \cite[Chap.~4,~Sect.~3.3.2]{ZolesioShape2011}. This construction can be designed such that $\beta \in \mc{X}_p(\hat{\region})$ for any integer $p \geq 0$. Let $\varphi_t$ be the family of diffeomorphisms induced by $\beta$, according to \thref{thm:FlowDiffs}.
			Note, that $\varphi_t$ is in general not volume preserving.
			
			The weak solution to Poisson's equation, describing the tangent vector lifting at time $t$ is given by the minimizer w.r.t.~$u \in H^1(\region(c_t))/\R$ of \eqref{eq:NeumannEnergy} with the following parameters
			\begin{align}
				E(u, \region(c_t), f_t ,g_t) & = \frac{1}{2} \int_{\region(c_t)} \| \nabla u\|^2\,dx + \int_{\region(c_t)} f_t\,u\,dx - \int_{\partial \region(c_t)} g_t\,u\,ds \\
				\intertext{with}
				g_t & = \langle \partial_t c_t, n_{c_t} \rangle \quad \text{and} \quad f_t = \frac{1}{|\region(c_t)|} \int_{\partial \region(c_t)} g_t\,ds \,.\\
				\intertext{By means of function space parametrization \cite[Chap.~10,~Sect.~2.2]{ZolesioShape2011} we can express $H^1(\region(c_t))/\R$ in terms of $H^1(\region(c_0))$ and $\varphi_t$:}
				H^1(\region(c_t))/\R & = \left \{ u \circ \varphi_t^{-1} \,\colon\, u \in H^1(\region(c_0))/\R \right\} \\
				\intertext{For some $u \circ \varphi_t^{-1} \in H^1(\region(c_t))/\R$ we then find}
				E(u \circ \varphi_t^{-1}, \region(c_t), f_t ,g_t) & = \frac{1}{2} \int_{\region(c_t)} \| \nabla (u \circ \varphi_t^{-1})\|^2\,dx + \int_{\region(c_t)} f_t\,(u \circ \varphi_t^{-1})\,dx \nonumber \\
				& \qquad - \int_{\partial \region(c_t)} g_t\,(u \circ \varphi_t^{-1})\,ds \\
				& = \frac{1}{2} \int_{\region(c_0)} \langle A_t \nabla u, \nabla u \rangle dx + \int_{\region(c_0)} \tilde{f}_t\,u\,dx - \int_{\partial \region(c_0)} \tilde{g}_t\,u\,ds \\
				& = E(u, \region(c_0), A_t, \tilde{f}_t, \tilde{g}_t)
				\intertext{with}
				A_t & = | \det J_{\varphi_t} | \cdot \left(J_{\varphi_t}^{-1} \right)\,\left(J_{\varphi_t}^{-1} \right)^\T \\
				\intertext{and}
				\tilde{f}_t & = | \det J_{\varphi_t} | \cdot (f \circ \varphi_t) \quad \text{and} \quad
					\tilde{g}_t = | \det J_{{\varphi_t}|_{\partial \region(c_0)}} | \cdot (g \circ \varphi_t)
			\end{align}
			where ${\varphi_t}|_{\partial \region(c_0)}$ denotes the restriction of $\varphi_t$ to the submanifold $\partial \region(c_0)$ and $J_{{\varphi_t}|_{\partial \region(c_0)}}$ is the Jacobian of this restriction.
			
			Since $\varphi_t$ is continuous in $t$ w.r.t.~the topology of uniform convergence in all its derivatives (\thref{thm:UniformDPhiContinuity}) there is a $t>0$ such that the matrix $A_{t'}$ will be positive-definite with bounds $0 < \lambda < \Lambda$ such that $A_{t'}$ satisfies \eqref{eq:LambdaBounds} for all $t' \in [0,t[$. $\tilde{f}_{t'}$ and $\tilde{g}_{t'}$ are always $C^\infty\big(\region(c_0)\big)$ and continuous in time w.r.t.~the supremum norm in any derivative. Also, the map $u \rightarrow u \circ \varphi_{t'}^{-1}$ is continuous w.r.t.~any Sobolev norm between the connected spaces.
			
			Hence, by virtue of the discussion in \thref{rem:MoreRegularity} the minimizers $u_{t'}$ to the functionals $E(\cdot,\region(c_{t'}), f_{t'} ,g_{t'})$ for $t' \in [0,t[$ have a uniformly bounded Sobolev norm $\|u_{t'}\|_{W^{m,p}(\region)/\R}$ for any positive integer $m$ and $1 < p < \infty$. Repeating this construction at different times until the whole interval $[0,1]$ is covered by finitely many `starting points', one can extend the uniform bound to $[0,1]$. By the embedding theorem and by taking the gradient $\alpha_t = \nabla u_t$ it follows then that $\alpha \in \mc{X}_p(\hat{\region})$ for any non-negative integer $p$.
		\end{proof}

		And similarly for the opposite direction:
		\begin{proposition}[Measures $\rightarrow$ Contours]
			\thlabel{thm:Measures2Contours}
			Let $(\mu_t,\alpha_t)$ be a pair of shape measure and flow field paths satisfying the conditions for \thref{thm:ShapeMeasurePaths}.
			Then there is a smooth path $[0,1] \ni t \mapsto c_t \in \emb$ such that
			$\liftPoint(c_t) = \mu_t$.
		\end{proposition}
		
		\begin{proof}
			Let $c_0$ be a contour that parametrizes the boundary of the region given by $\mu_0$.
			Then, as in the proof for \thref{thm:ShapeMeasurePaths} integrate $\alpha_t$ into a family of $C^\infty$-diffeomorphisms $\varphi_t$. As the pushforward of $\mu_0$ under $\varphi_t$ yields $\mu_t$ we can deduce that $c_t = \varphi_t \circ c_0$ parametrizes the boundary of $\mu_t$, i.e. $\liftPoint(c_t) = \mu_t$. Since $\varphi_t$ is a $C^\infty$-diffeomorphism at all times, $c_t$ will be a $C^\infty$-embedding $S^1 \rightarrow \R^2$ at all times, hence it will really be a path in $\emb$.
			Recall from \thref{thm:ConvergenceOnSFold}: since $\emb$ is an open submanifold of the space $C^\infty(S^1,\R^2)$, we show continuity of the path there, continuity in $\emb$ then follows. Convergence in $C^\infty(S^1,\R^2)$ is verified by uniform convergence on $S^1$ in all derivatives separately. 
			
			In analogy to \cite[Lemma 8.3]{YounesShape2010}	it is easy to proof by induction that for any non-negative integer $n$
			\begin{align}
				c^{(n)}_t &  = \partial_\theta^{n} (\varphi_t \circ c_0) \\
					& = \sum_{I : |I|\leq n} \big( (\partial_I \varphi_t) \circ c_0 \big)\,B_{I,n}(c_0)
			\end{align}
			where $B_{I,n}(c_0)$ is a linear combination of terms $(c_0^{(n_1)})_{i_1}\,(c_0^{(n_2)})_{i_2} \ldots (c_0^{(n_q)})_{i_q}$ such that the tuple $I = (i_1,i_2,\ldots,i_q)$ and $n = \sum_{r=1}^q n_r$.
			By virtue of \thref{thm:UniformDPhiContinuity} $\partial_I \varphi_{t_k} \rightarrow \partial_I \varphi_{t}$ uniformly as $t_k \rightarrow t$ for any multi-index $I$. Hence, uniform convergence $c^{(n)}_{t_k} \rightarrow c^{(n)}_k$ is implied.
		\end{proof}
		
	\subsection{Shape Measures as a Manifold}
		We have now established bijections between $\sfold$ and $\SMeas$ and between the deformations $T_{[c]} \sfold$ and $\STan\big(\liftPointEquiv([c])\big)$. Further, we have shown how regular paths in contour and measure descriptions transform into each other. In this section we will formally establish that the set of shape measures $\SMeas$ is a manifold, diffeomorphic to $\sfold$.
		
		Equip $\SMeas$ with the topology induced by the Wasserstein metric $\WD$. Then it is easy to see that $\liftPointEquiv$ is continuous but $\liftPointEquiv^{-1}$ is not. 

		\begin{proposition}[Continuity of $\liftPointEquiv$]
			\thlabel{thm:ContinuityLiftPointEquiv}
			Equip the set of shape measures with the topology induced by the Wasserstein metric $\WD$. Then the map $\liftPointEquiv$ is continuous.
		\end{proposition}		
		
		\begin{proof}
			Let $\{[c_n]\}_n$ be a sequence in $\sfold$ converging to some $[c]$, where by $[\cdot]$ we denote the equivalence class of reparametrizations of a given element of $\sfold$. Hence, there is a sequence of contours $\{c_n\}_n$ and a contour $c$ in $\emb$ with $c_n \in [c_n], c \in [c]$ such that $c_n \rightarrow c$ in $\emb$ (\thref{thm:ConvergenceOnSFold}).
			
			Since $c_n \rightarrow c$ uniformly, there exists for any $\veps > 0$ some $n(\veps) \in \N$ such that for $n > n(\veps)$ the contour $c_n$ in $\R^2$ lies completely within a tube of thickness $\varepsilon$ (both inwards and outwards) around the contour $c$. Anything within the inner boundary of the tube lies within both $\region(c_n)$ and $\region(c)$ and anything beyond the outer boundary is in neither of the two sets. The area of the tube goes to $0$ as $\varepsilon \rightarrow 0$. Therefore also $|\region(c_n)| \rightarrow |\region(c)|$. Hence, we find for any test function $\phi \in \testFunction$:
			\begin{align}
				\int \phi\,d\liftPoint(c_n) \rightarrow \int \phi\,d\liftPoint(c)
			\end{align}
			The measures $\liftPoint(c_n)$ as well as the measure $\liftPoint(c)$ have support limited to the union of $\region(c)$ and the aforementioned tube, which is bounded. Therefore convergence w.r.t.~test functions corresponds to the notion of narrow convergence \cite[Sect.~1.1]{Ambrosio2013} and we can also conclude that the second order moments of $\liftPoint(c_n)$ converge towards the second moments of $\liftPoint(c)$. Hence, by virtue of \cite[Thm.~2.7]{Ambrosio2013} we have $\WD\big(\liftPoint(c_n),\liftPoint(c)\big) \rightarrow 0$.
		\end{proof}
		
		\begin{proposition}
			$\liftPointEquiv^{-1}$ is not continuous.
		\end{proposition}
		
		\begin{proof}
			For sufficiently small $\lambda>0$, consider the sequence of contours $c_n$ and the contour $c$ in $\emb$ given by
			\begin{align}
				c_n(\theta) & = \big(1+ (\lambda/n) \sin(n \cdot \theta) \big)
					\begin{pmatrix}
						\cos(\theta) \\
						\sin(\theta)
					\end{pmatrix}\,, &
				c(\theta) & = 
					\begin{pmatrix}
						\cos(\theta) \\
						\sin(\theta)
					\end{pmatrix}
			\end{align}
			We have $c_n \rightarrow c$ uniformly, but not its derivatives. Hence, as in the reasoning for \thref{thm:ContinuityLiftPointEquiv}, we can conclude that $\liftPoint(c_n) \rightarrow \liftPoint(c)$ in the Wasserstein topology, but we have $c_n \not \rightarrow c$ and also $[c_n] \not \rightarrow [c]$ on $\sfold$. Hence, $\liftPointEquiv^{-1}$ is not continuous.
		\end{proof}
		
		We see from this example, that convergence in the optimal transport sense is only concerned with convergence of the regions towards each other, regardless of the boundary or even higher order regularity as required on $\emb$. For this reason additional assumptions on the regularity of $\alpha$ were necessary in \thref{thm:ShapeMeasurePaths} to be able to transform paths back and forth between contour and measure description.
		
		However, if we equip $\SMeas$ with the topology induced by $\liftPointEquiv$, then by definition $\liftPointEquiv$ and also $\liftPointEquiv^{-1}$ are continuous, thus constituting a homeomorphism between the two sets. Then $\SMeas$ inherits the manifold property from $\sfold$. 
		Let $\psi_i, \psi_j$ be any two charts mapping overlapping open sets $U_i, U_j \subset \sfold$ into the modelling space. Then
		$\psi_i \circ \liftPointEquiv^{-1}$ will be a chart on $\SMeas$. The corresponding chart change $(\psi_i \circ \liftPointEquiv^{-1}) \circ (\psi_j \circ \liftPointEquiv^{-1})^{-1}$ remains differentiable as the lifting onto $\SMeas$ cancels.
		Likewise, the map
		\[(\psi_j \circ \liftPointEquiv) \circ \liftPointEquiv^{-1} \circ \psi_i^{-1} = \id\,,\]
		that takes the modelling space of a chart on $\sfold$ onto the modelling space of a corresponding chart on $\SMeas$ is trivial and thus differentiable. Hence the two manifolds are actually diffeomorphic.

		By virtue of \thref{thm:LiftTangentInverse} we can represent the tangent space on $\SMeas$ at $\mu$ by $\STan(\mu)$. And due to \thref{thm:LiftingCommutation} we have that such tangent vectors naturally represent directional derivatives of functions on $\SMeas$ that are given by region integrals over test functions. Evaluation is given by the continuity equation \eqref{eq:ContinuityEq}.
		
		The diffeomorphism between the contour manifold $\sfold$ and the set of shape measures establishes that the shape measure representation is a formally equivalent way of describing shapes.
		The tangent space $\STan(\mu)$ gives a linear structure to describe deformations that is just as powerful as $T_{[c]} \sfold$ in terms of shape analysis and modelling.
		In addition, every shape is uniquely represented in the shape measure description, whereas one has to handle parametrization ambiguities in the contour representation.
		Shape measures are therefore an elegant way for shape representation in image segmentation tasks.

%% file: contour-manifolds-v6-Section-4-SpaceProperties.tex
\section{A Riemannian Metric on the Manifold of Shape Measures}	
	\label{sec:RiemannianMetric}
	\subsection{Metric Structure of the Tangent Space}
		\label{sec:MetricTangent}
		In Sect.~\ref{sec:BackgroundWasserstein} we have discussed the analogy of $\Meas$, metrized by $\WD$, to a Riemannian manifold with metric tensor \eqref{eq:WassersteinRiemannInnerProduct}.
		In the last Section we have formally established, that the set of shape measures $\SMeas$ can be viewed as an infinite dimensional manifold, diffeomorphic to $\sfold$, albeit with a topology which is not compatible with the metric topology induced by $\WD$.
		Nevertheless, since $\SMeas \subset \Meas$ and since the tangent space w.r.t.~$\SMeas$, Definition \ref{def:ShapeTangentSpace}, is a subset of the tangent space w.r.t.~$\Meas$, Definition \ref{def:TangentSpace}, $\STan(\mu) \subset \Tan(\mu)$ for $\mu \in \SMeas$, it suggests itself to informally view the shape measures as a submanifold of all measures and to equip $\SMeas$ with the Riemannian metric that is induced by `restricting' the metric tensor on $\Meas$ to the `submanifold'.
		A prominent example of how such treatments can yield valuable insights, is the Otto calculus and its success in the context of interpreting partial differential equations as gradient flows (see for example \cite[Chap.~15]{Villani-OptimalTransport-09}).

		This will yield a new type of metric on the contour manifold $\sfold$, as opposed to contour-based metrics, for example discussed in \cite{Michor2006,MichorHamiltonianCurves,sunmensoa10}.
		Formally one can find an expression for the new metric inner product by pull-back through $\liftPoint$. One would find a non-local integral involving the kernel for the PDE \eqref{subeq:TangentLifting}. In this article we study the new metric directly in the measure representation where the inner product is local.
		
		We start by analyzing the metric structure on $\STan(\mu)$.
		First note that the equivalence classes of tangent vectors, induced by the pseudo-metric \eqref{eq:WassersteinRiemannInnerProduct} (two vectors being equivalent if they have zero distance), are just those described in Remark \ref{rem:STanIdentification}.
		
		We now consider various subspaces of $\STan(\mu)$.
		\paragraph{Translation.}
			Let $\mu$ be some shape measure and $\alpha = v$ be a flow field that is constant in space for some $v \in \R^2$.
			Such fields span a two dimensional subspace of $\STan(\mu)$.
			Then $\mu_t = (\id + t \cdot \alpha)_\sharp \mu$ is for every $t$ just the translation of $\mu$ by the vector $t\cdot v$. 
			One finds for any test function $\phi \in \testFunction$
			\begin{align}
				\frac{d}{dt} \left. \int \phi\,d\mu_t \right|_{t=0} & = \frac{d}{dt} \left. \int \phi\,d(\id + t\cdot \alpha)_\sharp \mu \right|_{t=0} \nonumber \\
					& = \frac{d}{dt} \left. \int \phi \big(x + t \cdot \alpha(x) \big)\,d\mu(x) \right|_{t=0} = \int \langle \nabla \phi, \alpha \rangle d\mu\,,
			\end{align}
			i.e. $(\mu_t,\alpha_t=\alpha)$ satisfy the continuity equation.

			Strictly, for $\alpha_t = \alpha$ to be within $\STan(\mu_t)$ at any time, we need to smoothly truncate it, such that its support is compact. We will assume that such a truncation has been applied, but at such a large radius that at all times $t \in [0,1]$ we have $\alpha = v$ on the support of $\mu_t$.
			
			Let now $\alpha \in \STan(\mu)$ be a flow-field that is orthogonal to any translation flow field w.r.t.~the Riemannian inner product. That is
			\begin{align*}
				0 & = \int \langle \alpha(x), v \rangle_{\R^2} d\mu \qquad \text{for all } v \in \R^2\,. \\
				\intertext{We then find}
				0 & = \left \langle \int \alpha(x)\,d\mu, v \right \rangle_{\R^2}\, \qquad \text{for all $v \in \R^2$, and thus} \qquad 0 = \int \alpha(x)\,d\mu\,.
				\intertext{From this follows after a brief calculation}
				0 & = \frac{d}{dt} \left. \int x\,d(\id + t\,\alpha)_\sharp \mu(x) \right|_{t=0}
			\end{align*}
			where we need to smoothly truncate the function $x \mapsto x$ beyond the support of $\mu$ to turn it into a test function.
			We then see that any tangent vector that is locally orthogonal to any translation field keeps the center of mass of $\mu$ unchanged.
			
		\paragraph{Scale.}
			Assume now, we fix some tangent vector $\alpha_{\textnormal{scale}} \in \STan(\mu)$ with $\ddiv \alpha_{\textnormal{scale}} = 1$ in $\spt \mu$, that is orthogonal to the translation fields. We refer to $\alpha_{\textnormal{scale}}$ as \emph{scale component}. Then we can uniquely decompose any given tangent vector $\alpha$ into the following components:
			\begin{align}
				\label{eq:STanDecomposition}
				\alpha = \alpha_{\textnormal{trans}} + \lambda \cdot \alpha_{\textnormal{scale}} + \alpha_{\textnormal{def}}
			\end{align}
			where $\alpha_{\textnormal{trans}}$ is a \emph{translation component} as discussed above, $\lambda$ is given by $\ddiv \alpha$ and $\alpha_{\textnormal{def}}$ is a divergence-free residual, orthogonal to the translation component, we will refer to as \emph{deformation component}.
			We now discuss, how a scale component can be determined which is orthogonal to all divergence-free flow fields, this includes the translation component and the residual $\alpha_{\textnormal{def}}$, such that the decomposition \eqref{eq:STanDecomposition} is an orthogonal one.
			Demand for any $\alpha$ with $\ddiv \alpha=0$ that
			\begin{align}
				\label{eq:ScaleComponentAssumption}
				0 & = \int \langle \alpha_{\textnormal{scale}}, \alpha \rangle d\mu = \int_{\region} \langle \alpha_{\textnormal{scale}}, \alpha \rangle dx
				\intertext{where $\region = \spt(\mu)$ and we can neglect the normalization factor $|\region|^{-1}$, since the integral vanishes. We then take $\alpha_{\textnormal{scale}} = \nabla u_{\textnormal{scale}}$, $u_{\textnormal{scale}} \in \testFunction$, and find}
				\nonumber
				0 & = \int_\region \langle \nabla u_{\textnormal{scale}}, \alpha \rangle dx = \int_\region \nabla (u_{\textnormal{scale}} \cdot \alpha) dx \\
				\intertext{where the second equality holds since $\ddiv \alpha = 0$. Then by the divergence theorem}
				\label{eq:ScaleComponentCondition}
				0 & = \int_{\partial \region} u_{\textnormal{scale}} \langle n,\alpha \rangle_{\R^2} ds
			\end{align}
			where $n$ is the outward pointing unit-normal on $\partial \region$.
			
			If $u_{\textnormal{scale}}$ were non-constant on $\partial \region$, one could locate a region where $u_{\textnormal{scale}}$ is above average (w.r.t.~the boundary length as weight) and one, where $u_{\textnormal{scale}}$ is below average. One could then choose some smooth normal components for a field $\alpha$, $\langle n,\alpha \rangle_{\R^2}$, that have some influx in the above-average region and a corresponding outward flux in the below average region with zero net flux. For this normal component \eqref{eq:ScaleComponentCondition} would be non-zero. This normal component could then be lifted to a complete divergence-free flow-field $\alpha$ by virtue of Definition \ref{def:TangentLifting}, yielding a contradiction to assumption \eqref{eq:ScaleComponentAssumption}. Thus we can conclude that $u_{\textnormal{scale}}$ must be constant on $\partial \region$. We choose to set $u_{\textnormal{scale}} = 0$ on $\partial \region$.
			
			To obtain a valid $u_{\textnormal{scale}}$ throughout $\region$, one can solve the following Dirichlet problem:
			\begin{subequations}
				\label{eq:ScaleComponentDirichlet}
				\begin{align}
					\Delta\,u_{\textnormal{scale}} & = 1 & & \text{in } \region \\
					u_{\textnormal{scale}} & = 0 & & \text{on } \partial \region
				\end{align}
			\end{subequations}
			In analogy to \thref{thm:Neumann} this problem has a unique solution in $C^\infty(\ol{\region})$ and thus (after extension onto $\testFunction$) induces a unique scale component $\alpha_{\textnormal{scale}} = \nabla u_{\textnormal{scale}}$ which is orthogonal to all divergence-free modes. The effect on shapes when moving along the scale-component on the manifold of shape measures is illustrated in Fig. \ref{fig:ScaleMode}.
			
			\begin{figure}[hbt]
				\centering
				\includegraphics[width=7cm]{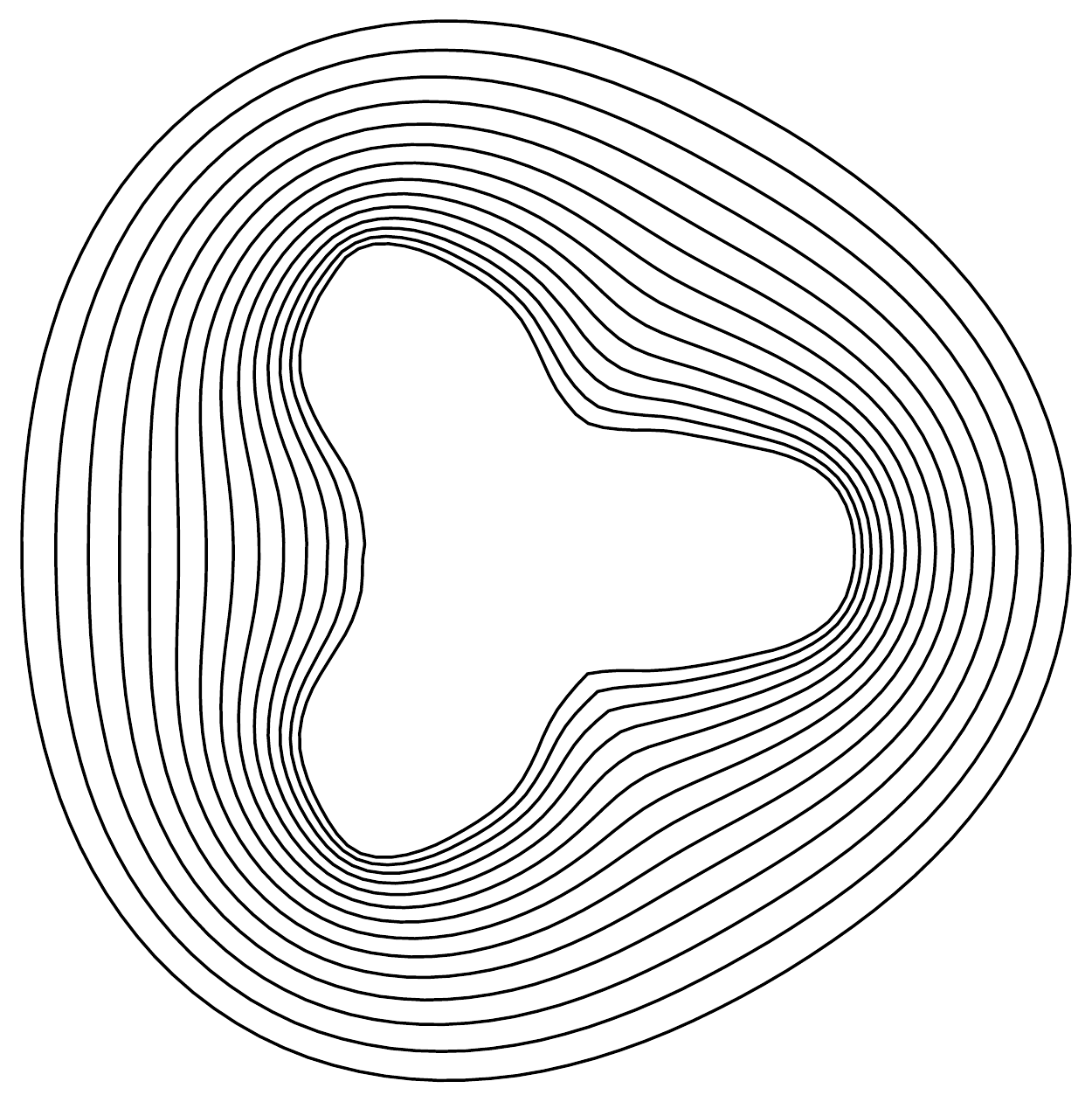}
				\caption{Moving along the scale-component: the contours correspond to different shape measures along a path in $\SMeas$ that is locally tangent to the scale-component (Sect.~\ref{sec:MetricTangent}). As one moves to larger scales, small details are increasingly smoothened. Towards smaller scales the become more emphasized.}
				\label{fig:ScaleMode}
			\end{figure}
	
	\subsection{Geodesic Equation on Shape Measures}
		We have recalled in Sect.~\ref{sec:BackgroundWassersteinManifold} some results from \cite{LottWassersteinRiemannian2008} about the manifold $\Meas^\infty$ of absolutely continuous measures with smooth density functions. Points in $\SMeas$ and $\Meas$ can be approximated to arbitrary precision by points in $\Meas^\infty$ as measured by $\WD$.
		Thus, encouraged also by Remark \ref{rem:GeodesicGenerality}, in this section we will pretend, expressions \eqref{eq:WassersteinCovariantDerivative} and \eqref{eq:WassersteinCovariantDerivativeDT} were also valid for sufficiently smooth tangent vectors on $\Meas$.
		From these we will want to find equivalent expressions on $\SMeas$. \emph{We emphasize that this is not a rigorously justified analysis.}
		It is yet a worthwhile excursion as one might gain some intuition on the new metric structure of the space of shape measures, which, as shown, is a new metric on the contour manifold $\sfold$.
		
		\paragraph{Geodesics on $\Meas$.}		
			Let us first have a look at geodesics on $\Meas$. For an initial measure $\mu \in \Meas$ and an initial tangent vector $\alpha \in \Tan(\mu)$, the solution to the geodesic equation for regular optimal transport, \eqref{eq:WassersteinGeodesicEquation}, is given by
			\begin{align}
				\label{eq:TangentTrajectory}
				\mu_t & = (\id + t \cdot \alpha)_\sharp \mu\,.
			\end{align}
			That is, every infinitesimal `mass particle' in $\mu_t$ moves along a straight line, direction and velocity determined by $\alpha$ at $t=0$. Once this flow-field has been chosen, no interaction between `mass particles' is necessary, which is why the corresponding geodesic equation \eqref{eq:WassersteinGeodesicEquation} is local in $u_t$.
		
			Let now $\mu \in \SMeas$ and $\alpha \in \STan(\mu)$. Then from the discussion around \eqref{eq:FirstOrderJacobian} we know that to first order $\mu_t$ as in \eqref{eq:TangentTrajectory} has homogeneous density within its support.
			However, let us check the geodesic equation \eqref{eq:WassersteinGeodesicEquation} for the potential function $u_t$ of $\alpha_t$, where $\alpha_t$ is the temporal evolution of $\alpha_0 = \alpha$ along the geodesic. Applying the Laplacian to both sides (assuming for now sufficient regularity), we find at $t=0$
			\begin{align}
				\label{eq:GeodesicHessian}
				\partial_t \left. \Delta u_t \right|_{t=0}= \left. -\frac{1}{2} \Delta \|\nabla u_t\|^2 \right|_{t=0} = \sum_{i,j=1}^2 (\partial_i \partial_j u)^2
			\end{align}
			which is the Frobenius norm of the Hessian of $u$. So, if the Hessian is not spatially constant, we find that $\mu_t$ will leave the subset $\SMeas$. Hence, for geodesics on $\SMeas$, `mass particles' will not always be allowed to simply move along straight lines. They will need to make sure, that their joint density remains spatially constant. Hence, there is need for another equation of evolution, which we will now informally try to motivate.
				
		\paragraph{Projection.}
			Recall the following result from differential geometry in finite dimensions: Let $M,N$ be Riemannian manifolds, let $N$ be a submanifold of $M$. Let $x \in N \subset M$ and let $a \in T_x N \subset T_x M$ and let $b$ be a vector field on $M$ with $b(x') \in T_{x'} N$ for all $x' \in N$. Then $b$ can be turned into a vector field on $N$ by restriction. Denote by $\nabla_M\big(b,(x,a)\big)$ the covariant derivative of $b$ at point $x$ w.r.t.~direction $a$, and likewise for other parameters. Then
			\begin{equation}
				\label{eq:CovariantDerivativeProjection}
				\nabla_N\big(b,(x,a)\big) = \Proj_{T_xN} \Big( \nabla_M\big(b,(x,a)\big) \Big)
			\end{equation}
			where projection is w.r.t.~the Riemannian inner product.
			
			Next, let us find the projection map $\Proj_{\STan(\mu)}$.
			For a given shape measure $\mu \in \SMeas$, let $u \in \testFunction$, so $\nabla u \in \Tan(\mu)$. Our goal is now to find $\hat{u} \in \testFunction$ such that $\nabla \hat{u} = \Proj_{\STan(\mu)}(\nabla u)$. In that case $\nabla (u - \hat{u})$ is orthogonal to any vector in $\STan(\mu)$. Let $u_{\perp}$ be the unique solution to the following Dirichlet problem:
			\begin{subequations}
				\label{eq:ProjectionComponentDirichlet}
				\begin{align}
					\Delta\,u_{\perp} & = \Delta u & & \text{in } \region \\
					u_\perp & = 0 & & \text{on } \partial \region
				\end{align}
			\end{subequations}
			Again, we find $u_\perp \in C^\infty(\ol{\region})$ and can suitably extend to $\testFunction$. Recall the discussion on the scale component in Sect.~\ref{sec:MetricTangent} to find that $\nabla u_\perp$ is perpendicular to any divergence-free vector in $\STan(\mu)$ w.r.t.~the inner product \eqref{eq:WassersteinRiemannInnerProduct}.
			Further, the vector $\nabla (u - u_\perp)$ lies in $\STan(\mu)$. Thus, all that remains to be done is, to orthogonalize w.r.t.~the scale component $\nabla u_{\textnormal{scale}}$ as introduced in Sect.~\ref{sec:MetricTangent}, which spans the only direction of $\STan(\mu)$ which has non-zero divergence.
			Thus, begin with the ansatz
			\begin{align}
				\hat{u} & = u - u_\perp + \lambda \cdot u_{\textnormal{scale}}
			\end{align}
			and determine $\lambda$ such that $\nabla(u - \hat{u}) \perp \nabla u_{\textnormal{scale}}$ w.r.t.~\eqref{eq:WassersteinRiemannInnerProduct}.
	
		\paragraph{Geodesic Equation.}
			Now we put together the pieces: Combining \eqref{eq:WassersteinCovariantDerivativePotential} and \eqref{eq:CovariantDerivativeProjection} to obtain the covariant derivative of $u_t$ in $\SMeas$ along itself, and setting this to zero, we find:
			\begin{align}
				0 & = \Proj_{\STan(\mu_t)} \left( \nabla \left( \frac{1}{2} \|\nabla u_t\|^2 + \partial_t u_t \right) \right) \\
				\intertext{Since the projection is linear, we can separately apply it to the $\|\nabla u_t\|^2$ and to the $\partial_t u_t$ terms. Further, since $\nabla u_t \in \STan(\mu_t)$, we have that $\nabla \partial_t u_t \in \STan(\mu_t)$ since also the divergence of the temporal derivative must be spatially constant. Hence, the projection of the second term is redundant and we can write:}
				\label{eq:GeodesicEquation}
				0 & = \Proj_{\STan(\mu_t)} \left( \nabla \frac{1}{2} \|\nabla u_t\|^2 \right) + \nabla \partial_t u_t
			\end{align}
			Since $\Proj_{\STan(\mu)}$ is a non-local operation, the new geodesic equation is non-local, in contrast to \eqref{eq:WassersteinGeodesicEquation}. This non-locality is necessary to keep the density of $\mu_t$ spatially constant along the path.
			
	\subsection{Geodesics on the Manifold of Shape Measures}
		\label{sec:Geodesics}
		We now discuss some particular solutions to \eqref{eq:GeodesicEquation}.
		Let $\mu \in \SMeas$ be some shape measure and the initial tangent vector $\alpha_{0} = \alpha_{\textnormal{trans}} = v$ be a spatially constant translation mode, as discussed in Sect.~\ref{sec:MetricTangent}.
		The geodesic in $\Meas$ through $\mu$, tangent to $\alpha_0$ is given by $\mu_t = (\id + t\cdot v)_\sharp \mu$. This is the translation of $\mu$ by the vector $t \cdot v$. Obviously this is a path in $\SMeas$. Since $\SMeas \subset \Meas$, it must therefore also be a geodesic in $\SMeas$.
		
		This can be verified explicitly: we have  $\alpha_{0} = \nabla u_{0}$ for $u_{0} = \langle x, v \rangle_{\R^2}$ and consequently find $\nabla \|\nabla u_{0}\|^2 = 0 \in \STan(\mu)$. Hence, the projection will change nothing and we find $\partial_t \alpha_{t}|_{t=0} = 0$. One can thus see that $\alpha_t = \alpha_0$ is in fact a solution to \eqref{eq:GeodesicEquation}.
		
		Consider further the initial tangent vector $\alpha_0(x) = x$. This corresponds to resizing the original shape. A possible potential function is given by $u(x) = \|x\|^2/2$. One can check that the induced optimal transport geodesic $\mu_t = (\id + t \cdot \alpha_0)_\sharp \mu$ lies within $\SMeas$, hence by the same reasoning as with the translations, it must therefore also be a geodesic on the shape measures.

		A numerical solution to the geodesic equation where the projection is important is illustrated in Fig.~\ref{fig:Geodesic}.
				
		\begin{figure}[bt]
			\centering
			\newlength{\figwidth}%
			\setlength{\figwidth}{2.5cm}%
			\setlength{\figwidthB}{0.8\figwidth}%
			\begin{tikzpicture}[img/.style={anchor=center, inner sep=0}]
				\node[img] at (0,0) {\includegraphics[width=\figwidth]{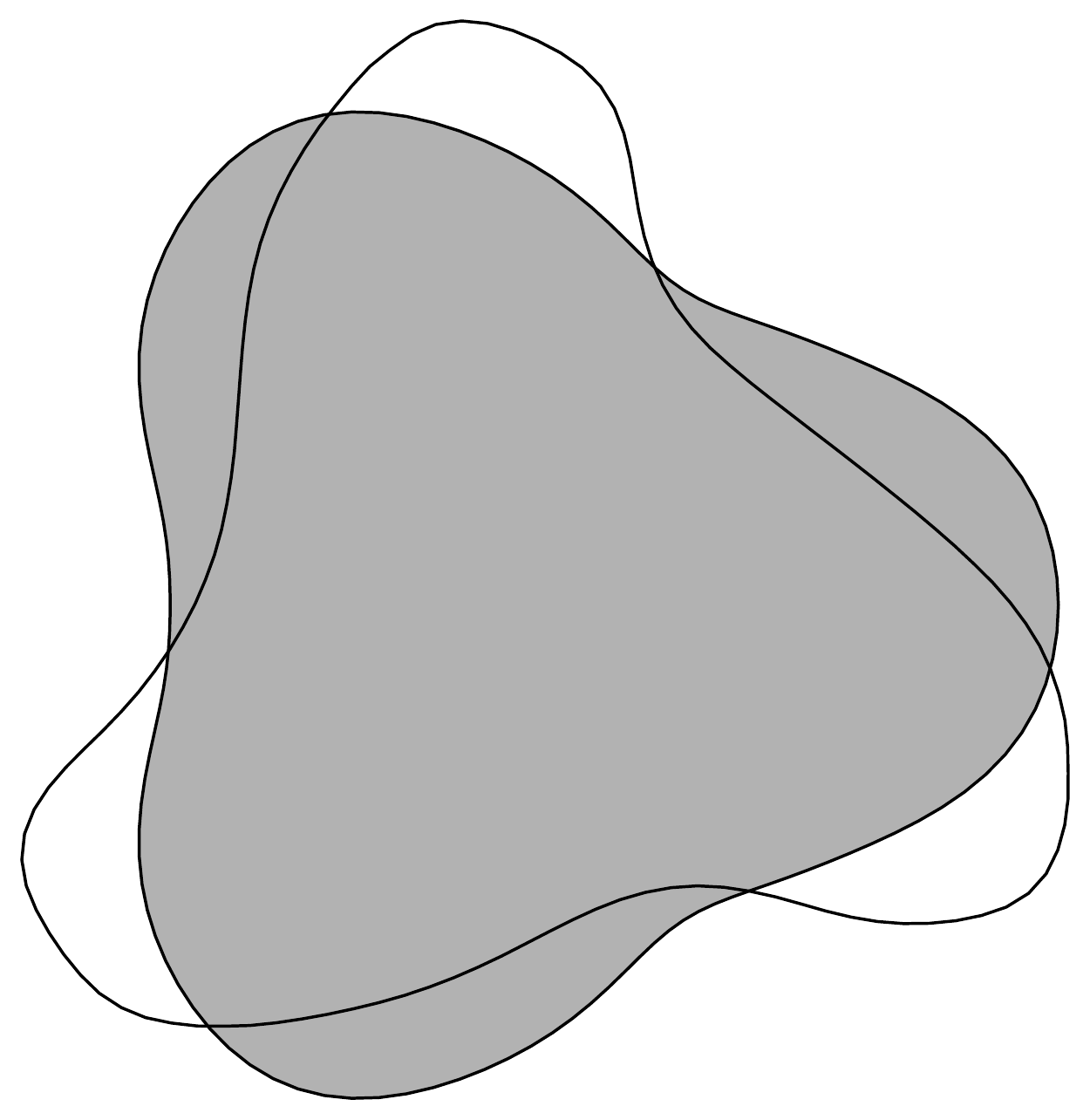}};
				\node[img] at (1.2\figwidth,0) {\includegraphics[width=\figwidth]{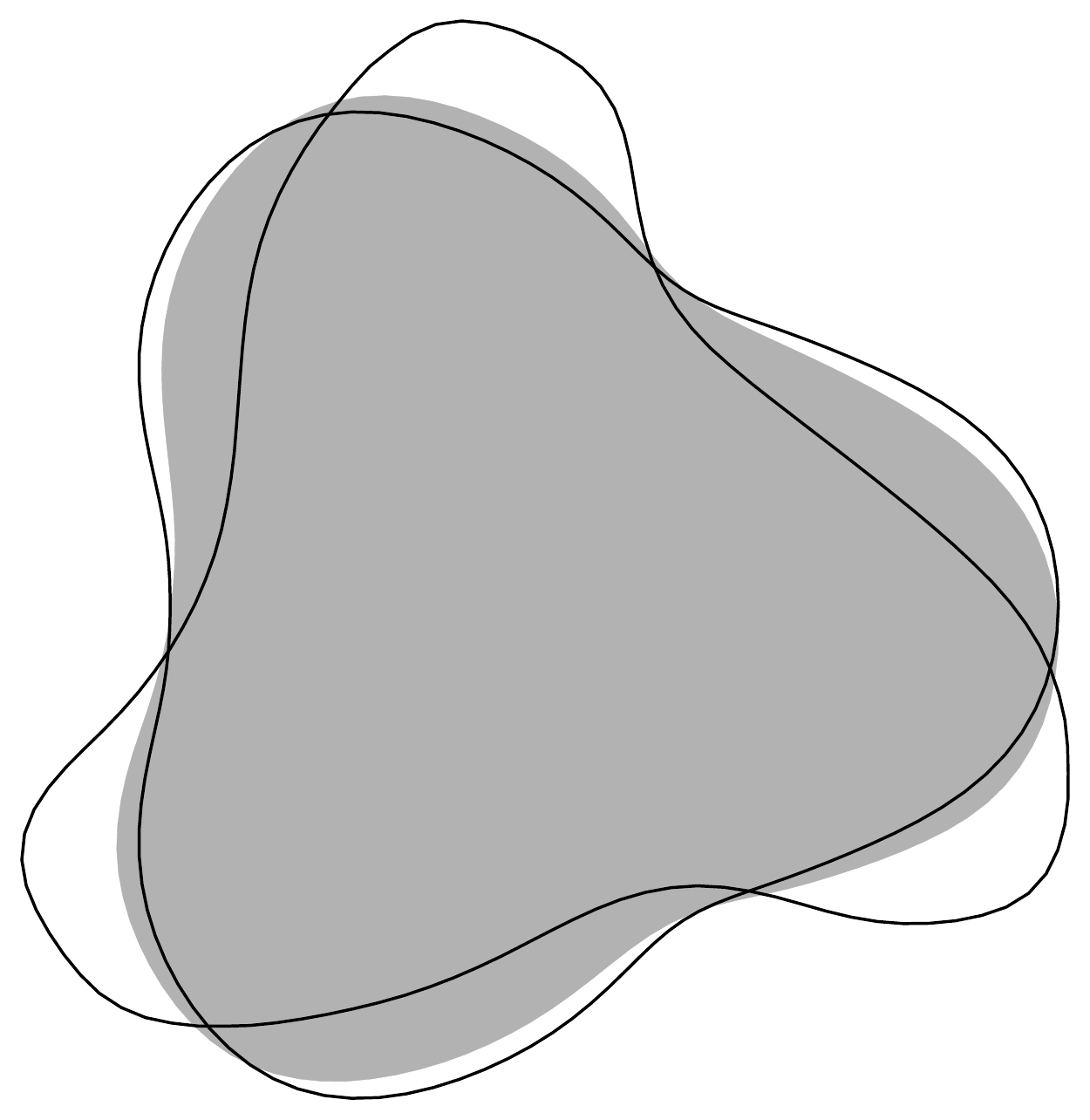}};
				\node[img] at (2.4\figwidth,0) {\includegraphics[width=\figwidth]{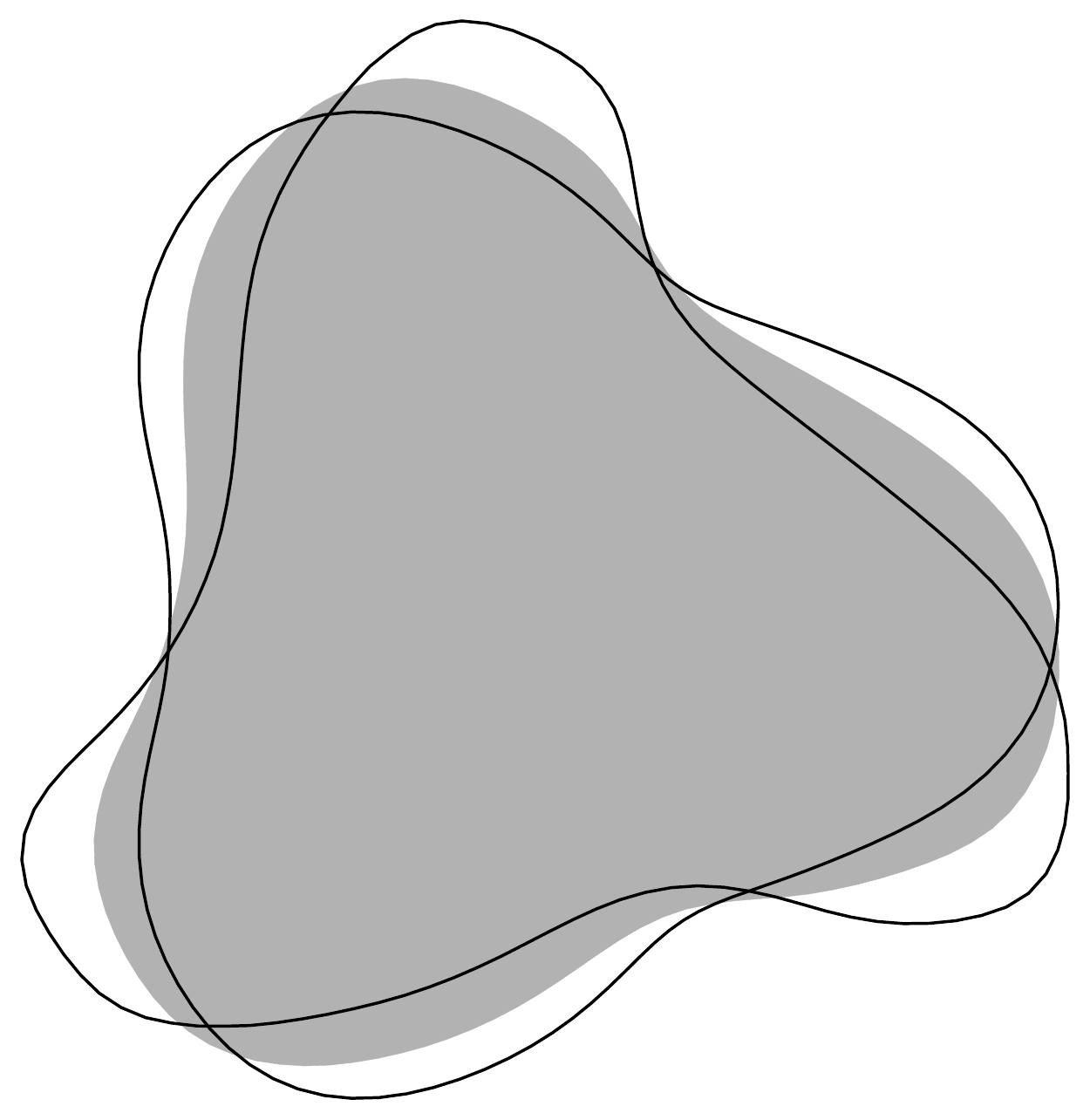}};
				\node[img] at (0,-1.2\figwidth) {\includegraphics[width=\figwidth]{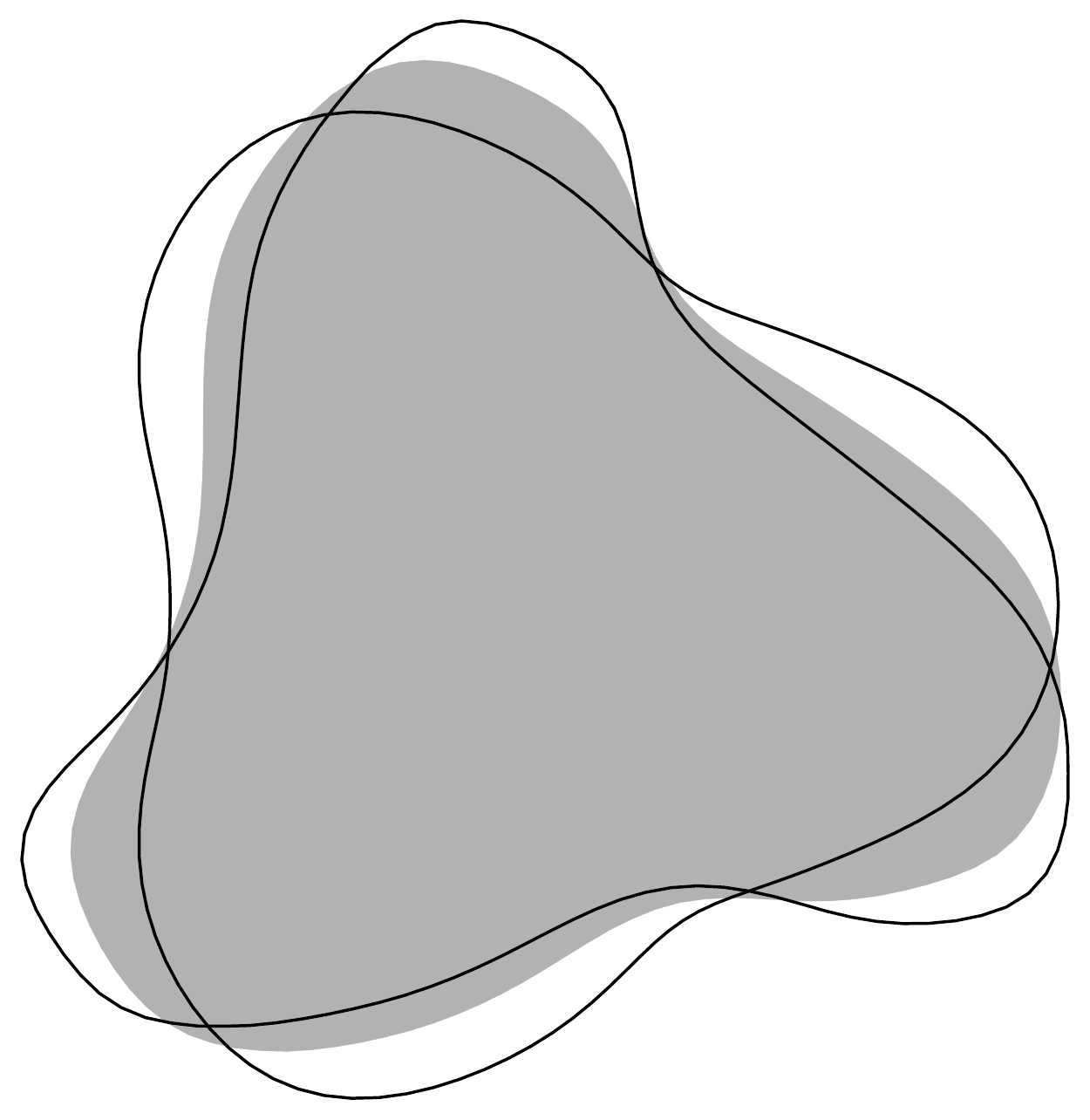}};
				\node[img] at (1.2\figwidth,-1.2\figwidth) {\includegraphics[width=\figwidth]{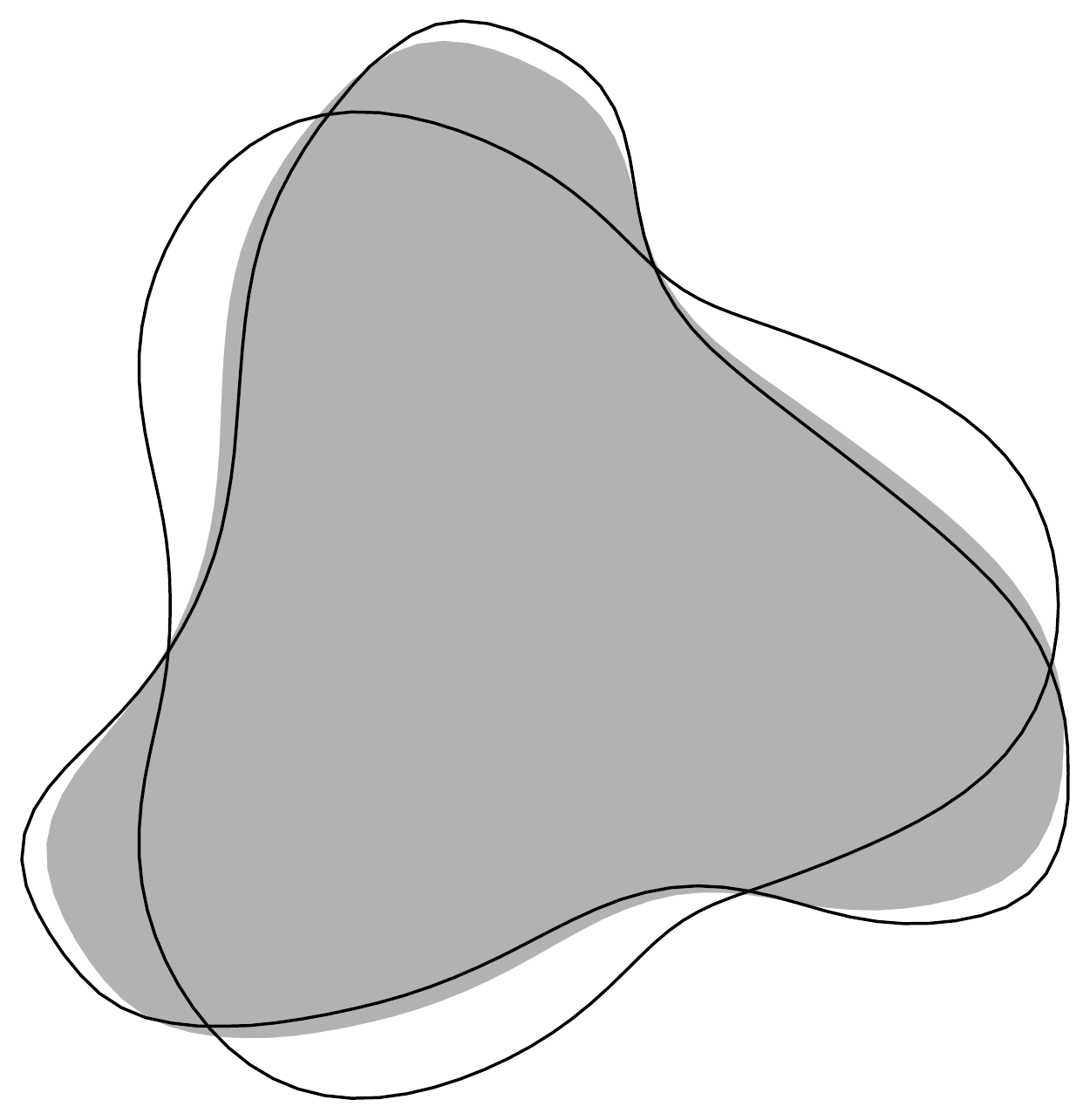}};
				\node[img] at (2.4\figwidth,-1.2\figwidth) {\includegraphics[width=\figwidth]{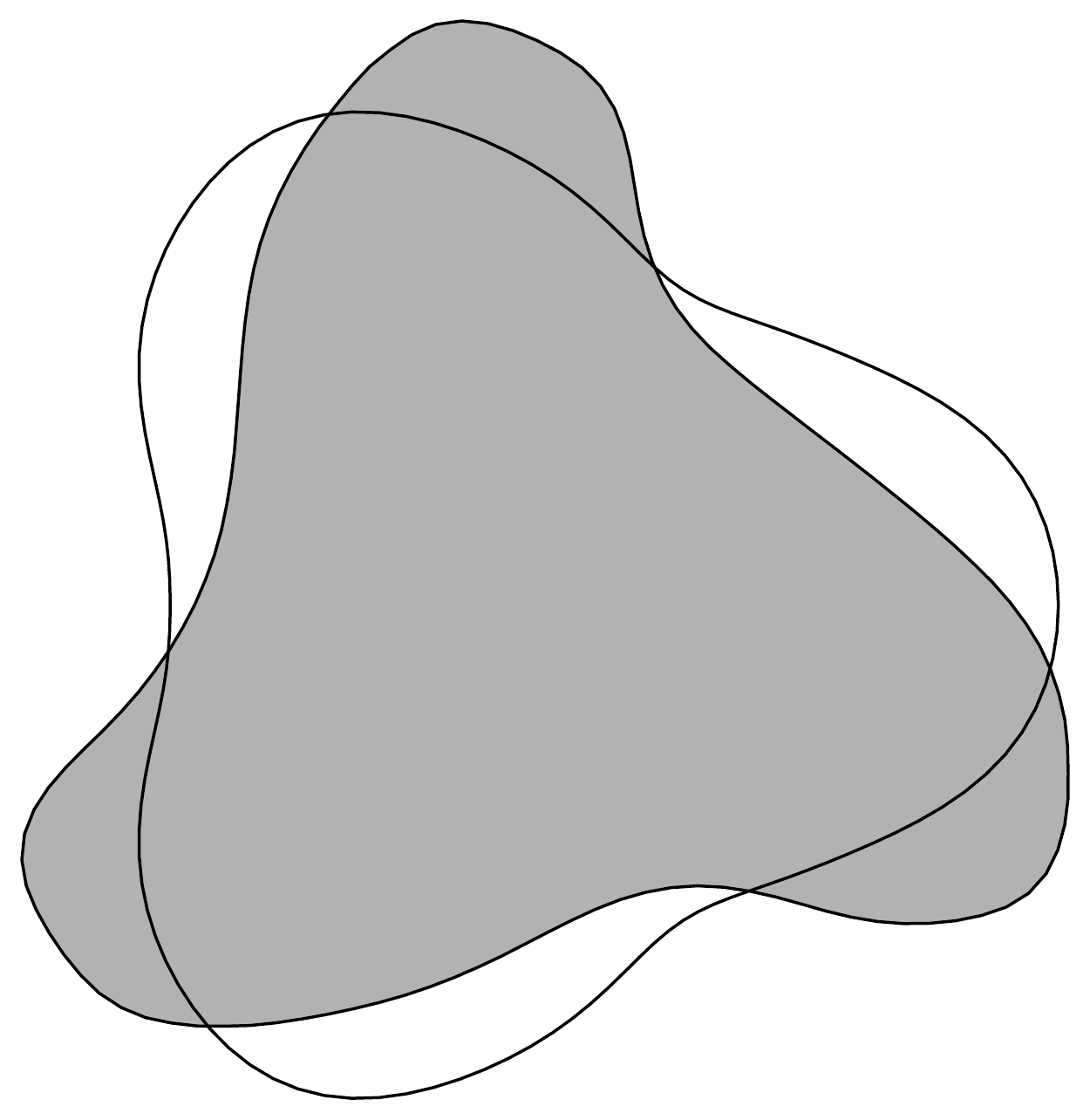}};
				\node at (1.2\figwidth,-2.0\figwidth) {(a)};
				
				\node[img] at (10cm,-0.6\figwidth) {\includegraphics[width=2\figwidth]{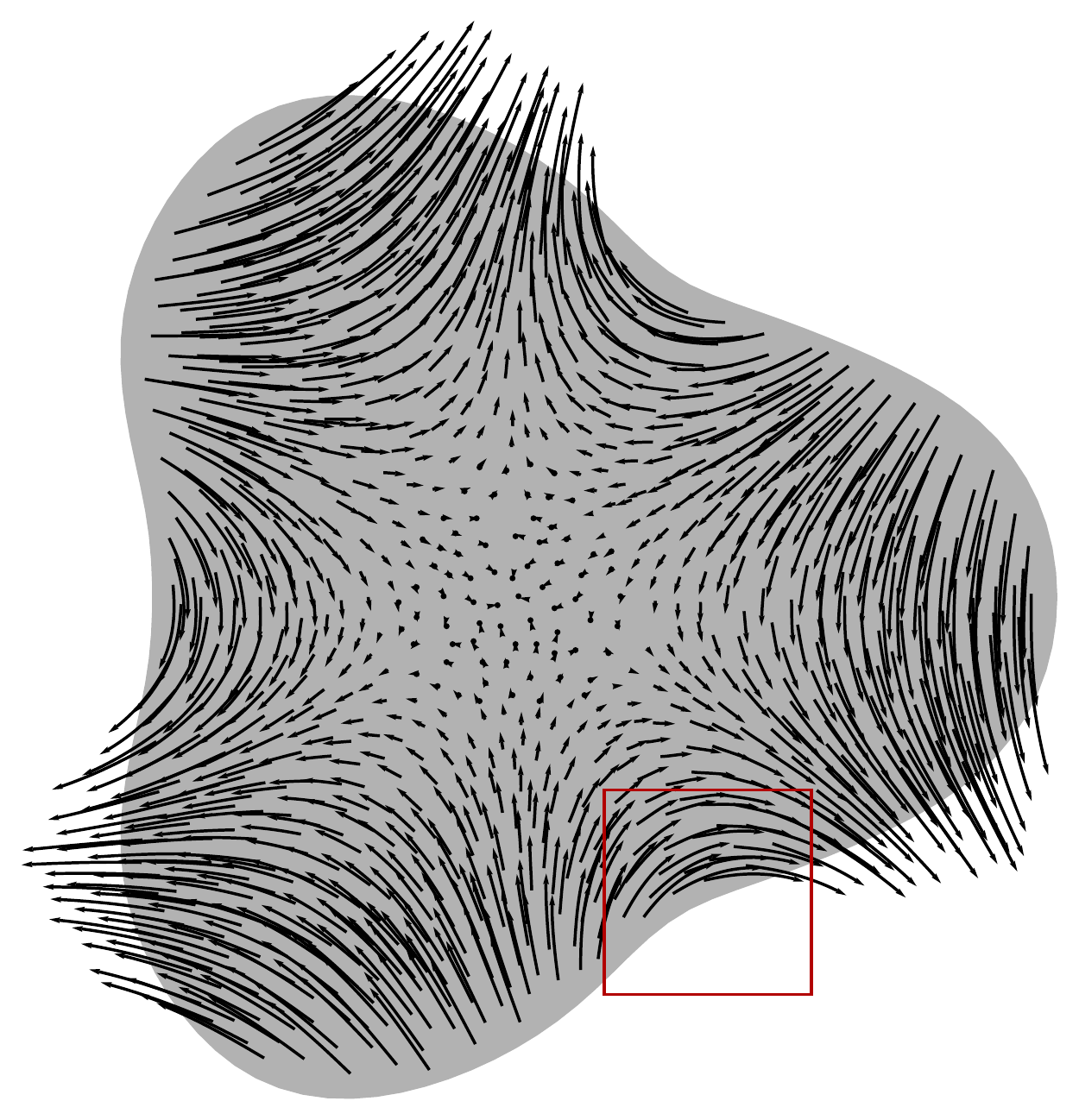}};
				\node[img] at (10cm+\figwidth-0.2\figwidthB,-0.6\figwidth+\figwidth-0.3\figwidthB) {\includegraphics[width=\figwidthB]{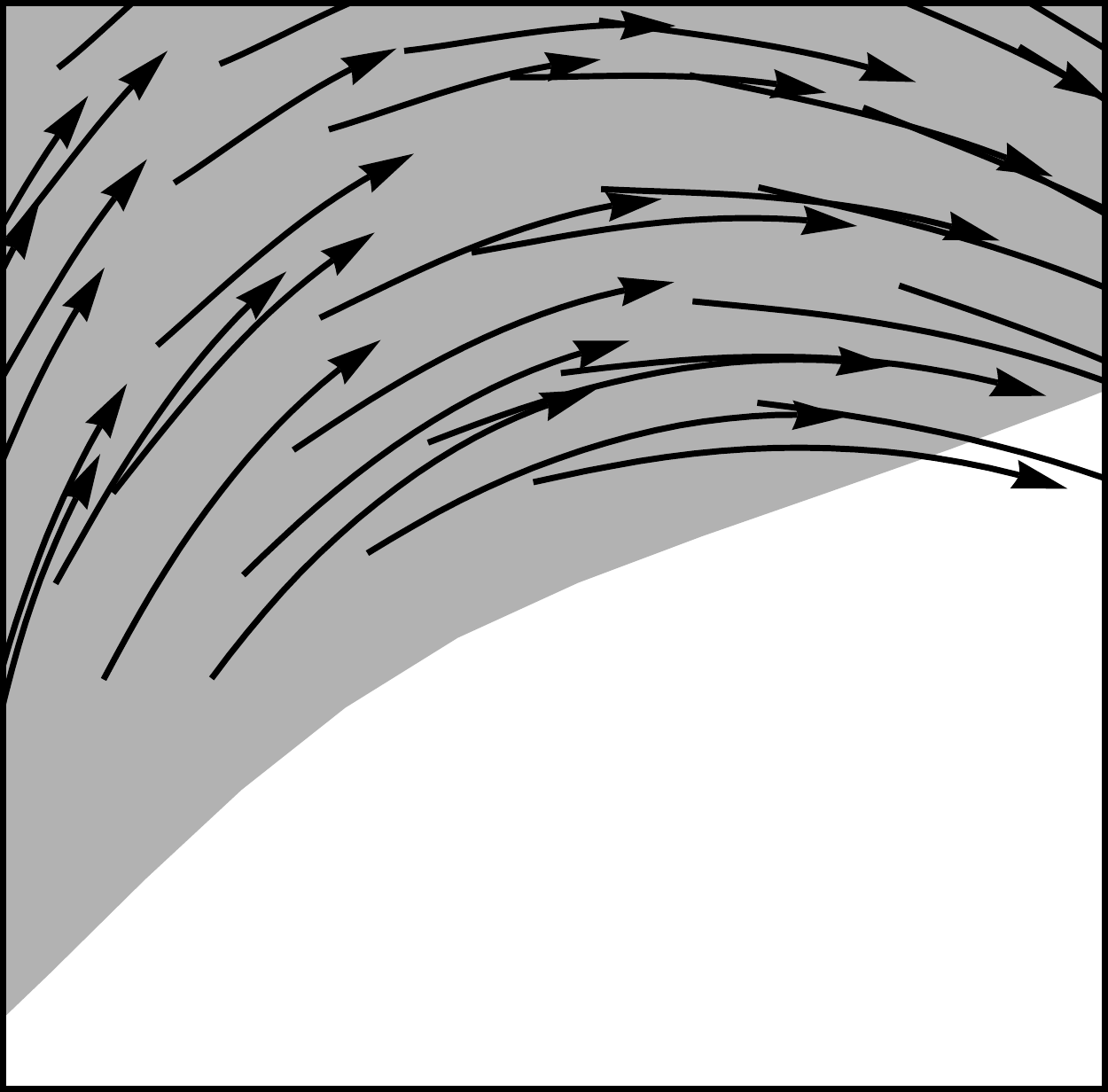}};
				\node at (10cm,-2.0\figwidth) {(b)};
			\end{tikzpicture}
			\caption{Numerically computed geodesic on $\SMeas$.
			(a) Top left to bottom right: geodesic between two shape measures (initial and final shape indicated by contours for orientation).
			Unlike the natural linear structure on measures (cf.~Fig.~\ref{fig:LinearStructures}) this gives a meaningful interpolation between the two shapes.
			(b) Trajectories of `mass particles' in the measure (with close-up). Unlike in conventional optimal transport, particles do not all travel on straight lines. This is necessary such that the intermediate measure will always represent a shape.}
			\label{fig:Geodesic}
		\end{figure}
	

%% file: contour-manifolds-v6-Section-5-Conclusion.tex
\section{Conclusion and Outlook}
	\label{sec:Conclusion}
	Shape measures, based on the pseudo-Riemannian structure of the Wasserstein space of measures, as a shape representation have been introduced to unite the complimentary strengths of representing shapes by indicator functions and parametrized contours \cite{SchmitzerSchnoerr-EMMCVPR2013}.
	
	In this article we have defined shape measures in a formally precise way and studied the mathematical relation to the manifold of $S^1$-like contours in $\R^2$.
	Bijections, acting as conversions, between the two representations and their differential structures have been introduced and the equivalence of suitable regularity classes of paths under these bijections has been established.
	Eventually it was shown that the set of shape measures is formally a manifold which is diffeomorphic to the manifold of contours.
	
	We have then equipped this manifold with the metric induced by optimal transport. This yields a new metric on the manifold of contours which has not been studied so far. We have examined the local structure in the tangent space and discussed the corresponding geodesic equation.
	
	These results prove that shape measures are in fact a representation that is equivalent in a mathematically precise way to the representation by parametrized contours. Through absence of reparametrization ambiguities and the closeness to convex variational methods via the Kantorovich formulation of optimal transport shape measures appear to be more suitable for tasks such as object segmentation.
	
	The rigorous study of the geodesic equation on the manifold of shape measures and the corresponding logarithmic map are yet an open problem on the theoretical side.
	On the practical side it will be interesting to see how shape measures can be applied to other problems, such as object tracking, or be combined with more complex appearance and shape models for image segmentation as done in \cite{SchmitzerSchnoerr-EMMCVPR2013}.